\documentclass[12pt,reqno]{amsart}
\usepackage{amssymb, amsmath, amsthm,amssymb,amsfonts}
\usepackage{tikz-cd}
\tikzcdset{
	arrow style=tikz,
    diagrams={>={Straight Barb[scale=0.8]}}
}
\usepackage[a4paper,hmargin=3cm,vmargin=3cm]{geometry}

\usepackage{extarrows}
\usepackage{mathrsfs}
\usepackage{arydshln}

\usepackage{mathtools} %such as \coloneqq

\usepackage{enumerate}
\usepackage{booktabs,float,bm}
\usepackage{xfrac}

\usepackage[square, comma, sort&compress, numbers]{natbib}
\usepackage[colorlinks = true, linkcolor = blue, citecolor = blue]{hyperref}
\numberwithin{equation}{section}
\theoremstyle{plain}
\newtheorem{theorem}{Theorem}[section]
\newtheorem{proposition}[theorem]{Proposition}
\newtheorem{corollary}[theorem]{Corollary}
\newtheorem{lemma}[theorem]{Lemma}

\theoremstyle{definition}
\newtheorem{definition}[theorem]{Definition}
\newtheorem{example}[theorem]{Example}
\newtheorem{remark}[theorem]{Remark}

\newcommand{\Z}{\ensuremath{\mathbb{Z}}}
\newcommand{\R}{\ensuremath{\mathbb{R}}}

\newcommand{\CP}[1]{\mathbb{C}P^{#1}}
\newcommand{\RP}[1]{\mathbb{R}P^{#1}}

\newcommand{\z}[1]{\ensuremath{\Z/2^{#1}}}

\newcommand{\ZZ}[1]{\ensuremath{\mathbb{Z}_{(#1)}}}

\newcommand{\PP}{\ensuremath{\mathcal{P}^1}}
\newcommand{\PD}{\ensuremath{\mathrm{PD}}}

\newcommand{\xra}{\ensuremath{\xrightarrow}}

\newcommand{\ol}{\ensuremath{\overline}}

\newcommand{\wtd}{\ensuremath{\widetilde}}

\newcommand{\lra}[1]{\ensuremath{\langle #1\rangle}}

\newcommand{\Hom}{\ensuremath{\mathrm{Hom}}}

\newcommand{\Sq}{\ensuremath{\mathrm{Sq}}}
\newcommand{\id}{\ensuremath{\mathrm{id}}}
\newcommand{\pt}{\ensuremath{\mathrm{pt}}}

\newcommand{\SO}{\ensuremath{\mathrm{SO}}}
\newcommand{\Spin}{\ensuremath{\mathrm{Spin}}}
\newcommand{\String}{\ensuremath{\mathrm{String}}}
\newcommand{\Fivebrane}{\ensuremath{\mathrm{Fivebrane}}}
\newcommand{\MSO}{\ensuremath{\mathrm{MSO}}}
\newcommand{\MSpin}{\ensuremath{\mathrm{MSpin}}}
\newcommand{\MString}{\ensuremath{\mathrm{MString}}}

\newcommand{\Lie}{\ensuremath{\mathrm{Lie}}}

\DeclareMathOperator{\im}{im}
\DeclareMathOperator{\cok}{coker}

\DeclareMathOperator{\sk}{sk}
\DeclareMathOperator{\fr}{fr}

 \newcommand{\smatwo}[2]{\ensuremath{\left(\begin{smallmatrix}
  #1\\
  #2
 \end{smallmatrix}\right)}}

\title[Cohomotopy Groups in Codimensions two and three]{Stable Cohomotopy in Codimensions Two and Three: From Algebraic Characterizations to Bordism-Theoretic Interpretations}

\author[P. Li]{Pengcheng Li}
\address{Department of Mathematics, Great Bay University,  {\rm 523000}, Dongguan, China}
\email{lipcaty@outlook.com}

\author[J. Pan]{Jianzhong Pan}
\address{
Institute of Mathematics,  Chinese Academy of Mathematics and Systems Science,  University of Chinese Academy of Sciences, {\rm 100190}, Beijing, China}
\email{pjz@amss.ac.cn}

\author[J. Wu]{Jie Wu}
\address{Hebei Normal University, Shijiazhuang, Hebei Province, \rm{050024}, China; Beijing Key Laboratory of Topological Statistics and Applications for Complex Systems, Beijing Institute of Mathematical Sciences and Applications  \rm{101408}, Beijing, China}
\email{wujie@bimsa.cn}


\begin{document}

\begin{abstract}
This paper investigates stable cohomotopy groups in codimensions two and three from complementary algebraic and geometric viewpoints. For general CW complexes, we give a complete characterization of stable cohomotopy in codimension two and a characterization in codimension three up to a $3$-primary parameter. Geometrically, we provide bordism-theoretic interpretations of these stable cohomotopy groups for oriented manifolds in codimension two and string manifolds in codimension three. As an application, we derive necessary and sufficient conditions for the existence of nowhere-vanishing sections of vector bundles, extending the foundational codimension-one results of Konstantis.
\end{abstract}

 \maketitle

 \tableofcontents

 \section{Introduction}\label{sec:intro}

For a based CW complex $X$, the $n$-th cohomotopy set $\pi^n(X)=[X,S^n]$ consists of homotopy classes of based maps from $X$ to the $n$-sphere. 
Cohomotopy theory has long played an important role in topology, geometry and physics, providing a homotopy-theoretic counterpart to generalized cohomology \cite{Borsuk36, Peterson56-2} and offering refined invariants for topological spaces and manifolds \cite{BF04,Bauer04}.

Bridging cohomotopy theory and geometric topology, the Pontryagin-Thom construction \cite[Chapter IX, Theorem 5.5]{Kosinski93} identifies the cohomotopy set $\pi^n(M)$ with the set $\Omega_k^{\fr}(M)$ of normally framed $k$-dimensional submanifolds for a closed smooth $(n+k)$-manifold $M$, yielding a group isomorphism in the stable range $n\geq k+2$. This correspondence has motivated a central program focused on establishing bordism-theoretic interpretations for the algebraic structures of cohomotopy groups. Taylor \cite{Taylor2012} resolved the Steenrod extension problem \cite{Steenrod1947} for the cohomotopy groups in codimension one and characterized $\pi^{\ast}(M^4)$ in all degrees by homotopy techniques, while Kirby, Melvin and Teichner \cite{KMT2012} provided a complete geometric description of the corresponding normally framed bordism groups of four manifolds. More recently, a systematic framework for codimension-one cohomotopy has been established through the work of Konstantis \cite{Kons2020} and Jung-Rot \cite{JR25}, who
developed bordism-theoretic interpretations for Taylor's algebraic characterization under spin, nonspin, and nonoriented settings.
Parallel to these advancements, recent progress in the homotopy decompositions of suspension spaces has enabled explicit characterizations of cohomotopy sets for manifolds of dimensions $5$ and $6$, see \cite{ACS24, LZ-5mfd, LPW-tlms}.

Despite these advancements, a systematic description of stable cohomotopy groups in higher codimensions remains significantly more elusive. %This research gap is primarily due to the rapidly increasing complexity of the group extension problems within the Postnikov tower of $S^n$, where the presence of nontrivial $k$-invariants and higher-order cohomology operations poses formidable algebraic challenges. 
In this paper, we advance the program of cohomotopy sets of manifolds by establishing complete algebraic characterizations of stable cohomotopy groups for closed smooth oriented manifolds in codimensions two and three. 
Our approach integrates two complementary perspectives. On the homotopy-theoretic side, we work with the Postnikov tower of spheres and apply the method of Larmore and Thomas \cite{LT72} to resolve the extension problems arising from nontrivial $k$-invariants. On the geometric side, we analyze the corresponding normally framed bordism groups by relating them to $G$-bordism groups along the Whitehead tower of the classifying space $B\mathrm{O}$ of orthogonal groups. Notably, we provide a pure bordism-theoretic interpretation for the stable cohomotopy of oriented manifolds in codimension two. 

A key insight of this work is that the geometric and homotopical perspectives on the stable cohomotopy groups of manifolds are intrinsically intertwined. The geometric filtration arising from the Whitehead tower of the classifying space $B\mathrm{O}$ reflects the homotopical filtration derived from the stable Postnikov tower of spheres. To be specific, the successive geometric obstructions to lifting tangential or normal structures, detected by characteristic classes, reflect the action of the cohomology operations (including primary and higher order operations) that are induced by the stable Postnikov system of $S^n$. This is also indicated by the homotopy equivalence $\mathbf{S}\to\varinjlim_m\mathbf{MO}\lra{m}$ of spectra \cite[Proposition 2.1.1]{Hovey97}.

%As a key application, we utilize these results to establish complete necessary and sufficient conditions for the existence of nowhere-vanishing sections of vector bundles, effectively refining the classical obstruction theory via spin and string orientations. This work recovers and generalizes the codimension-one results of Konstantis \cite{Kons2020} by extending the geometric framework to higher codimensions.

%Let $\varTheta$  the secondary cohomology operation based on the Adem relation $\Sq^2\Sq^2_{\Z}=0$. 
To state the main results, we introduce the following conventions and notation.  For integers $m,n\geq 1$, let $\Sq^m_\Z=\Sq^m\circ \rho_2$
be the composition of the usual Steenrod square operation $\Sq^m$ and the mod $2$ reduction $\rho_2$. For an abelian group $G$, we denote by ${}_2G$ the elementary $2$-torsion subgroup of $G$, which consists of elements of order dividing $2$.  
For any homomorphism $f\colon G(m)\to H(n)$ between abelian groups indexed by $m,n$, we denote by $\ker(f:m)$ the kernel of $f$ to indicate the degree $m$ of the domain, and denote by 
\[QH^{n}(X;f)=\cok(f)=H^{n}(X;\z{})/f(G(m))\] 
to indicate the codomain degree $m$. Examples of the homomorphisms $f$ involved here are cohomology operations and induced homomorphisms of $k$-invariants; for instance, we use notations 
\[\ker(\Sq^m_\Z:n)\subseteq H^{n}(X;\Z),\quad QH^{m+2}(X;\Sq^m)=\tfrac{H^{m+2}(X;\z{})}{\Sq^m_\Z(H^m(X;\Z))}.\]
Let $\delta\colon H^{m}(X;\z{})\to H^{m+1}(X;\Z)$ be the Bockstein homomorphism associated to the short exact sequence $0\to \Z\xra{\times 2}\Z\xra{\rho_2}\z{}\to 0.$ 
By abusing notation, we also write $\delta$ for the corresponding Bockstein homomorphism in homology; this should cause no ambiguity, as the intended meaning will always be clear from the context. 

Recall that for each $n\geq 1$, the canonical inclusion $\iota_n\colon S^n\to K_n(\Z)$ induces the $n$-th cohomotopy Hurewicz map 
\[h^n\colon \pi^n(M)\to H^n(M;\Z),\]
which is a homomorphism of groups whenever $\dim(M)\leq 2n-2$. 
Our first main result provides a complete characterization of the cohomotopy groups in codimension two for smooth oriented manifolds.

\begin{theorem}[see Theorem \ref{thm:cohtp=-2:mfds}]\label{mainthm:codim2}
Let $M$ be a connected closed smooth oriented $(n+2)$-manifold with $n\geq 4$. There is a short exact sequence of groups 
\begin{equation}\label{SES:cohtp=-2:mfds}
	0\to \z{1-\varepsilon}\oplus QH^{n+1}(M;\Sq^2_\Z)\to \pi^n(M)\xra{h^n}\ker(\Sq^2_\Z:n)\to 0,
\end{equation}
where $\varepsilon=0$ if $M$ is spin, and $\varepsilon=1$ otherwise. Moreover, the above short exact sequence splits if and only if 
 \[\Sq^2\big(\delta^{-1}({}_2\ker(\Sq^2_\Z:n))\big)\subseteq\Sq^2_\Z\big(H^{n-1}(X;\Z)\big).\]
 
\end{theorem}

In fact, Theorem \ref{mainthm:codim2} is a special case of the characterization of the codimension-$2$ cohomotopy group for general CW complexes such that $H_{n+2}(X;\z{})\cong\z{}$, see Theorem \ref{thm:cohtp=-2:CW}.

Let $w_i\in H^i(B\mathrm{O};\z{})$ be the universal Stiefel-Whitney class, where $B\mathrm{O}$ is the classifying space of the orthogonal group $\mathrm{O}=\mathrm{colim}_{n}\mathrm{O}_n$. 
Note that by Wu's formula and Lemma \ref{lem:mfd:Sq2}, for a closed oriented $(n+2)$-manifold $M$, the action of $\Sq^2$ on $H^{n+1}(M;\z{})$ or $H^{n}(M;\z{})$ is given by the cup product with $w_2(M)$.
It follows that when $M$ is spin,  the short exact sequence (\ref{SES:cohtp=-2:mfds}) splits and $QH^{n+1}(M;\Sq^2_\Z)=H^{n+1}(M;\z{})$ by Lemma \ref{lem:spin:Theta=0}.

Under the Poincar\'e duality, the actions of the Steenrod squares $\Sq^2_\Z$ on $H^n(M;\Z)$ and $\Sq^2$ on $H^{n-1}(M;\Z)$ are given as follows. 
\begin{enumerate}
	\item The Steenrod square $\Sq^2_\Z\colon H^{n}(M;\Z)\to H^{n+2}(M;\z{})$ is Poincar\'e dual to the pairing
	\[\lra{w_2(M),-}_\Z:H_2(M;\Z)\xra{\rho_2}H_2(M;\z{})\xra{\lra{w_2(M),-}} \z{}.\]
	\item The Steenrod square $\Sq^2\colon H^{n-1}(M;\Z)\to H^{n+1}(M;\z{})$ is Poincar\'e dual to the cap product
	\[w_2(M)\smallfrown_\Z -\colon H_3(M;\Z)\xra{\rho_2}H_3(M;\z{})\xra{w_2(M)\smallfrown -} H_1(M;\z{}).\]
\end{enumerate}
Then applying the Pontryagin-Thom construction for cohomotopy, we get the following geometric interpretation of Theorem \ref{mainthm:codim2}. %characterization of the normally framed bordism group $\Omega^{\mathrm{fr}}_2(M)$, which serves as a 

\begin{theorem}\label{mainthm:frbdsm-dim2}
	Let $M$ be a connected closed smooth oriented $(n+2)$-manifold with $n\geq 4$. 
	Then there is a short exact sequence of groups
	\begin{equation}\label{SES:mainthm2}
		0\to \Omega_2^{\fr}[1-\varepsilon]\oplus \tfrac{H_1(M;\z{})}{w_2(M)\smallfrown_\Z H_3(M;\Z)}\to \Omega_2^{\fr}(M)\xra{h_2} \ker\big(\lra{w_2(M),-}_\Z\big)\to 0,
	\end{equation}
	which splits if and only if
   \[w_2(M)\smallfrown \delta^{-1}\big(\ker(\lra{w_2(M),-}_\Z)\big)\subseteq w_2(M)\smallfrown_\Z H_3(M;\Z).\]
	Here $h_2$ is the forgetful homomorphism sending a normally framed bordism class $[N,f,\varphi]$ to the homology class $f_\ast([N])\in H_2(M;\Z)$ with $[N]$ the fundamental class; $\Omega_2^{\fr}[1-\varepsilon]=\Omega_2^{\fr}$ if $w_2(M)=0$, and $\Omega_2^{\fr}[1-\varepsilon]=0$ otherwise.
	
\end{theorem}

We also provide a pure geometric proof of Theorem \ref{mainthm:frbdsm-dim2}, see Section \ref{sec:cohtp-2-geom}. %Let $\Omega_2^{\fr,\Spin}(M)$ be the set of bordism classes of embedded closed $2$-manifolds $f\colon N\to M$ with a spin structure on the normal bundle $\nu_f$, which coincides with the spin bordism group $\Omega_2^{\Spin}(M)$ when $M$ is spin by Remark \ref{rmk:bdsm-frbdsm}. As a byproduct of our geometric proof of Theorem \ref{mainthm:frbdsm-dim2}, we also obtain the characterization of the framed spin bordism group $\Omega_2^{\fr,\Spin}(M)$, see Theorem \ref{thm:spinbdsm-dim2}. Moreover, comparing these two descriptions clarifies the relationship between the normally framed bordism group and the framed spin bordism group, yielding Corollary \ref{cor:nonspin-fr-to-spin}: there is a split short exact sequence of groups 
%\begin{equation}\label{eq:codim2:fr-spin}
%	0\to\Omega_2^{\fr}(M)\xra{\phi_2}\Omega_2^{\fr,\Spin}(M)\to \Omega_2^{\Spin}[\varepsilon]\to 0,
%\end{equation}
%where $\phi_2$ is the canonical homomorphism, $\varepsilon=0$ if $M$ is spin, otherwise $\varepsilon=1$.

In codimension three, the situation involves three higher-order cohomology operations, and becomes considerably more complicated. Let $\Theta$ and $\Psi$ be the secondary cohomology operations based on the Adem relations \eqref{eq:Adem-rel-Theta} and \eqref{eq:Adem-rel-Psi} respectively, and let $\mathbb{T}$ be the tertiary cohomology operation based on \eqref{eq:Adem-rel-T}. Recall that these cohomology operations detect $\eta^2\in\pi_2^S$, $2\nu\in\pi_3^S$ and $\eta^3\in\pi_3^S$, respectively.
For each of the following cohomology operations $\mathcal{O}$, we 
denote 
\begin{equation}\label{eq:epsilon-operations}
	\varepsilon(\mathcal{O})=\begin{cases}
0,& \text{ if $\mathcal{O}$  is trivial};\\
1,& \text{otherwise}.
\end{cases}
\end{equation}
Here $\mathcal{O}$ ranges over the following operations, where $M$ is a connected closed smooth oriented $(n+3)$-manifold:
\begin{align*}
	\Sq^4_\Z&\colon \ker(\Sq^2_\Z:n-1)\to H^{n+3}(M;\z{}),\\
	\varPhi&\colon \ker(\varTheta:n-1)\cap \ker(\Sq^4_\Z:n-1)\to H^{n+3}(M;\z{}),\\
	\mathbb{T}& \colon \ker(\alpha^3_\ast:n-1)\cap \ker(\alpha^4_\ast:n-1)\to H^{n+3}(M;\z{}).
\end{align*}
Let $\ker(\alpha^3_\ast:n)$ be given by the short exact sequence of groups
\begin{align*}
	0\to \tfrac{\ker(\Sq^2:n+1)}{\Sq^2_\Z(H^{n-1}(M;\Z))} \!\oplus\! \z{1-\varepsilon(\Sq^4_\Z)}\to \ker(\alpha^3_\ast:n)\xra{(q_1)_\ast} \ker(\Sq^2_\Z:n)\to 0,
\end{align*}
whose splitting criterion is given in Proposition \ref{prop:Q1S}. Denote 
\[\ker(\ol{\Sq^2}_\ast:n)=\tfrac{\ker(\alpha^3_\ast:n)}{\z{1-\varepsilon(\Sq^4_\Z)}}.\]

Our third main theorem characterizes the cohomotopy group in codimension three up to a $3$-primary parameter $\epsilon \in \{0,1\}$.

\begin{theorem}[see Theorem \ref{thm:cohtp=-3:mfds}]\label{mainthm:cohtp=-3:mfds}
Let $M$ be a connected closed smooth oriented manifold of dimension $(n+3)$ with $n\geq 5$. 
	\begin{enumerate}
		\itemsep=5pt
		\item In each of the following three cases, there is a short exact sequence of groups 
		\begin{equation*}
			0\to \Z/2^k\oplus\Z/3^{1-\epsilon}\to \pi^n(M)\to \ker(\ol{\Sq^2}_\ast:n)\!\oplus\! \tfrac{QH^{n+2}(M;\Sq^2)}{\varTheta(\ker(\Sq^2_\Z\cap\Sq^4_\Z:n-1))}\to 0,
		\end{equation*}
	where $\epsilon\in\{0,1\}$ and $k$ is given by the following table:
\begin{center}
\begin{tabular}{|l|l|}
\hline
\textbf{conditions on $M$} & \textbf{values of $k$} \\
\hline
$w_2(M) = 0$ & $3 - \varepsilon(\Sq^4_\Z) - \varepsilon(\varPhi) - \varepsilon(\mathbb{T})$ \\
\hline
$w_2(M) \neq 0, \ w_3(M) = 0$ & $2 - \varepsilon(\Sq^4_\Z) - \varepsilon(\varPhi)$ \\
\hline
\begin{tabular}{@{}l@{}}
$w_2(M) \neq 0, \ w_3(M) \neq 0,$ \\
$\ker(w_2\!\smallsmile\!-: n) \subseteq \ker(w_3\!\smallsmile\!-: n)$
\end{tabular} 
& $2 - \varepsilon(\Sq^4_\Z)$ \\
\hline
\end{tabular}
\end{center}

	\item  If $w_2(M)\neq 0$ and $\ker(w_2(M)\!\smallsmile\!-:n)\nsubseteq \ker(w_3(M)\!\smallsmile\!-:n)$, then there is a short exact sequence of groups
	\begin{equation*}
		0\to \Z/3^{1-\epsilon}\to \pi^n(M)\to\ker(\alpha^3_\ast:n)\!\oplus\! \tfrac{QH^{n+2}(M;\Sq^2)}{\varTheta(\ker(\Sq^2_\Z\cap\Sq^4_\Z:n-1))}\to 0,
	\end{equation*}
where $\epsilon\in\{0,1\}$ and $\ker(\alpha^3_\ast:n)$ contains the direct summand $\Z/2^{1-\varepsilon(\Sq^4_\Z)}$ if and only if
\[\Sq^4\left(\delta^{-1}({}_2\ker(\Sq^2_\Z:n))\right)\subseteq \Sq^4_\Z\left(\ker(\Sq^2_\Z:n-1)\right).\]
	\end{enumerate}
If, in addition, $p_1(M)\equiv 0\pmod 3$, then $\epsilon=0$ and all the above short exact sequences are split. 
\end{theorem}

% When $M$ is a closed string $(n+3)$-manifold with $n\geq 5$, we have the following characterizations of $\pi^n(M)$.

\begin{corollary}[see Theorem \ref{thm:cohtp=-3:string} and Example \ref{ex:string-bdsm}]\label{mainthm:cohtp=-3:string}
Let $M$ be a connected closed string manifold of dimension $n+3$ with $n\geq 5$. There is a split short exact sequence of groups 
	\begin{equation}\label{SES:cohtp=-3:string}
			0\to \Z/24\to \pi^n(M)\to \ker(\ol{\Sq^2}_\ast:n)\oplus QH^{n+2}(M;\varTheta)\to 0
	\end{equation}
with $\ker(\ol{\Sq^2}_\ast:n)$ characterized by the short exact sequence of groups 
\begin{equation*}
	0\to QH^{n+1}(M;\Sq^2_\Z)\to \ker(\ol{\Sq^2}:n)\to H^n(M;\Z)\to 0,
\end{equation*}
which splits if and only if 
\[\Sq^2_\Z(H^{n-1}(M;\Z))=\Sq^2(H^{n-1}(M;\Z/2)).\]
Moreover, applying the Pontryagin-Thom construction to (\ref{SES:cohtp=-3:string}) yields the split short exact sequence of groups
\begin{equation}\label{SES:cor-string}
	0\to \Omega_3^{\fr}\xra{e_\ast}\Omega_3^{\fr}(M)\xra{\phi_3}\Omega_3^{\Spin}(M)\to 0,
\end{equation}
where $\phi_3$ is the canonical homomorphism sending a normal framed bordism $[N,f,\varphi]$ to $[N,f,\sigma_{\varphi}]$ with $\sigma_\varphi$ the normal spin structure induced by the normal framing $\varphi$, and $e_\ast$ is the homomorphism induced by the base-point inclusion map $e\colon \{x_0\}\to M$.
\end{corollary}

Our last main theorem serves as an application of the above short exact sequences (\ref{SES:mainthm2}) with $w_2(M)=0$, and (\ref{SES:cor-string}), together with Konstantis's result \cite[Theorem 3.13]{Kons2020} for $k=1$; the codimension-one case was originally proved by Konstantis \cite{Kons2020}.

Define the groups $G(k)$ and $H(k)$ by the table 
	\begin{table}[H]
\begin{tabular}{c|cc}
\hline
$k$ & $G(k)$ & $H(k)$ \\
\hline
$1,2$ & $\Spin$   & $\SO$ \\
\hline
$3$ & $\String$ & $\Spin$ \\
\hline
\end{tabular}
\caption{The groups $G_k$ and $H_k$}\label{table:GH-k}
\end{table}

\begin{theorem}
	\label{mainthm:appl}
Let $k=1,2,3$, and let $(G,H)=(G(k),H(k))$ be the pair of groups in Table \ref{table:GH-k}. Let $M$ be a connected closed smooth $G$-manifold of dimension $n+k$ with $n\geq k+2$ and let $E\to M$ be a $G$-vector bundle of rank $n$. Then the following are equivalent:
\begin{enumerate}
	\item there exists a nowhere-vanishing section of $E$;
	\item the $G$-Euler class $e_{G}(E)=0$;
	\item the $G$-divisor $\kappa_{G}(E)=0$,  the $H$-Euler class $e_{H}(E)=0$, and the $G$-defect class $\delta_G(E)=0$.
\end{enumerate}

\end{theorem}
Here $e_{\SO}(E)=e(E)$ is the ordinary Euler class of an oriented vector bundle, while $e_{\Spin}(E)$ and $e_{\String}(E)$ are the spin Euler class and the string Euler classes induced by the Atiyah-Bott-Shapiro orientation \cite{ABS64} for spin vector bundles and the Ando-Hopkins-Rezk orientation \cite{AHR10} for string vector bundles, respectively. The $G$-divisors $\kappa_G(E)\in \Omega_k^G$ and the $G$-defect class $\delta_G(E)\in H_1(M;\z{})$ of $G$-vector bundles $E$ over $M$ refer to Definition \ref{def:G-divisor}. For more details of these concepts, see Section \ref{sec:appl-VB}.

The paper is organized as follows. Section \ref{sec:prelim} reviews the stable Postnikov tower of spheres and the relevant higher-order cohomology operations. Sections \ref{sec:cohtp:codim2} and \ref{sec:cohtp:codim3} apply these tools to investigate the algebraic characterizations of the cohomotopy groups in codimensions $2$ and $3$, respectively; in particular, Theorems \ref{mainthm:codim2} and \ref{mainthm:cohtp=-3:mfds} are proved. Section \ref{sec:general-Gbdsm} characterizes the normally framed bordism groups $\Omega_k^{\fr}(M)$ of $G$-manifolds for small $k$ via the Whitehead tower of $B\mathrm{O}$, yielding a unified description of the $G$-bordism groups $\Omega_k^G(M)$ (and hence $\Omega_k^{\fr}(M)$) for $G$-manifolds, see Theorem \ref{thm:SES-G-H}. Section \ref{sec:cohtp-2-geom} gives a pure geometric proof of Theorem \ref{mainthm:frbdsm-dim2}. Finally, in Section \ref{sec:appl-VB}, we complete the proof of Theorem \ref{mainthm:appl}.

\section{Preliminaries}\label{sec:prelim}

Throughout this paper, all spaces are based CW complexes and all maps and homotopies are base-point preserving. For spaces $X,Y$, we denote by $[X,Y]$ the set of homotopy classes of based maps from $X$ to $Y$. 
Since short exact sequences of abelian groups occur frequently, we will use ``SES'' as shorthand for ``short exact sequence'' in the remainder of the paper.

\subsection{Postnikov towers of spheres and cohomology operations}\label{sec:prelim-Postnikov}

Let $K_n(G)=K(G,n)$ be the Eilenberg-MacLane space with the unique nontrivial homotopy group $G$ in degree $n$; denote $K_n=K_n(\z{})$ for simplicity.   For $G=\Z, \Z/p^r$ or $\ZZ{p}$ and $1\leq s\leq r$, we denote by 
\[\rho_{p^s}\colon K_n(G)\to K_n(\Z/p^s),\quad \mu_{p^{r-s}}\colon K_n(\Z/p^s)\to K_n(\Z/p^r)\]  
the mod $p^s$ reduction and the map induced by the multiplication by $p^{r-s}$; the domain $K_n(G)$ can be inferred from the context. For a primary cohomology operation $\Sq^i$ or $\PP_p$ with $p$ an odd prime, we denote \[\Sq^i_G=\Sq^i\circ \rho_2,\quad \PP_G=\PP_p\circ \rho_p.\]

Let $P_rS^n=P_{n+r}(S^n)$ denote the $(n+r)$-th Postnikov section of $S^n$ defined by the principal homotopy fibration sequence 
	\[K_{n+r}(\pi_{n+r}(S^n))\xra{j_r} P_rS^n\xra{p_r}P_{r-1}S^n\xra{k_{r+1}} K_{n+r+1}(\pi_{n+r}(S^n)),\]
where $k_{r+1}$ is the $k$-invariant. Note that the canonical map $\imath_r\colon S^n\to P_rS^n$ is $(n+r+1)$-connected, and that $\pi_{n+r}(S^n)\cong \pi_r^S$ is the $r$-th stable homotopy group of $S^0$ for  $n\geq r+2$.  

\begin{lemma}\label{lem:tower:Sn}
	Up to $n+3$ stages, the stable Postnikov tower of $S^n$ takes the form:
\begin{equation*}
	\begin{tikzcd}
		K_{n+3}(\Z/24)\ar[r,"j_{3}"]&P_3S^n\ar[d,"p_{3}"swap]&\\
		K_{n+2}\ar[r,"j_{2}"]&P_2S^n\ar[r,"{\smatwo{\wtd{\Sq^4_\Z}\circ p^2_1}{\PP_\Z\circ p^2_1}}"]\ar[d,"p_{2}"swap]&K_{n+4}(\Z/8)\times K_{n+4}(\Z/3)\\
		K_{n+1}\ar[r,"j_{1}"]&P_{1}S^n\ar[r,"\ol{\Sq^2}"]\ar[d,"p_{1}"swap]&K_{n+3}\\
		&K_n(\Z)\ar[r,"\Sq^2_\Z"]&K_{n+2}
	\end{tikzcd}
\end{equation*}
Here $p^2_1=p_1\circ p_2$, $\ol{\Sq^2}$ and $\wtd{\Sq^4_\Z}\colon K_n(\Z)\to K_{n+4}(\Z/8)$ respectively satisfy the formulas
\begin{equation}\label{eq:Sq2-Sq4}
\ol{\Sq^2}\circ j_1\simeq\Sq^2,\quad \rho_2\circ \wtd{\Sq^4_\Z}\simeq \Sq^4_{\Z}.
\end{equation}

\end{lemma}

\begin{proof}
	See  \cite[Figure 1, page 122]{MT68}.
\end{proof}

\begin{lemma}\label{lem:tower-Sn-v2}
	Up to $n+3$ stages, the stable $2$-primary Postnikov tower of $S^{n}$ takes the form:
\[\begin{tikzcd}
K_{n+3}\ar[r,"j_3"]&Q_3S^n\ar[d,"q_3"]&\\
K_{n+2}\times K_{n+3}\ar[r,"j_2"]&Q_2S^n\ar[r,"\beta^4"]\ar[d,"q_2"]&K_{n+4}\\
K_{n+1}\times K_{n+3}\ar[r,"j_1"]&Q_1S^n\ar[d,"q_1"]\ar[r,"
{\smatwo{\alpha^3}{\alpha^4}}"]&K_{n+3}\times K_{n+4}\\
&K_n(\Z)\ar[r,"{\smatwo{\Sq^2_\Z}{\Sq^4_\Z}}"]&K_{n+2}\times K_{n+4}
\end{tikzcd},\]
where $\alpha^3,\alpha^4,\beta^4$ are $k$-invariants satisfying the  formulas
\begin{equation}\label{eq:alpha-beta}
	\alpha^3\circ j_1\simeq (\Sq^2,0),~~ \alpha^4\circ j_1\simeq (\Sq^2\Sq^1,\Sq^1),~~ 
	\beta^4\circ j_2\simeq (\Sq^2,\Sq^1).
\end{equation}
\begin{proof}
	See  \cite[Example I]{LT72}.
\end{proof}
\end{lemma}

The following three remarks clarify the relationship between the two Postnikov towers of $S^n$ in Lemmas \ref{lem:tower:Sn} and \ref{lem:tower-Sn-v2}, and the relationship between the $k$-invariants and the associated cohomology operations.

\begin{remark}\label{rmk:towers-rels}
Comparing the $k$-invariants in Lemmas \ref{lem:tower:Sn} and \ref{lem:tower-Sn-v2}, there are homotopy commutative diagrams of homotopy fibration sequences with $K_{r,s}=K_r\times K_s$:
	\[\begin{tikzcd}[column sep=21pt]
	K_{n+3}\ar[r]\ar[d]&\ast\ar[d]\ar[r]&K_{n+4}\ar[d]\\
	Q_1S^n\ar[r,"q_1"]\ar[d,"g_1"]&K_n(\Z)\ar[r,"{\smatwo{\Sq^2_\Z}{\Sq^4_\Z}}"]\ar[d,equal]&K_{n+2,n+4}\ar[d,"\mathrm{pr}_1"]\\
	P_1S^n\ar[r,"p_1"]&K_n(\Z)\ar[r,"\Sq^2_\Z"]&K_{n+2}
	\end{tikzcd}\begin{tikzcd}[column sep=21pt]
		K_{n+3}(\Z/4)\ar[r,"\rho_2"]\ar[d]&K_{n+3}\ar[r,"\Sq^1"]\ar[d]&K_{n+4}\ar[d]\\
		Q_2S^n\ar[r,"q_2"]\ar[d,"g_2"]&Q_1S^n\ar[r,"{\smatwo{\alpha^3}{\alpha^4}}"]\ar[d,"g_1"]&K_{n+3,n+4}\ar[d,"\mathrm{pr}_1"]\\
		P_2S^n\ar[r,"p_2"]&P_1S^n\ar[r,"\ol{\Sq^2}"]&K_{n+3}
	\end{tikzcd}\]
\[\begin{tikzcd}
Q_3S^n\ar[r,"q_3"]\ar[d,"g_3"]&Q_2S^n\ar[r,"\beta^4"]\ar[d,"g_2"]&K_{n+4}\ar[d,"\mu_4"]\\
P_3'S^n\ar[d]\ar[r,"p_3"]&P_2S^n\ar[d,"\rho_4\wtd{\Sq^4_\Z}\circ p^2_1"]\ar[r,"\wtd{\Sq^4_\Z}\circ p^2_1"]&K_{n+4}(\Z/8)\ar[d,"\rho_4"]\\
\ast\ar[r]&K_{n+4}(\Z/4)\ar[r,equal]&K_{n+4}(\Z/4)
\end{tikzcd}.\]
Here $g_1,g_2,g_3$ are excision maps, $P_3'S^n$ is the homotopy fibre of $\wtd{\Sq^4_\Z}\circ p^2_1$.
In particular, the left-most homotopy fibration in the third diagram shows that the natural map $g_3\colon Q_3S^n\to P_3'S^n$ is a homotopy equivalence. %Note also that there are group isomorphisms 
%\[H^{n+4}(P_2S^n;\Z/8)\cong\Z/8\lra{\wtd{\Sq^4_\Z}\circ p^2_1},\quad H^{n+4}(P_2S^n;\Z/4)\cong \Z/4\lra{\rho_4\wtd{\Sq^4_\Z}\circ p^2_1}.\]
\end{remark}

\begin{remark}\label{rmk:PQT3}
	The stable $3$-primary Postnikov tower of $S^n$ up to $(n+3)$ stages defines the space $R_3S^n$ by the homotopy fibration sequence
	\[R_3S^n\xra{r_3} K_n(\Z)\xra{\PP_\Z}K_{n+4}(\Z/3).\] 
	Let $K(2)=K_{n+2,n+3}\times K_{n+4}(\Z/8)$. Then we have the homotopy commutative diagram with homotopy fibration rows:
\[\begin{tikzcd}
Q_3S^n\ar[r,"q^3_1"]&K_n(\Z)\ar[r,"\kappa"]\ar[d,equal]&K(2)\\
P_3S^n\ar[r,"p_1^3"]\ar[u]\ar[d]&K_n(\Z)\ar[r,"\smatwo{\kappa}{\PP_\Z}"]\ar[d,equal]&K(2)\times K_{n+4}(\Z/3)\ar[u,"\mathrm{pr}_1"]\ar[d,"\mathrm{pr}_2"]\\
R_3S^n\ar[r,"r_3"]&K_n(\Z)\ar[r,"\PP_\Z"]&K_{n+4}(\Z/3)
\end{tikzcd}.\]	
Here $\kappa=(\Sq^2,\Sq^2\Sq^1,\wtd{\Sq^4_\Z})^T$ with $\wtd{\Sq^4_\Z}$ a lift of $\Sq^4_\Z$ satisfying $\rho_2\circ \wtd{\Sq^4_\Z}\simeq \Sq^4_\Z$.
It follows that there is a homotopy pullback diagram
\begin{equation*}
	\begin{tikzcd}P_3S^n\ar[r]\ar[d]&Q_3S^n\ar[d,"q^3_1"]\ar[r,"\PP_\Z\circ q_1^3"]&K_{n+4}(\Z/3)\ar[d,equal]\\
R_3S^n\ar[r,"r_3"]&K_n(\Z)\ar[r,"\PP_\Z"]&K_{n+4}(\Z/3)
\end{tikzcd}.
\end{equation*}
In particular, there is a principal homotopy fibration sequence
\begin{equation}\label{fib:P3-Q3}
	K_{n+3}(\Z/3)\xra{j_3} P_3S^n\xra{} Q_3S^n\xra{\PP_\Z\circ q^3_1} K_{n+4}(\Z/3).
\end{equation}

\end{remark}

\begin{remark}\label{rmk:ChlgyOps}
	There are natural stable higher order cohomology operations associated to the stable Postnikov towers in Lemma \ref{lem:tower:Sn} and Lemma \ref{lem:tower-Sn-v2}. For example, let $X$ a based CW complex. Let 
	\[u\in \ker(\Sq^2_\Z:n)\cap\ker(\Sq^4_\Z:n)\subseteq H^n(X;\Z)\] and let $\tilde{u}_1\in [X,Q_1S^n]$ be a lift of $u$, then in Lemma \ref{lem:tower-Sn-v2}, $\alpha^3$ and $\alpha^4$ respectively represent the secondary cohomology operations $\varTheta$ and $\varPsi$ defined by 
	\[\varTheta(u)=\alpha^3\circ \tilde{u}_1,\quad \varPsi(u)=\alpha^4\circ \tilde{u}_1;\]
	The lifts $\tilde{u}_1$ of $u$ may not be unique, but the values $\varTheta(u)$ and $\varPsi(u)$ are well-defined up to the indeterminacies
	\begin{align*}
		\mathrm{ind}(\varTheta)&=\im(\alpha^3(j_1)_\ast)=\Sq^2(H^{n+1}(X;\Z/2)),\\
        \mathrm{ind}(\varPsi)&=\im(\alpha^4(j_1)_\ast)=\Sq^2\Sq^1(H^{n+1}(X;\Z/2))+\Sq^1(H^{n+3}(X;\Z/2))
	\end{align*}
	In other words, the secondary cohomology operations $\varTheta$ and $\varPsi$ are respectively based on the relations
	\begin{equation}\label{eq:Adem-rel-Psi:modSq1}
			\Sq^2\Sq^2_\Z=0,\quad \Sq^1\Sq^4_\Z+\Sq^2\Sq^1\Sq^2_\Z=0.
	\end{equation}
  Since $\Sq^1_\Z=\Sq^1\circ \rho_2=0$, $\varTheta$ takes the same values as the secondary operation $\Theta$ based on the Adem relation
  \begin{equation}\label{eq:Adem-rel-Theta}
	\Sq^3\Sq^1+\Sq^2\Sq^2=0;
  \end{equation}
 and that they both detect $\eta^2\in\pi_2^S$, see  \cite[Page 96]{Harperbook}. Similarly, $\varPsi$ takes the same values as the secondary operation $\Psi$ based on the Adem relation
\begin{equation}\label{eq:Adem-rel-Psi}
	\Sq^4\Sq^1+\Sq^2\Sq^3+\Sq^1\Sq^4=0;
\end{equation}
 and they both detect $2\nu_n\in\pi_3^S$ for $n\geq 5$, see  \cite[Lemma 4.4.4]{Adams60}.

  The $k$-invariant $\beta^4$ in Lemma \ref{lem:tower-Sn-v2} represents the tertiary cohomology operation $\mathbb{T}$ defined on $u\in \ker(\Sq^2_\Z:n)\cap\ker(\Sq^4_\Z:n)$ whose lifts $\tilde{u}_1$ belong to $\ker((\alpha^3,\alpha^4)_\ast^t)$, and takes values $\beta^4\circ \tilde{u}_2$ in $H^{n+4}(X;\Z/2)/\im(\beta^4_\ast\circ (j_2)_\ast)$. Here $\tilde{u}_2\in [X,Q_2S^n]$ is a lift of $\tilde{u}_1$. In other words, $\mathbb{T}$ is the tertiary cohomology operation based on the relation 
 \begin{equation}\label{eq:Adem-rel-T}
	\Sq^2\varTheta+\Sq^1\varPsi=0, \text{ or }\Sq^2\Theta+\Sq^1\Psi=0.
 \end{equation}
 Note that $\mathbb{T}$ detects $\eta^3\in\pi_3^S$; see  \cite[Exercise 4.2.5, Page 108--111]{Harperbook}.
\end{remark}

\subsection{$G$-bordisms and $G$-framed bordisms}\label{subsect:prelim-Gbdsm}

%\begin{lemma}\label{lem:G-structure-correspondence}
%	Let $f\colon N^k\hookrightarrow M^{n+k}$ be an immersion of $G$-manifolds with $n\geq k+2$. Then there is a bijection between the set of $G$-structures on the tangent bundle $TN$ and the set of $G$-structures on the normal bundle $\nu(N)$ that are compatible with the $G$-structure on $f^\ast(TM)$.
%\end{lemma}

	%\begin{proof}	
%Let $\tau_M\colon M\to BG_{n+k}$ be the $G$-structure of the tangent bundle $TM$. By the bundle isomorphism
%\[
%f^\ast(TM) \cong TN \oplus \nu_f,
%\] there is a homotopy commutative diagram
%	   \[\begin{tikzcd}
%&BG_k\times	BG_n\ar[d,"g_{k,n}"]\ar[r,"f_k\times f_n"]&B\mathrm{O}_k\times B\mathrm{O}_n\ar[d,"j_{k,n}"]\\
%N\ar[ur,"{(\tau_N,\nu_N)}",bend left=2ex]\ar[r,"\tau_M\circ f"]&BG_{n+k}\ar[r,"f_{n+k}"]&B\mathrm{O}_{n+k}	
%\end{tikzcd},\]
%where $g_{k,n}$ is the Whitney sum map of $G$-structures. It follows that given a tangential $G$-structure $\tau_N\colon N\to BG_k$, there is a unique normal $G$-structure $\nu_N\colon N\to BG_n$ that is compatible with $f^\ast(TM)$, and vice versa.
%	\end{proof}

We first review the concept of the $G$-bordism group of a space $X$ in terms of stable normal $G$-structures. 
A $G$-manifold in $X$ is a triple $(M^n,f,\sigma)$, where $M^n$ is a closed smooth $n$-manifold, 
$f\colon M \to X$ is a continuous map, and $\sigma$ is a stable normal $G$-structure on $M$.
Two $G$-manifolds $(M_0^n,f_0,\sigma_0)$ and $(M_1^n,f_1,\sigma_1)$ in $X$ are said to be 
\emph{$G$-bordant} if there exists a compact $(n+1)$-dimensional $G$-manifold $(W^{n+1},F,\tau)$ in $X\times I$ such that 
\[\partial W = M_0 \sqcup -M_1,\quad F|_{N_i} = (f_i,i)  \text{ and  }\tau|_{N_i} = \sigma_i \text{ for }i=0,1.\] 
We denote by $[M,f,\sigma]$ the $G$-bordism class of $(M,f,\sigma)$. Denote by $\Omega_n^{G}(X)$ the set of $n$-dimensional $G$-bordism classes in $X$, which is an abelian group under disjoint union;  in particular, denote $\Omega_n^{G}=\Omega_n^{G}(\pt)$.
The $G$-bordism theory is a generalized homology theory represented by the \emph{Thom spectrum} $\mathbf{MG}$, which consists of the Thom spaces $\mathrm{MG}_n$ of the universal vector bundle $\gamma_n$ over $BG_n$; see \cite{Stongbook,Rudyakbook} for more details.

The following is the \emph{generalized Pontryagin-Thom Theorem for $G$-bordism}, where the case $G=\SO$ can be found in \cite{Thom54,Atiyah61}.

\begin{theorem}\label{thm:Pontryagin-Thom}
	Let $n\geq k+2$ and let $X$ be a closed smooth $(n+k)$-manifold with a (stable) tangential $G$-structure. There are group isomorphisms 
	\begin{align*}
		\Omega_k^{G}(X)&\cong \mathbf{MG}_k(X)\cong \mathbf{MG}^n(X)\cong [X,\mathrm{MG}_n].
	\end{align*}
\end{theorem}	

\begin{proof}
	The first isomorphism is due to Lashof \cite{Lashof63}, see also  \cite[Chapter IV, Theorem 7.27]{Rudyakbook}; the second isomorphism is due to the Poincar\'e-Thom duality (\cite[Chapter V, Theorems 2.4, 2.9 and Proposition 2.8]{Rudyakbook}), and the last isomorphism follows by the generalized Freudenthal suspension theorem: $\mathrm{MG}_n$ is $(n-1)$-connected and $X$ has dimension $n+k$ with $n\geq k+2$ implying that the stabilization map
	\[[X,\mathrm{MG}_n]\to [\Sigma^r X,\Sigma^r \mathrm{MG}_n]\to [\Sigma^r X,\mathrm{MG}_{n+r}]\] 
	is an isomorphism for any $r\geq 0$. Hence $\mathbf{MG}^n(X)\cong [X,\mathrm{MG}_n]$.
\end{proof}

There is also the concept of \emph{$G$-framed manifolds}. A $G$-framed manifold in a smooth closed manifold $X$ is a triple $(M,f,\sigma)$ such that $f\colon M\to X$ is an embedding and $\sigma$ is a $G$-structure on the normal bundle $\nu_f$. Two framed $G$-manifolds $(N_0,f_0,\sigma_0)$ and $(N_1,f_1,\sigma_1)$ in $X$ are said to be \emph{framed $G$-bordant} or \emph{$L$-equivalent} if there exists a $(k+1)$-dimensional $G$-framed manifold $(W,F,\tau)$ in $X\times I$ such that
\[\partial W = M_0 \sqcup -M_1,\quad F|_{N_i} = (f_i,i)  \text{ and  }\tau|_{N_i} = \sigma_i \text{ for }i=0,1.\] 
In the stable range $n\geq k+2$, the \emph{$G$-framed bordism set}  $\Omega_k^{\fr,G}(X)$ of framed $G$-bordism classes of $k$-dimensional framed $G$-manifolds in $X$ forms an abelian group under disjoint union. Note that the group $\Omega_k^{\fr,G}(X)$ was denoted by $V^k(X,G)$ in \cite{Novikov07}, and was denoted by $L_k(X)$ with $G=\SO$ in \cite{Thom54,Atiyah61}. 

The following is the \emph{generalized Pontryagin-Thom Theorem for $G$-framed bordism}, see \cite[Lemma 2.1]{Novikov07}. 

\begin{theorem}\label{thm:Pontryagin-Thom-framed}
	Let $X$ be a closed smooth oriented $(n+k)$-manifold with $n\geq k+2$. There is a group isomorphism 
\[\Omega_k^{\fr,G}(X)\cong [X,\mathrm{MG}_n].\]	
\end{theorem}
When $G$ is the trivial group, $\mathrm{MG}_n=S^n$, $\Omega_k^{\fr,G}(X)$ is the normally framed bordism group $\Omega_k^{\fr}(X)$, and Theorem \ref{thm:Pontryagin-Thom-framed} reduces to the classical Pontryagin-Thom theorem for cohomotopy groups: 
\[\Omega_k^{\fr}(X)\cong \pi^n(X).\]

\begin{remark}\label{rmk:bdsm-frbdsm}
Suppose that the ambient manifold $M$ is a closed smooth
$(n+k)$-manifold equipped with a stable tangential $G$-structure, or
equivalently with a stable normal $G$-structure. Then
Theorems~\ref{thm:Pontryagin-Thom} and
\ref{thm:Pontryagin-Thom-framed} yield a natural isomorphism
\[
\Omega_k^G(M)\cong \Omega_k^{\fr,G}(M).
\]

Geometrically, when $n\ge k+2$, Whitney's embedding theorem (see \cite[Theorem 8.9.5]{Mukh15}) implies
that any map $f\colon N^k\to M^{n+k}$ is homotopic to an embedding. For such an embedding, the stable normal bundles satisfy
\[
\nu_N^{st}\cong f^*\nu_M^{st}\oplus \nu_f,
\]
where $\nu_f$ is the normal bundle of the embedding $f\colon N\to M$.
Thus a $G$-structure on $\nu_f$, together with the stable normal
$G$-structure on $M$, determines a stable normal $G$-structure on $N$.
Equivalently, if the $G$-structure on $M$ is described tangentially, the
same conclusion follows after passing to stable normal bundles.
\end{remark}

\section{Stable cohomotopy groups in codimension two}\label{sec:cohtp:codim2}

In this section, we apply the homotopical tools in Section \ref{sec:prelim-Postnikov} to investigate the stable codimension-$2$ cohomotopy groups of manifolds and general CW complexes of dimension $n+2$ with $H_{n+2}(X;\z{})\cong\z{}$.

Let $\delta\colon H^{n-1}(X;\z{})\to H^n(X;\Z)$ be the Bockstein homomorphism associated to the SES $0\to \Z\xra{\times 2}\Z\xra{\rho_2}\z{}\to 0$.
Using the first nontrivial stage of the Postnikov tower in Lemma \ref{lem:tower:Sn}, Taylor \cite[Section 6.1]{Taylor2012} solved the Steenrod problem (\cite{Steenrod1947}) by proving the following lemma in the case $k=1$; the proof of the general case is similar, and we omit the details here.

\begin{lemma}\label{lem:chtp=-1}
	Let $X$ be a CW complex of dimension at most $n+k$ with $n\geq k+2$. There is a SES of groups
	\[0\to QH^{n+1}(X,\Sq^2_\Z)\to [X,P_1S^n]\xra{(p_1)_\ast}\ker(\Sq^2_\Z:n)\to 0,\]
   which splits if and only if 
   \[\Sq^2\big(\delta^{-1}({}_2\ker(\Sq^2_\Z:n))\big)\subseteq \Sq^2_\Z\big(H^{n-1}(X;\Z)\big).\] 
\end{lemma}
Note that when $k=1$ or $\Sq^2_\Z$ acts trivially on $H^n(X;\Z)$, then $\ker(\Sq^2_\Z:n)=H^n(X;\Z)$, and the splitting criterion in Lemma \ref{lem:chtp=-1} turn outs to be 
\[\Sq^2\big(H^{n-1}(X;\z{})\big)=\Sq^2_\Z\big(H^{n-1}(X;\Z)\big).\]

Let $X$ be a CW complex of dimension $n+2$ with $n\geq 4$.
The Postnikov tower of $S^n$ in Lemma \ref{lem:tower:Sn} induces the following two SESs of groups 
	\begin{equation}\label{SES:P1P2}
			\begin{aligned}
	0\to QH^{n+1}(X;\Sq^2_\Z)\xra{(j_1)_\ast} [X,P_1S^n]\xra{(p_1)_\ast}\ker(\Sq^2_\Z:n)\to 0,\\
	0\to QH^{n+2}(X;\Omega\ol{\Sq^2})\to [X,P_2S^n]\xra{(p_2)_\ast}[X,P_1S^n]\to 0.
		\end{aligned} 
	\end{equation}
The canonical map $\imath_2\colon S^n\to P_2S^n$ is $(n+3)$-connected, and hence induces a group isomorphism 
\[(\imath_2)_\ast\colon \pi^n(X)\to  [X, P_2S^n].\]
For $n\geq 4$, the suspension map $E\colon S^{n-1}\to \Omega S^n$ induces homotopy equivalences 
\[P_1S^{n-1}\simeq P_1(\Omega S^{n})\simeq \Omega P_1S^n.\]
Then by (\ref{SES:P1P2}), we get a SES of abelian groups 
		\begin{equation}\label{ses:chtp=-2}
			0\to  QH^{n+2}(X;\ol{\Sq^2})\to \pi^n(X)\xra{}[X,P_1S^n]\to 0.
		\end{equation}

Recall the secondary operation $\varTheta$ based on $\Sq^2\Sq^2_\Z=0$.

\begin{lemma}\label{lem:im-Sq2-P1}
There holds an equality	
\[\ol{\Sq^2}_\ast\big([X,P_1S^{n-1}]\big)=\Sq^2(H^n(X;\z{}))+\varTheta(\ker(\Sq^2_\Z:n-1)).\]
It follows that there are group isomorphisms
\[QH^{n+2}(X;\ol{\Sq^2})\cong\tfrac{H^{n+2}(X;\z{})}{\Sq^2(H^n(X;\z{}))+\varTheta(\ker(\Sq^2_\Z:n-1))}\cong \tfrac{QH^{n+2}(X;\Sq^2)}{\varTheta(\ker(\Sq^2_\Z:n-1))}.\] 
\end{lemma}
	\begin{proof}
		Note that the first SES in (\ref{SES:P1P2}) gives rise to a commutative diagram with exact rows:
\[\begin{tikzcd}[column sep=small]
	0\ar[r]& QH^{n}(X;\Sq^2_\Z)\ar[d,"\Sq^2",two heads]\ar[r,"{(j_1)_\ast}"]& {[X,P_1S^{n-1}]}\ar[r,"{(p_1)_\ast}"]\ar[d,"\ol{\Sq^2}_\ast"]&\ker(\Sq^2_\Z:n-1)\ar[r]\ar[d,"\varTheta"]& 0\\
0\ar[r]&\Sq^2(H^n(X;\z{}))\ar[r]&H^{n+2}(X;\z{})\ar[r]&QH^{n+2}(X;\Sq^2)\ar[r]& 0
\end{tikzcd},\]
where the left-most vertical homomorphism is surjective by the Adem relation $\Sq^2\Sq^2_\Z=0$.
The equaliy in the lemma then follows.
	\end{proof}

\begin{theorem}\label{thm:cohtp=-2:CW}
	Let $n\geq 4$ and let $X$ be a based CW complex of dimension $n+2$ such that $H_{n+2}(X;\z{})\cong \z{}$.
	\begin{enumerate}
	 \item\label{cohtp=-2:Sq2not0} If $\Sq^2$ acts nontrivially on $H^n(X;\z{})$, then there is a SES of groups 
		\[0\to QH^{n+1}(X;\Sq^2_\Z)\to \pi^n(X)\xra{h^n}\ker(\Sq^2_\Z:n)\to 0,\]
		which splits if and only if
		\[\Sq^2\big(\delta^{-1}({}_2\ker(\Sq^2_\Z:n))\big)\subseteq\Sq^2_\Z\big(H^{n-1}(X;\Z)\big).\]
	
	 \item\label{cohtp=-2:Sq2=0} If $\Sq^2$ acts trivially on $H^n(X;\z{})$, then there is a split SES of groups 
	 \[0\to \z{1-\varepsilon(\varTheta)}\oplus QH^{n+1}(X;\Sq^2_\Z)\to \pi^n(X)\xra{h^n}H^n(X;\Z)\to 0,\]
	 where $\varepsilon(\varTheta)=0$ if $\varTheta$ acts trivially on $H^{n-1}(X;\Z)$, and $\varepsilon(\varTheta)=1$ otherwise.  
	\end{enumerate}
\end{theorem}

	\begin{proof}
(1) If $\Sq^2$ acts nontrivially on $H^n(X;\z{})$, then by Lemma \ref{lem:im-Sq2-P1}, the secondary operation $\varTheta$ has zero codomain, and hence 
\[QH^{n+2}(X;\ol{\Sq^2})=0,\quad \pi^n(X)\cong [X,P_1S^n].\] 
The statement in (\ref{cohtp=-2:Sq2not0}) then follows by Lemma \ref{lem:chtp=-1}.

(2) If $\Sq^2$ acts trivially on $H^{n}(X;\z{})$, by the SES (\ref{ses:chtp=-2}) and above arguments, there is a SES of groups 
\begin{equation}\label{SES:cohtp=-2}
			 0\to \z{1-\varepsilon(\varTheta)}\to \pi^n(X)\xra{}[X,P_1S^n]\to 0,
	 \end{equation}
	 where $\varepsilon(\varTheta)\in\{0,1\}$ is defined as in the statement. 
It suffices to determine the group extension for $\varepsilon(\varTheta)=0$. Note that in this case, the SES (\ref{SES:cohtp=-2}) is determined by some homomorphism 
\[_2[X,P_1S^n]\to H^{n+2}(X;\z{})\cong\z{}\]
by \cite[Section 5]{LT72}, and that the first SES in (\ref{SES:P1P2})  implies an inclusion of elementary $2$-torsion groups 
\begin{equation}\label{incl:P1S}
	_2[X,P_1S^n]\subseteq {}_2H^n(X;\Z)\oplus (j_1)_\ast(QH^{n+1}(X;\Sq^2_\Z)),
\end{equation}
where the equality holds if and only if the first SES in (\ref{SES:P1P2}) splits.
Consider the following homotopy commutative diagram with homotopy fibration rows 
\begin{equation*}
	\begin{tikzcd}
		K_{n-1,n+1}\ar[r,"\delta\times \id"]\ar[dd,"{(\Sq^2\Sq^1,0)}"]&K_n(\Z)\times K_{n+1}\ar[r,"{(2,0)}"]\ar[dd,"g"]&K_n(\Z)\ar[d,"\Sq^1_\Z=0"]\ar[r,"{\smatwo{\rho_2}{0}}"]&K_{n}\times K_{n+2}\ar[d,"{\Sq^1\times 0}"]\\
		&&K_{n+1}\ar[r,"i_1"]\ar[d,"j_1"]&K_{n+1}\times K_{n+3}\ar[d,"{(\Sq^2,\id)}"]\\
		K_{n+2}\ar[r,"j_2"]&P_2S^n\ar[r,"p_2"]&P_1S^n\ar[r,"\ol{\Sq^2}"]&K_{n+3}
		\end{tikzcd},
\end{equation*}
where $g$ is an excision map.
Write $H^k_G=H^k(X;G)$ for simplicity, then the above diagram induces a commutative diagram of groups 
\[\begin{tikzcd}[column sep=6ex]
	H^{n-1}_{\Z}\ar[r,"\smatwo{\rho_2}{0}"]\ar[d]&H^{n-1}_{\Z/2}\oplus H^{n+1}_{\Z/2}\ar[d,"{(\Sq^2\Sq^1,0)}"]\ar[r,"\Sq^1\oplus \id"]&H^n_\Z\oplus H^{n+1}_{\z{}}\ar[d,"g_\ast"]\ar[r,"{(\times 2,0)}"]&H^n_{\Z}\ar[d,"0"swap]\\
{[X,P_1S^{n-1}]}\ar[r,"(\ol{\Sq^2})_\ast=0"]&H^{n+2}_{\z{}}\ar[r,"(j_2)_\ast"]&{[X,P_2S^n]}\ar[r,"(p_2)_\ast", two heads]&{[X,P_1S^n]}
\end{tikzcd},\]
where $(\ol{\Sq^2})_\ast=0$ follows by Lemma \ref{lem:im-Sq2-P1} and the above assumptions. 
It follows that the SES (\ref{SES:cohtp=-2}) is determined by the homomorphism 
 \begin{equation*}
	\psi\colon {}_2[X,P_1S^n]\to H^{n+2}(X;\z{}),\quad \alpha\mapsto \Sq^2\Sq^1(x'),
 \end{equation*}
 where $x'\in H^{n-1}(X;\z{})$ satisfies $x=(p_1)_\ast(\alpha)=\delta(x')\in 
 {}_2H^n(X;\Z)$. Since $\Sq^2$ acts trivially on $H^{n}(X;\z{})$, the homomorphism $\psi$ is trivial, and hence the SES (\ref{SES:cohtp=-2}) splits. The splitting criterion in the statement then follows by Lemma \ref{lem:chtp=-1}, and the proof of the theorem is completed.
	\end{proof}

\begin{remark}
	The inclusion (\ref{incl:P1S}) follows from the observation: Given a SES of abelian groups 
\[0\to (\z{})^m\to G\to H\to 0,\]
there is an inclusion of elementary $2$-torsion subgroups
\[_2G\cong (\z{})^m\oplus \ker(\phi)\subseteq (\z{})^m\oplus {}_2H,\]  
where $\phi\colon {}_2H\to (\z{})^m$ is the connecting homomorphism arising from the snake lemma applied to the multiplication by $2$. Indeed, $\phi$ is the homomorphism classifying the extension, and hence the inclusion is an isomorphism if and only if the above SES splits. 
\end{remark}

%The first two nontrivial stages of the Postnikov tower of $S^{n-1}$ in Lemma \ref{lem:tower:Sn} defines the secondary cohomology operation $\varTheta$ as follows: Since $\Sq^2$ acts trivially on $H^{n-1}(M;\z{})$,  $\Sq^2_\Z(u)=\Sq^2\circ \rho_2(u)=0$ for any $u\in H^{n-1}(M;\z{})$, and hence there exists a lift $\tilde{u}\in [M,P_1S^{n-1}]$ of $u$ such that $u\simeq p_1\circ \tilde{u}$. The homotopy clases of all such lifts $\tilde{u}$ are bijective to elements of $(j_1)_\ast\big(H^{n}(M;\z{}))$ and the image of the composition 
%\[(\ol{\Sq^2}\circ j_1)_\ast=\Sq^2\colon H^{n+1}(M;\z{})\to H^{n+3}(M;\z{})\] is the indeterminacy of the operation $\varTheta$, which is zero because $\Sq^2$ acts trivially on $H^{n}(M;\z{})$. Thus the secondary operation $\varTheta$ is given by 
%\[\varTheta\colon H^{n}(M;\z{})\to H^{n+3}(M;\z{}),\quad \varTheta(u)=\ol{\Sq^2}\circ \tilde{u},\]
%which is exactly the secondary cohomology operation based on the Adem relation $\Sq^2\Sq^2_{\Z}=0$. Thus the claim is proved.

%Now we consider the case $X=M$ is a closed smooth oriented manifold. Let $w_k(M),v_k(M)\in H^k(M;\z{})$ be respectively the $k$-th Stiefel-Whitney class and the $k$-th Wu class of the tangent bundle $TM$. 

Next, we apply Theorem \ref{thm:cohtp=-2:CW} to characterize the cohomotopy group in codimension $2$ for closed smooth oriented manifolds.

\begin{lemma}\label{lem:mfd:Sq2} 
	Let $M$ be a connected closed smooth oriented manifold of dimension $n+2$ with $n\geq 3$. Then for any $x\in H^{n-1}(M;\z{})$, the Steenrod square 
	\[\Sq^2\colon H^{n-1}(M;\z{})\to H^{n+1}(M;\z{})\] 
	is given by  
	\[\Sq^2(x)=w_2(M)\!\smallsmile\! x.\] 
It follows that when $M$ is spin, there is a group isomorphism 
\[[M,P_1S^n]\cong H^{n}(M;\Z)\oplus H^{n+1}(M;\z{}).\] 
\end{lemma}

\begin{proof}
Since $M$ is oriented, the first Stiefel-Whitney class $w_1(M)=0$, and hence $\Sq^1$ acts trivially on $H^{n+1}(M;\z{})$. 
For any $x\in H^{n-1}(M;\z{})$ and any $y\in H^1(M;\z{})$, the Cartan
formula gives
\begin{align*}
	\Sq^2(x\!\smallsmile\! y)
&= \Sq^2(x)\!\smallsmile\! y + \Sq^1(x)\!\smallsmile\! \Sq^1(y) + x\!\smallsmile\! \Sq^2(y)\\
&= \Sq^2(x)\!\smallsmile\! y+\Sq^1(x\!\smallsmile\! \Sq^1(y))\\
&= \Sq^2(x)\!\smallsmile\! y.
\end{align*}
By Wu's formula \cite[Page 132]{MSbook}, we then have 
\[
\Sq^2(x\!\smallsmile\! y)= w_2(M)\!\smallsmile\! x\!\smallsmile\! y.
\]
If $H^1(M;\z{})=0$, the formula $\Sq^2(x)=w_2(M)\!\smallsmile\! x$ trivially holds, that is, both sides equal zero; if $H^1(M;\z{})\neq 0$, then we obtain the formula from the nondegeneracy of the cup product pairing
\[H^{n+1}(M;\z{})\times H^1(M;\z{})\to H^{n+2}(M;\z{})\cong\z{}.\]

If $M$ is spin, then $w_2(M)=0$, so $\Sq^2$ acts trivially on $H^{n-1}(M;\z{})$. The isomorphism for $[M,P_1S^n]$ then follows by Lemma \ref{lem:chtp=-1}.
\end{proof}

Recall that the secondary cohomology operation $\Theta$ based on the Adem relation $\Sq^2\Sq^2+\Sq^3\Sq^1=0$ takes the same vales as $\varTheta$ on integral classes.
The proof of the following lemma is similar to the arguments in \cite[Page 32-33]{MMbook}, and we include the details here for later use.
\begin{lemma}\label{lem:spin:Theta=0}
	If $M$ is a connected closed spin $(n+2)$-manifold, then the secondary cohomology operation $\Theta$ acts trivially on $H^{n-1}(M;\z{})$.

	\begin{proof}
	Let $d$ be a sufficiently large integer and let $M\hookrightarrow \R^{n+d+2}$ be an embedding. Since $M$ is spin, $\nu(M)$ is spin and is induced by a map $f\colon M\to B\Spin_d$. 
	By  \cite[Lemma 2.4]{MMbook}, the Thom space $\mathrm{T}(\nu(M))$ of the normal bundle $\nu(M)$ is Spanier-Whitehead dual to the iterated suspension space of $M_+=M\sqcup\{\ast\}$. 
	Since $M$ has a unique top-dimensional cell, dually $\mathrm{T}(\nu(M))$ has a unique bottom cell of dimension $d$. 
   The secondary cohomology operation $\Theta$ is defined on the Thom class $u\in H^{d}(\MSpin_d;\z{})$ and takes the value $\Theta(u)$ in 
  \[H^{d+3}(\MSpin_d;\z{})\cong H^3(B\Spin_d;\z{}),\] 
  which is zero because the classifying space $B\Spin_d$ is $3$-connected. It follows that $\Theta$ acts trivially on the Thom class 
  \[u(M)\in H^d(\mathrm{T}(\nu(M));\z{})\cong\z{},\]  and therefore $\Theta$ acts trivially on $H^{n-1}(M;\z{})$.
	\end{proof} 
\end{lemma}

\begin{theorem}\label{thm:cohtp=-2:mfds}
Let $M$ be a connected closed smooth oriented manifold of dimension $n+2$ with $n\geq 4$. Denote $\varepsilon=0$ if $M$ is spin, and $\varepsilon=1$ otherwise. Then there is a SES of groups 
\begin{equation}\label{SES:codim2:mfds}
	0\to \z{1-\varepsilon}\oplus QH^{n+1}(M;\Sq^2_\Z)\to \pi^n(M)\xra{h^n}\ker(\Sq^2_\Z:n)\to 0,
\end{equation}
which splits if and only if 
 \[\Sq^2\big(\delta^{-1}({}_2\ker(\Sq^2_\Z:n))\big)\subseteq\Sq^2_\Z\big(H^{n-1}(M;\Z)\big).\]
 \begin{proof}
	Combine Theorem \ref{thm:cohtp=-2:CW}, Lemmas \ref{lem:mfd:Sq2} and \ref{lem:spin:Theta=0}, and Remark \ref{rmk:ChlgyOps}.
 \end{proof}
\end{theorem}

%\begin{remark}\label{rmk:chtp=-2:mfds}
%	When $M$ is a smooth spin $(n+2)$-manifold, the split extension in Theorem \ref{thm:cohtp=-2:mfds} (\ref{cohtp=-2:spin}) implies the induced homomorphism
%	\[q^\ast\colon \pi^{n}(S^{n+2})\to \pi^n(M)\]
%	is a split monomorphism, where $q$ is the  pinch map by collapsing the $(n+1)$-skeleton of $M$.
%	This implies the direct summand $\z{}$ of $\pi^n(M)$ is generated by $\eta^2\circ q$, where $\eta^2=\eta_n\circ \eta_{n+1}$ is the composition of the Hopf map $\eta_r\in \pi_{r+1}(S^r)$.
%\end{remark}

%\begin{example}
%	If the manifold $M$ is simply-connected, then Theorem \ref{thm:cohtp=-2:mfds} states that there is a split SES of groups 
%	\[0\to \z{1-\varepsilon}\to \pi^n(M)\xra{h^n}\ker(\Sq^2_\Z:n)\to 0,\]
%	where $\varepsilon=0$ if $M$ is spin, and $\varepsilon=1$ otherwise.
%	Compare the characterizations of codimension $2$ cohomotopy groups in  \cite[Theorem 1.1]{LPW-tlms} and \cite[Theorem 1.3]{HL-7mfd} for simply-connected $6$-, $7$-manifolds, respectively.
%\end{example}

%Next, we compute the cohomotopy group $\pi^n(M)$ for a smooth string $(n+3)$-manifold $M$, $n\geq 5$. %Recall that $B\mathrm{String}_{n+3}$ is the $4$-connected cover of $B\Spin_{n+3}$ and
%the obstruction to lift a map $f\colon M\to B\mathrm{Spin(n+3)}$ to $B\mathrm{String}_{n+3}$ is given by the first fractional Pontryagin class $\frac12p_1(M)\in H^4(M;\z{})$.

\section{Stable cohomotopy groups in codimension three}\label{sec:cohtp:codim3}
Let $n\geq 5$ and let $X$ be a based CW complex of dimension $n+3$ satisfying
\[H_{n+3}(X;\Z/p)\cong\Z/p \text{ for }p=2,3.\]
Identifying $\pi^n(X)$ with $[X,P_3S^n]$, the homotopy fibration sequence (\ref{fib:P3-Q3})
\[K_{n+3}(\Z/3)\xra{j_3} P_3S^n\xra{} Q_3S^n\xra{\PP_\Z\circ q^3_1} K_{n+4}(\Z/3)\]
induces a SES of groups
\begin{equation}\label{SES:XP3-XQ3}
	0\to QH^{n+3}(X;\PP_\Z\circ (q^3_1)_\ast)\xra{(j_3)_\ast} \pi^n(X)\xra{} [X,Q_3S^n]\to 0.
\end{equation}

\subsection{The $2$-primary group $[X,Q_3S^n]$}
%We first determine the group $[X,Q_3S^n]$.
The stable $2$-primary Postnikov tower of $S^n$ in Lemma \ref{lem:tower-Sn-v2} induces the following three SESs of groups
\begin{equation}\label{SESs:Q1Q2Q3}
\begin{aligned}
	0 \to \cok(\Sq^2_\Z,\Sq^4_\Z) \xra{(j_1)_\ast}[X, Q_1S^n] \xrightarrow{(q_1)_\ast} \ker(\Sq^2_\Z:n) \to 0,\\
	0\to\cok(\alpha^3,\alpha^4)_\ast \xra{(j_2)_\ast}  [X, Q_2S^n] 
	\xrightarrow{(q_2)_\ast} \ker(\alpha^3_\ast:n)\to 0,\\
	0\to QH^{n+3}(X; \beta^4)\xra{(j_3)_\ast} [X, Q_3S^n] \xrightarrow{(q_3)_\ast} [X, Q_2S^n] \to 0.
\end{aligned}	
\end{equation}
Here $\ker(\alpha^3_\ast:n)\subseteq [X,Q_1S^n]$, $\cok(\Sq^2_\Z,\Sq^4_\Z)$ and $\cok(\alpha^3,\alpha^4)_\ast$ are respectively quotient groups of the elementary $2$-torsion groups
\[H^{n+1}(X;\z{})\oplus H^{n+3}(X;\z{}),\quad H^{n+2}(X;\Z/2)\oplus H^{n+3}(X;\Z/2).\] 

Given a homomorphism $(f,g)\colon A\to B\oplus C$ of groups, there is an induced SES of groups
\begin{equation}\label{SES:cok(f,g)}
	0\to C/g(\ker(f))\to \cok(f,g)\to \cok(f)\to 0.
\end{equation}
For simplicity, we denote the kernel of $(f,g)$ by $\ker(f\cap g)$, that is,
\[\ker(f\cap g)=\ker(f,g)=\ker(f)\cap \ker(g).\]

\begin{proposition}\label{prop:Q1S}
	 The following hold with $G_1=\tfrac{H^{n+3}(X;\z{})}{\Sq^4_\Z(\ker(\Sq^2_\Z:n-1))}$:
\begin{enumerate}
	\item There is a SES of groups
	\begin{equation}\label{SES:Q1S}
		0\to QH^{n+1}(X;\Sq^2_\Z)\oplus G_1 \xra{}[X, Q_1S^n] \xrightarrow{(q_1)_\ast} \ker(\Sq^2_\Z:n) \to 0,
	\end{equation}
	Moreover, the group $QH^{n+1}(X;\Sq^2_\Z)$ splits off if and only if 
	\[ \Sq^2\left(\delta^{-1}({}_2\ker(\Sq^2_\Z:n))\right)\subseteq \Sq^2_\Z\left(H^{n-1}(X;\Z)\right),\]
	and the group $G_1$ splits off if and only if
	\[\Sq^4\left(\delta^{-1}({}_2\ker(\Sq^2_\Z:n))\right)\subseteq \Sq^4_\Z\left(\ker(\Sq^2_\Z:n-1)\right).\]

 \item Restricting to $\ker(\alpha^3_\ast:n)$, there is an induced SES of groups
	\begin{equation}\label{SES:ker-alpha3}
		0\to \tfrac{\ker(\Sq^2:n+1)}{\Sq^2_\Z(H^{n-1}(X;\Z))} \oplus G_1\to \ker(\alpha^3_\ast:n)\xra{(q_1)_\ast} \ker(\varTheta:n)\to 0,
	\end{equation}
and the splitting criterion is analogous, with ${}_2\ker(\Sq^2_\Z:n)$ above replaced by ${}_2\ker(\varTheta:n)$.
\end{enumerate}
	
\end{proposition}

\begin{proof}
(1)	The SES (\ref{SES:Q1S}) follows from the first SES in (\ref{SESs:Q1Q2Q3}) and the observation above, its splitting criterion  can be proved by similar arguments to that of Lemma \ref{lem:chtp=-1} or \cite[Theorem 6.2]{Taylor2012}, we omit the details here.

(2)	For the SES (\ref{SES:ker-alpha3}), by the homotopy $\alpha^3\circ j_1\simeq (\Sq^2,0)$ and Remark \ref{rmk:ChlgyOps}, there is a commutative diagram with exact rows 
\begin{equation*}
	\begin{tikzcd}[column sep=small]
0\ar[r]&QH^{n+1}(X;\Sq^2_\Z)\oplus G_1\ar[d,"{(\Sq^2,0)}",two heads]\ar[r,"(j_1)_\ast"]&{[X,Q_1S^n]}\ar[d,"\alpha^3_\ast"]\ar[r,"(p_1)_\ast"]&\ker(\Sq^2_\Z:n)\ar[r]\ar[d,"\varTheta"]&0\\
0\ar[r]&\Sq^2(H^{n+1}(X;\z{}))\ar[r]&H^{n+3}(X;\z{})\ar[r]&H^{n+3}(X;\Sq^2)\ar[r]&0
	\end{tikzcd},
\end{equation*}
where the left vertical homomorphism is surjective by the Adem relation $\Sq^2\circ \Sq^2_\Z=0$. The description of $\ker(\alpha^3_\ast:n)$ in (\ref{SES:ker-alpha3}) then follows by the snake lemma, and the splitting criterion is obtained from that of the SES (\ref{SES:Q1S}) by restricting to ${}_2\ker(\varTheta:n)$.
\end{proof}

\begin{remark}\label{rmk:Q1S:n-1}
Denote \[\Sigma^{-1}G_1=\tfrac{H^{n+2}(X;\z{})}{\Sq^4_\Z(\ker(\Sq^2_\Z:n-2))}.\]  
By similar proof of Proposition \ref{prop:Q1S}, there is a SES of groups 
	\[0\to QH^{n}(X;\Sq^2_\Z)\oplus \Sigma^{-1}G_1\xra{(j_1)_\ast}[X,Q_1S^{n-1}]\xra{(q_1)_\ast}\ker(\Sq^2_\Z\cap\Sq^4_\Z:n-1)\to 0,\]
	which induces a SES of groups
\[0\to \tfrac{\ker(\Sq^2:n)}{\Sq^2_\Z(H^{n-2}(X;\Z))} \oplus \Sigma^{-1}G_1\to \ker(\alpha^3_\ast:n-1)\to \ker(\varTheta\cap \Sq^4_\Z:n-1)\to 0.\]
Moreover, there holds an equality
\begin{equation}\label{eq:im-alpha3:n-1}
	\alpha^3_\ast\big([X,Q_1S^{n-1}]\big)= \Sq^2(H^{n}(X;\z{}))+\varTheta(\ker(\Sq^2_\Z\cap \Sq^4_\Z:n-1)),
\end{equation}
and hence there is a group isomorphism 
\begin{equation}\label{eq:cok-alpha3:n-1}
	QH^{n+2}(X;\alpha^3_\ast)\cong \tfrac{QH^{n+2}(X;\Sq^2)}{\varTheta(\ker(\Sq^2_\Z\cap \Sq^4_\Z:n-1))}.
\end{equation}
\end{remark}

\begin{proposition}\label{prop:Q2S}
There is a SES of groups with $G_2=\tfrac{H^{n+3}(X;\z{})}{\alpha^4_\ast(\ker(\alpha^3_\ast:n-1))}$:
	\begin{equation}\label{SES:Q2S}
	0\to \tfrac{QH^{n+2}(X;\Sq^2)}{\varTheta(\ker(\Sq^2_\Z\cap\Sq^4_\Z:n-1))}\oplus G_2\xra{(j_2)_\ast} [X, Q_2S^n]\xra{(q_2)_\ast} \ker(\alpha^3_\ast:n)\to 0.
	\end{equation}
Moreover, the quotient $\tfrac{QH^{n+2}(X;\Sq^2)}{\varTheta(\ker(\Sq^2_\Z\cap\Sq^4_\Z:n-1))}$ splits off, $G_2$ corresponds to a $\Z/4$ direct summand if both $G_1$ and $G_2$ are isomorphic to $\z{}$, otherwise $G_2$ also splits off. 
\end{proposition}

\begin{proof}
The SES (\ref{SES:Q2S}) follows by the second SES in (\ref{SESs:Q1Q2Q3}) and the SES (\ref{SES:cok(f,g)}): there is a SES of groups
\[0\to \cok(\alpha^3,\alpha^4)_\ast\to [X, Q_2S^n]\to \ker(\alpha^3_\ast:n)\to 0\]
with an isomorphism of elementary $2$-torsion groups 
\begin{align*}
	\cok(\alpha^3,\alpha^4)_\ast &\cong QH^{n+2}(X;\alpha^3_\ast)\oplus \tfrac{H^{n+3}(X;\z{})}{\alpha^4_\ast(\ker(\alpha^3_\ast:n-1))},
\end{align*}
where $QH^{n+2}(X;\alpha^3_\ast)$ is described in Remark \ref{rmk:Q1S:n-1}.

For the splitting criterion, by the SES (\ref{SES:ker-alpha3}), there is an inclusion of elementary $2$-torsion groups
\begin{equation}\label{incl:Q2S}
	{}_2\ker(\alpha^3_\ast:n)\subseteq {}_2H^n(X;\Z)\!\oplus\! (j_1)_\ast\big(\tfrac{\ker(\Sq^2:n+1)}{\Sq^2_\Z(H^{n-1}(X;\Z))}\!\oplus\!\tfrac{H^{n+3}(X;\z{})}{\Sq^4_\Z(\ker(\Sq^2_\Z:n-1))}\big).
\end{equation}
Consider the homotopy commutative diagram, in which the first and the third rows homotopy fibration sequences
\[\begin{tikzcd}[row sep=16pt]
	K_{n-1}\!\times\! \Omega K\ar[dd,"\mathrm{St_2}"]\ar[r,"\delta\times \id"]&\!K_n(\Z)\times\!\Omega K\ar[r,"{(2,0)}"]\ar[dd]&\!K_n(\Z)\ar[r,"{\smatwo{\rho_2}{0}}"]\ar[d,"0"]&K_{n}\times K\ar[d,"\Sq^1\times \Sq^1\times \id"]\\
	&&K_{n+1,n+3}\ar[r,"k_2"]\ar[d,"j_1"]  &K_{n+1,n+3,n+4}\ar[d,"\mathrm{St_2'}"]\\
K_{n+2,n+3}\ar[r,"j_2"]&Q_2S^n\ar[r,"q_2"]&Q_1S^n\ar[r,"{\smatwo{\alpha^3}{\alpha^4}}"]&K_{n+3,n+4}
\end{tikzcd},\]	
where $K=K_{n+2}\times K_{n+4}$, $k_2$ is the canonical inclusion map, and 
\[  
%k_2=\left(\begin{smallmatrix}
%  \id&0\\
%  0&\id\\
%  0&0
%\end{smallmatrix}\right), \quad
\mathrm{St_2'}=\begin{pmatrix}
	\Sq^2&0&0\\
	\Sq^2\Sq^1 &\Sq^1&\id\\
\end{pmatrix},\quad  \mathrm{St_2}=\begin{pmatrix}
	\Sq^2\Sq^1&0&0\\
	0&0&\id\\
\end{pmatrix}.\] 
Note that by (\ref{eq:im-alpha3:n-1}), we have
\begin{align*}
	\alpha^3_\ast\big([X,Q_1S^{n-1}]\big)&\supseteq \Sq^2(H^{n}(X;\z{}))\supseteq \Sq^2\Sq^1(H^{n-1}(X;\z{})).
\end{align*}
It follows that the SES (\ref{SES:Q2S}) is determined by the homomorphism 
\[\phi_2\colon {}_2\ker(\alpha^3_\ast:n)\to QH^{n+2}(X;\alpha^3_\ast)\oplus G_2,\quad (x,y,z)\mapsto z,\]
where $(x,y,z)\in {}_2\ker(\alpha^3_\ast:n)$ by the inclusion (\ref{incl:Q2S}).
Thus, in the group extension (\ref{SES:Q2S}), the group $QH^{n+2}(X;\alpha^3_\ast)$ splits off, and the group $G_2$ corresponds to a $\Z/4$ direct summand of $[X,Q_2S^n]$ whenever $G_2\cong G_1\cong\z{}$, otherwise $G_2$ also splits off. 
The proof of the proposition is completed.
\end{proof}

\begin{lemma}\label{lem:im-alpha4}
If $\Sq^1$ acts trivially on $H^{n+2}(X;\z{})$, then there hold 
\begin{align*}
	\alpha^4_\ast\big([X,Q_1S^n]\big)&=\Sq^2\Sq^1(H^n(X;\z{}))+\varPhi\big(\ker(\Sq^2_\Z\cap\Sq^4_\Z:n-1)\big),\\
	\alpha^4_\ast\big(\ker(\alpha^3_\ast:n-1)\big)&=\Sq^2\Sq^1(\ker(\Sq^2:n))+\varPhi\big(\ker(\varTheta\cap\Sq^4_\Z:n-1)\big),
\end{align*}
and there is a SES of groups 
\[0\to \tfrac{\ker(\Sq^2\Sq^1:n)}{\Sq^2_\Z(H^{n-2}(X;\Z))}\to \ker(\alpha^3_\ast\cap\alpha^4_\ast:n-1)\xra{(q_1)_\ast}\ker(\varPhi:n-1) \to 0,\]
where $\ker(\varPhi:n-1)\subseteq \ker(\varTheta\cap\Sq^4_\Z:n-1)$.
\end{lemma}

\begin{proof}
By Remarks \ref{rmk:Q1S:n-1} and \ref{rmk:ChlgyOps}, there is a commutative diagram with rows SESs:
\[\begin{tikzcd}
	QH^{n}(X;\Sq^2_\Z) \oplus \Sigma^{-1}G_1\ar[r,tail]\ar[d,"{(\Sq^2\Sq^1,\Sq^1)}"]&{[X,Q_1S^n]}\ar[r,"(q_1)_\ast",two heads]\ar[d,"\alpha^4_\ast"]& \ker(\Sq^2_\Z\cap \Sq^4_\Z:n-1)\ar[d,"\varPhi"]\\
	 I(\Sq^2\Sq^1,\Sq^1) \ar[r,tail]&H^{n+3}(X;\z{})\ar[r,two heads]& \tfrac{H^{n+3}(X;\z{})}{I(\Sq^2\Sq^1,\Sq^1)}
\end{tikzcd},\]
where $\Sigma^{-1}G_1=\tfrac{H^{n+2}(X;\z{})}{\Sq^4_\Z(\ker(\Sq^2_\Z:n-2))}$, and 
\[I(\Sq^2\Sq^1,\Sq^1)=\Sq^2\Sq^1(H^{n}(X;\z{}))+\Sq^1(H^{n+2}(X;\z{})).\]
When $\Sq^1$ acts trivially on $H^{n+2}(X;\z{})$, the left-vertical homomorphism is surjectively mapped onto $\Sq^2\Sq^1(H^n(X;\z{}))$ by the Adem relation $\Sq^2\Sq^1\Sq^2_\Z=\Sq^1\Sq^4_\Z$ (\ref{eq:Adem-rel-Psi:modSq1}). Thus 
\begin{align*}
	\alpha^4_\ast\big([X,Q_1S^n]\big)%&=\tfrac{\Sq^2\Sq^1(H^n(X;\z{}))}{\Sq^2\Sq^1\Sq^2_\Z(H^{n-2}(X;\Z))}+ \tfrac{\Sq^1(H^{n+2}(X;\z{}))}{\Sq^1\Sq^4_\Z(\ker(\Sq^2_\Z:n-1))}+\im(\varPhi)\\
	&=\Sq^2\Sq^1(H^n(X;\z{}))+\varPhi\big(\ker(\Sq^2_\Z\cap\Sq^4_\Z:n-1)\big),
\end{align*}
which proves the first equality in the statement. 

The proof of the second equality for $\alpha^4_\ast\big(\ker(\alpha^3_\ast:n-1)\big)$ is similar, with the first SES in the above commutative diagram replaced by the SES \eqref{SES:ker-alpha3} for $\ker(\alpha^3_\ast:n-1)$. The proof of the SES for $\ker(\alpha^3_\ast\cap\alpha^4_\ast:n-1)$ then follows by the snake lemma.
\end{proof}

\begin{corollary}\label{cor:Q2S:Sq1=0}
Suppose that $\Sq^1$ acts trivially on $H^{n+2}(X;\z{})$. 
\begin{enumerate}
	\item If $\Sq^2\Sq^1\big(\ker(\Sq^2:n)\big)\neq 0$, then there is a group isomorphism
	\[[X,Q_2S^{n}]\cong \tfrac{QH^{n+2}(X;\Sq^2)}{\varTheta(\ker(\Sq^2_\Z\cap\Sq^4_\Z:n-1))}\oplus \ker(\alpha^3_\ast:n);\]
	\item If $\Sq^2\Sq^1\big(\ker(\Sq^2:n)\big)=0$  but $\Sq^2\Sq^1\big(H^n(X;\z{})\big)\neq 0$, then there is a group isomorphism
	\[[X,Q_2S^{n}]\cong \tfrac{QH^{n+2}(X;\Sq^2)}{\varTheta(\ker(\Sq^2_\Z\cap\Sq^4_\Z:n-1))}\oplus \tfrac{\ker(\alpha^3_\ast:n)}{\z{1-\varepsilon(\Sq^4_\Z)}}\oplus \z{2-\varepsilon(\Sq^4_\Z)};\]
	\item If $\Sq^2\Sq^1\big(H^n(X;\z{})\big)=0$,  then there is a group isomorphism
\[[X,Q_2S^{n}]\cong \tfrac{QH^{n+2}(X;\Sq^2)}{\varTheta(\ker(\Sq^2_\Z\cap\Sq^4_\Z:n-1))}\oplus \tfrac{\ker(\alpha^3_\ast:n)}{\Z/2^{1-\varepsilon(\Sq^4_\Z)}}\oplus \z{2-\varepsilon(\varPhi)-\varepsilon(\Sq^4_\Z)},\]
where $\varepsilon(\Sq^4_\Z)$ and $\varepsilon(\varPhi)$ are defined in \eqref{eq:epsilon-operations}.
\end{enumerate}

\begin{proof}
	By Lemma \ref{lem:im-alpha4}, there is a group isomorphism 
\[G_2\cong\tfrac{H^{n+3}(X;\z{})}{\Sq^2\Sq^1(\ker(\Sq^2:n))+\varPhi\big(\ker(\varTheta\cap\Sq^4_\Z:n-1)\big)}.\]
The discussion of $[X,Q_2S^n]$ then follows by Propositions \ref{prop:Q2S} and \ref{prop:Q1S}.
\end{proof}

\end{corollary}

%\begin{remark}\label{rmk:mfd:Theta=0}
%	Assume that $X$ is a closed smooth oriented $(n+3)$-manifold. If $X$ is spin, then $\varTheta$ acts trivially on $\ker(\Sq^2_\Z:n)$ by Lemma \ref{lem:spin:Theta=0}; otherwise, the codomain of $\varTheta$ is zero, so $\varTheta$ acts trivially on $\ker(\Sq^2_\Z:n)$ as well. Consequently, there holds an equality
%	\[\ker(\varTheta:n)=\ker(\Sq^2_\Z:n).\]	
%\end{remark}

\begin{lemma}\label{lem:im-beta4}
If $\Sq^1$ acts trivially on $H^{n+2}(X;\z{})$, then there hold 
\begin{align*}
	\beta^4_\ast\left([X,Q_2S^{n-1}]\right)&=\Sq^2\left(H^{n+1}(X;\z{})\right)+\mathbb{T}\left(\ker(\alpha^3_\ast\cap\alpha^4_\ast:n-1)\right),\\
	(q_2)_\ast\left(\ker(\beta^4_\ast:n-1)\right)&=\ker(\mathbb{T}:n-1),
\end{align*}
where $\mathbb{T}\colon \ker(\alpha^3_\ast\cap\alpha^4_\ast:n-1)\to QH^{n+3}(X;\Sq^2)$ is the tertiary cohomology operation based on the relation (\ref{eq:Adem-rel-T}).
\end{lemma}

\begin{proof}
By Proposition \ref{prop:Q2S}, there is a commutative diagram with rows SESs:
\[\begin{tikzcd}
	QH^{n+1}(X;\alpha^3_\ast)\oplus \Sigma^{-1}G_2\ar[d,"{(\Sq^2,\Sq^1)}"]\ar[r,"(j_2)_\ast",tail]&{[X,Q_2S^{n-1}]}\ar[d,"\beta^4_\ast"]\ar[r,"(q_2)_\ast", two heads]&\ker(\alpha^3_\ast\cap \alpha^4_\ast:n-1)\ar[d,"\mathbb{T}"]\\
I(\Sq^2,\Sq^1)\ar[r,tail]&H^{n+3}(X;\z{})\ar[r,two heads]&\tfrac{H^{n+3}(X;\z{})}{I(\Sq^2,\Sq^1)}
\end{tikzcd},\]
where $\Sigma^{-1}G_2=\tfrac{H^{n+2}(X;\z{})}{\alpha^4(\ker(\alpha^3:n-2))}$, and
\[I(\Sq^2,\Sq^1)=\Sq^2(H^{n+1}(X;\z{}))+\Sq^1(H^{n+2}(X;\z{})).\] 
It follows that 
\[\beta^4_\ast\big([X,Q_2S^{n-1}]\big)=\Sq^2\big(QH^{n+1}(X;\alpha^3_\ast)\big)+\mathbb{T}\big(\ker(\alpha^3_\ast\cap\alpha^4_\ast:n-1)\big).\]
When $\Sq^1$ acts trivially on $H^{n+2}(X;\z{})$, we note that the left-most vertical homomorphism is surjective:
\begin{align*}
	\Sq^2\big(QH^{n+1}(X;\alpha^3_\ast)\big)&=\Sq^2\big(\tfrac{QH^{n+1}(X;\Sq^2)}{\varTheta(\ker(\Sq^2_\Z\cap\Sq^4_\Z:n-2))}\big)=\tfrac{\Sq^2(H^{n+1}(X;\z{}))}{\Sq^1\varPhi(\ker(\Sq^2_\Z\cap\Sq^4_\Z:n-2))}\\
	&=\Sq^2(H^{n+1}(X;\z{}))=I(\Sq^2,\Sq^1).
\end{align*}
Here the first equality is due to the isomorphism analogous to (\ref{eq:cok-alpha3:n-1}), the second equality follows by the Adem relations 
\[\Sq^2\Sq^2=\Sq^3\Sq^1=\Sq^1\Sq^2\Sq^1,\quad \Sq^2\varTheta=\Sq^1\varPsi,\] 
and the condition $\Sq^1(H^{n+2}(X;\z{}))=0$. Then the descriptions for the image and the kernel of the homomorphism $\beta^4_\ast$ in the statement follow.
\end{proof}

\begin{proposition}\label{prop:Q3S}
Suppose that $\Sq^1$ acts trivially on $H^{n+2}(X;\z{})$. There is a SES of groups
\begin{equation}\label{SES:Q3S}
	0\to \tfrac{QH^{n+3}(X;\Sq^2)}{\mathbb{T}\left(\ker(\alpha^3_\ast\cap\alpha^4_\ast:n-1)\right)}\xra{(j_3)_\ast} [X, Q_3S^n] \xrightarrow{(q_3)_\ast} [X, Q_2S^n] \to 0,
\end{equation}
Furthermore, the following hold:
\begin{enumerate}
	\item If $\Sq^2\left(H^{n+1}(X;\z{})\right)\neq 0$, then there is a group isomorphism 
\[[X,Q_3S^n]\cong [X,Q_2S^n];\]  
\item  If $\Sq^2\left(H^{n+1}(X;\z{})\right)=0$, then there is a group isomorphism 
\[[X,Q_3S^n]\cong \tfrac{[X,Q_2S^n]}{\z{2-\varepsilon(\Sq^4_\Z)-\varepsilon(\varPhi)}}\oplus \z{3-\varepsilon(\Sq^4_\Z)-\varepsilon(\varPhi)-\varepsilon(\mathbb{T})},\]
where $\varepsilon(\Sq^4_\Z)$, $\varepsilon(\varPhi)$ and $\varepsilon(\mathbb{T})$ are defined in \eqref{eq:epsilon-operations}.
\end{enumerate}

\end{proposition}

\begin{proof}
The SES (\ref{SES:Q3S}) follows by the third SES in (\ref{SESs:Q1Q2Q3}) and Lemma \ref{lem:im-beta4}:  there are group isomorphisms
\[QH^{n+3}(X;\beta^4_\ast)\cong\tfrac{H^{n+3}(X;\Z/2)}{\Sq^2\left(H^{n+1}(X;\z{})\right)+\mathbb{T}\left(\ker(\alpha^3_\ast\cap\alpha^4_\ast:n-1)\right)}\cong \tfrac{QH^{n+3}(X;\Sq^2)}{\mathbb{T}\left(\ker(\alpha^3_\ast\cap\alpha^4_\ast:n-1)\right)}.\]

(1) If $\Sq^2\left(H^{n+1}(X;\z{})\right)=\z{}$, then $QH^{n+3}(X;\beta^4_\ast)=0$, and hence 
\[[X,Q_3S^n]\cong [X,Q_2S^n].\] 

(2) If $\Sq^2\left(H^{n+1}(X;\z{})\right)=0$, then $QH^{n+3}(X;\beta^4_\ast)\cong QH^{n+3}(X;\mathbb{T}).$ 
By Proposition \ref{prop:Q2S}, there is an inclusion of elementary $2$-torsion groups	
\begin{align*}
	{}_2[X, Q_2S^n]&\subseteq QH^{n+2}(X;\alpha^3_\ast)\oplus \tfrac{H^{n+3}(X;\z{})}{\alpha^4_\ast(\ker(\alpha^3_\ast:n-1))}\oplus {}_2\ker(\alpha^3_\ast:n)\\
	&\subseteq {}_2 H^n(X;\Z)\!\oplus\! QH^{n+1}(X;\Sq^2_\Z) \!\oplus\! QH^{n+2}(X;\alpha^3_\ast)\!\oplus\! \tfrac{H^{n+3}(X;\z{})}{\alpha^4_\ast(\ker(\alpha^3_\ast:n-1))}. 
\end{align*}
Consider the following homotopy commutative diagram, in which the first and third rows are homotopy fibration sequences:
\[\begin{tikzcd}
	K_{n-1}\times  \Omega K'\ar[dd,"\mathrm{St_3}"]\ar[r,"\delta\times \id"]&K_n(\Z)\times \Omega K'\ar[r,"{(2,0)}"]\ar[dd]&K_n(\Z)\ar[r,"{\smatwo{\rho_2}{0}}"]\ar[d,"\smatwo{\Sq^2_\Z}{0}"]&K_n\times K'\ar[d,"\mathrm{St}_3'"]\\
	&&K_{n+2,n+3}\ar[r,"k_3"]\ar[d,"j_2"] &K_{n+2,n+3,n+4}\ar[d,"{(\Sq^2,\Sq^1,\id)}"]\\
K_{n+3}\ar[r,"j_3"]&Q_3S^n\ar[r,"q_3"]&Q_2S^n\ar[r,"\beta^4"]&K_{n+4}
\end{tikzcd},\]	
where $K'=K_{n+2}\times K_{n+3}\times K_{n+4}$, $\mathrm{St_3}=(\Sq^2\Sq^2,\Sq^2,\Sq^1,\id)$, and 
\begin{align*}
k_3=\begin{pmatrix}
\id&0\\
0&\id\\
0&0
\end{pmatrix},\quad &\mathrm{St_3'}= \begin{pmatrix}
\Sq^2&\id&0&0\\
0&\Sq^1&\id&0\\
0&0&0&\id
\end{pmatrix} 
\end{align*}
It follows that the SES (\ref{SES:Q3S}) is determined by the homomorphism
\begin{equation*}
	\begin{aligned}
	\phi_3\colon {}_2[X,Q_2S^n]&\to QH^{n+3}(X;\beta^4_\ast),\\
	 (x,y,z,w)&\mapsto\Sq^2\Sq^2(x')+\Sq^1(z)+w=w,
\end{aligned}
\end{equation*}
where $(x,y,z,w)\in {}_2[X,Q_2S^n]$ by the inclusion above, and $x'\in H^{n-1}(X;\Z)$ satisfying $\delta(x')=x$. 
The description of $[X,Q_3S^n]$ in the statements then follows by Lemmas \ref{lem:im-alpha4} and \ref{lem:im-beta4}.
\end{proof}

By proofs similar to that of Proposition \ref{prop:Q1S} and Theorem \ref{thm:cohtp=-2:CW}, we obtqin the following characterization of the groups $\ker(\ol{\Sq^2}:n)$ and $[X,P_2S^n]$. 

\begin{proposition}\label{prop:MP2S}
	Let  $n\geq 5$ and let $X$ be a CW complex of dimension $n+3$ such that $H_{n+3}(X;\z{})\cong\z{}$. There is a split SES of groups 
\begin{equation}\label{SES:P2S}
	0\to \tfrac{QH^{n+2}(X;\Sq^2)}{\varTheta(\ker(\Sq^2_\Z:n-1))}\xra{(j_2)_\ast} [X,P_2S^n]\xra{(p_2)_\ast}\ker(\ol{\Sq^2}:n)\to 0,
\end{equation}
where $\ker(\ol{\Sq^2}:n)$ is described by the SES of groups
\begin{equation}\label{SES:kerSq2}
0\to \tfrac{\ker(\Sq^2:n+1)}{\Sq^2_\Z(H^{n-1}(X;\Z))}\xra{(j)_\ast} \ker(\ol{\Sq^2}:n)\xra{(p_1)_\ast} \ker(\varTheta:n)\to 0.
\end{equation}
Moreover, the SES (\ref{SES:kerSq2}) splits if and only if 
 \begin{equation*}
	\Sq^2\big(\delta^{-1}({}_2\ker(\varTheta:n))\big)\subseteq\Sq^2_\Z\big(H^{n-1}(X;\Z)\big).
\end{equation*}

\end{proposition}

\begin{remark}
	Comparing the SESs (\ref{SES:ker-alpha3}) and (\ref{SES:kerSq2}), and combining the first homotopy commutative diagram in Remark \ref{rmk:towers-rels}, we get SESs of groups 
	\begin{align*}
		0\to \z{1-\varepsilon(\Sq^4_\Z)}\to [X,Q_1S^n]\xra{(g_1)_\ast} [X,P_1S^n]\to 0,\\
		0\to \z{1-\varepsilon(\Sq^4_\Z)}\to \ker(\alpha^3_\ast:n)\xra{(g_1)_\ast} \ker(\ol{\Sq^2}_\ast:n)\to 0.
	\end{align*}
	From now on we use the identification
	\[\tfrac{\ker(\alpha^3_\ast:n)}{\z{1-\varepsilon(\Sq^4_\Z)}}=\ker(\ol{\Sq^2}_\ast:n).\] 
%Comparing the SESs (\ref{SES:P2S}) and (\ref{SES:Q2S}), and combining the second homotopy commutative diagram in Remark \ref{rmk:towers-rels}, we get that there are SESs of groups 
%\begin{align*}
%	0\to \ker((g_2)_\ast)\to [X,Q_2S^n]\xra{(g_2)_\ast} [X,P_2S^n]\to 0,\\
%	0\to \tfrac{\varTheta(\ker(\Sq^2_\Z:n-1))}{\varTheta(\ker(\Sq^2_\Z\cap\Sq^4_\Z:n-1))}\to \oplus G_2\to \ker((g_2)_\ast)\to \z{1-\varepsilon(\Sq^4_\Z)}\to 0.
%\end{align*}

\end{remark}

\begin{remark}\label{rmk:MP2S}
When $X$ is a connected closed spin $(n+3)$-manifold with $n\geq 5$, the Pontryagin-Thom construction (Theorem \ref{thm:Pontryagin-Thom}) yields the group isomorphism
	\[\Omega_3^{\Spin}(X)\cong [X,\mathrm{MSpin}_n]\]
for the spin bordism group $\Omega_3^{\Spin}(X)$ of spin $3$-manifolds in $X$.
Let $\mathrm{MSpin}_n$ be the Thom space of the universal $n$-plane spin bundle over $B\Spin_n$ and let $P_{r}\mathrm{MSpin}_n$ be the $(n+r)$-th Postnikov section of $\mathrm{MSpin}_n$.
Since the canonical inclusion map $S^n\to \mathrm{MSpin}_n$ is $(n+3)$-connected and $\Omega_3^{\Spin}=0$, there hold group isomorphisms
\[[X,\mathrm{MSpin}_n]\cong [X,P_{2}\mathrm{MSpin}_n]\cong [X,P_2S^n].\]
Hence Proposition \ref{prop:MP2S} also gives a characterization of  $\Omega_3^{\Spin}(X)$ for connected closed spin $(n+3)$-manifolds.
\end{remark}

\subsection{The cohomotopy $\pi^n(X)$}
Note that the SES (\ref{SES:XP3-XQ3}) of groups
\[0\to QH^{n+3}(X;\PP_\Z\circ (q^3_1)_\ast)\xra{(j_3)_\ast} \pi^n(X)\xra{} [X,Q_3S^n]\to 0\] 
is determined by the homomorphism 
\begin{equation}\label{eq:splitting-homo}
	\psi\colon {}_3[X,Q_3S^n]\to QH^{n+3}(X;\PP_\Z\circ (q^3_1)_\ast),\quad \psi(\gamma)=(j_3)_\ast^{-1}(3\gamma'),
\end{equation}
where $\gamma'\in \pi^n(X)$ is a lift of the $3$-torsion class $\gamma\in {}_3[X,Q_3S^n]$. Since the quotient group $QH^{n+3}(X;\PP_\Z\circ (q^3_1)_\ast)$ is isomorphic to $\Z/3$ or $0$,  the above SES of groups implies that the group $\pi^n(X)$ has three possible isomorphism types:
\begin{align*}
	[X,Q_3S^n],\quad [X,Q_3S^n]\oplus \Z/3,\quad \text{or}\quad \tfrac{[X,Q_3S^n]}{\Z/3^{\bar{r}}}\oplus \Z/3^{\bar{r}+1},
\end{align*}
where $\Z/3^{\bar{r}}$ is a direct summand of $H^n(X;\Z)$.

\begin{remark}\label{rmk:splitting-3local}
 Suppose that $\Sq^1$ acts trivially on $H^{n+2}(X;\z{})$. Then by Lemmas \ref{lem:tower-Sn-v2}, \ref{lem:im-alpha4} and \ref{lem:im-beta4}, we have 
\begin{align*}
\PP_\Z\circ (q^3_1)_\ast\big([X,Q_3S^{n-1}]\big)&=\PP_\Z\circ (q^2_1)_\ast(\ker(\beta^4_\ast:n-1))\\
&=\PP_\Z\circ (q_1)_\ast(\ker(\mathbb{T}:n-1))\\
&\subseteq \PP_\Z(\ker(\varPhi:n-1))\\
&\subseteq \PP_\Z\big(\ker(\varTheta\cap\Sq^4_\Z:n-1)\big).
\end{align*}
This is still very complicated. By Lemma \ref{lem:splitting-3local} below,  if the equality 
\[QH^{n+3}(X;\PP_\Z\circ (q_1^3)_\ast)=QH^{n+3}(X;\PP_\Z)\]  
holds, then the SES (\ref{SES:XP3-XQ3}) splits if and only if 
\begin{equation}\label{eq:splitting-3local}
	\PP_3(\delta^{-1}({}_3\ker(\varTheta:n)))\subseteq \PP_\Z(H^{n-1}(X;\Z)).
\end{equation}
In particular, the SES (\ref{SES:XP3-XQ3}) splits if $\PP_3$ acts trivially on $H^{n-1}(X;\Z/3)$.
\end{remark}

\begin{lemma}\label{lem:splitting-3local}
Let $\psi$ be given by (\ref{eq:splitting-homo}). Localized at $3$, there holds 
\[\psi(\gamma)_{(3)}=[\PP_3(x')]\in QH^{n+3}(X;\PP_{\ZZ{3}}),\]
where $x'\in H^{n-1}(X;\Z/3)$ satisfies $\delta(x')=(q_1^3)_\ast(\gamma)$ with $\delta$ the Bockstein homomorphism. 

\begin{proof}
%By the SESs in (\ref{SESs:Q1Q2Q3}), there is an inclusion of $3$-torsion groups  
%\[(q_1^3)_\ast\big({}_3[X,Q_3S^n]\big)={}_3\ker(\varTheta:n)\subseteq {}_3\ker(\Sq^2_\Z:n).\]
When localized at $3$, the composition $q^3_1\colon Q_3S^{n}\to K_n(\Z)$ is a homotopy equivalence. Consider the following homotopy commutative diagram with homotopy fibration rows
\[\begin{tikzcd}
	K_{n-1}(\Z/3)\ar[d,"\PP_3"]\ar[r,"\delta"]& K_n(\ZZ{3})\ar[r,"\times 3"]\ar[d]& K_n(\ZZ{3})\ar[r,"\rho_3"]\ar[d,"(q_1^3)^{-1}"] & K_n(\Z/3)\ar[d,"\PP_3"]\\
K_{n+3}(\Z/3)\ar[r,"j_3"]&P_3S^n_{(3)}\ar[r]&Q_3S^n_{(3)}\ar[r,"\PP_{\ZZ{3}}\circ q_1^3"]&K_{n+4}(\Z/3)
\end{tikzcd}.\]
It follows that when localized at $3$, the obstruction class $\psi(\gamma)$ is represented by $\PP_3(x')$, where $x'\in H^{n-1}(X;\Z/3)$ is as described in the statement. 
\end{proof}
\end{lemma}

\begin{theorem}\label{thm:cohtp=-3:CW}
Let $n\geq 5$ and let $X$ be a based CW complex of dimension $n+3$ such that $H_{n+3}(X;\Z/p)\cong\Z/p$ for $p=2,3$, and that $\Sq^1$ acts trivially on $H^{n+2}(X;\z{})$.
Then there is a SES of groups 
\begin{equation}\label{SES:cohtp=-3:CW}
	0\to \Z/3^{1-\epsilon}\to \pi^n(X)\to [X,Q_3S^n]\to 0,
\end{equation}
where $\epsilon\in\{0,1\}$, and the group $[X,Q_3S^n]$ is described as follows.
\begin{enumerate}
	\itemsep=5pt
	\item If $\Sq^2$ acts trivially on $H^{n+1}(X;\z{})$, then there is a group isomorphism
    \begin{align*}
		[X,Q_3S^n]\cong \ker(\ol{\Sq^2}:n)\!\oplus\! \tfrac{H^{n+2}(X;\z{})}{\varTheta(\ker(\Sq^2_\Z\cap\Sq^4_\Z:n-1))}\!\oplus\!\z{3-\varepsilon(\Sq^4_\Z)-\varepsilon(\varPhi)-\varepsilon(\mathbb{T})}.
	\end{align*}
\item If $\Sq^2$ acts nontrivially on $H^{n+1}(X;\z{})$, and $\Sq^2\Sq^1$ acts trivially on $H^n(X;\z{})$, then there is a group isomorphism 
 \[[X,Q_3S^n]\cong\ker(\ol{\Sq^2}:n)\!\oplus\!  \tfrac{QH^{n+2}(X;\Sq^2)}{\varTheta(\ker(\Sq^2_\Z\cap\Sq^4_\Z:n-1))}\!\oplus\! \z{2-\varepsilon(\varPhi)-\varepsilon(\Sq^4_\Z)}.\]
\item If $\Sq^2$ acts nontrivially on $H^{n+1}(X;\z{})$,  $\Sq^2\Sq^1$ acts nontrivially on $H^n(X;\z{})$ but acts trivially on $\ker(\Sq^2:n)$, then there is a group isomorphism 
  \[[X,Q_3S^n]\cong \ker(\ol{\Sq^2}:n)\!\oplus\!\tfrac{QH^{n+2}(X;\Sq^2)}{\varTheta(\ker(\Sq^2_\Z\cap\Sq^4_\Z:n-1))}\!\oplus\! \z{2-\varepsilon(\Sq^4_\Z)}.\] 
\item If $\Sq^2$ acts nontrivially on $H^{n+1}(X;\z{})$, and $\Sq^2\Sq^1$ acts nontrivially on $\ker(\Sq^2:n)$, then there is a group isomorphism 
  \[[X,Q_3S^n]\cong \ker(\alpha^3_\ast:n)\!\oplus\! \tfrac{QH^{n+2}(X;\Sq^2)}{\varTheta(\ker(\Sq^2_\Z\cap\Sq^4_\Z:n-1))}.\]
  
\end{enumerate}
If, in addition, the third power operation $\PP_3$ acts trivially on $H^{n-1}(X;\Z/3)$, then the above SES (\ref{SES:cohtp=-3:CW}) splits with $\epsilon=0$.
\begin{proof}
The SES (\ref{SES:cohtp=-3:CW}) follows from the SES (\ref{SES:XP3-XQ3}), the descriptions of $[X,Q_3S^n]$ are obtained by combining Proposition \ref{prop:Q3S} and Corollary \ref{cor:Q2S:Sq1=0}, and the last ``in addition'' statement follows by Lemma \ref{lem:splitting-3local} and Remark \ref{rmk:splitting-3local}.
\end{proof}
\end{theorem}

Let $M$ be a connected closed smooth oriented manifold of dimension $n+3$ with $n\geq 5$. The mod-$3$ Wu's formula \cite{Wu54} states that 
\[\PP_3=\rho_3(p_1(M))\!\smallsmile\! \colon H^{n-1}(M;\Z/3)\to H^{n+3}(M;\Z/3),\]
where $p_1(M)\in H^4(M;\Z)$ is the first Pontryagin class of $M$. By Lemma \ref{lem:mfd:Sq2},
\[\Sq^2=w_2(M)\!\smallsmile\!\colon H^n(M;\z{})\to H^{n+2}(M;\z{}),\quad x\mapsto w_2(M)\!\smallsmile\! x.\] 
Since $w_1(M)=0$, $w_3(M)=\Sq^1(w_2(M))$, we have
\begin{equation*}
	\Sq^2\Sq^1=w_3(M)\!\smallsmile\! \colon H^{n}(M;\z{})\to H^{n+3}(M;\z{}).	
\end{equation*}

By Lemma \ref{lem:spin:Theta=0}, the secondary cohomology operation 
\[\varTheta\colon  \ker(\Sq^2_\Z:n)\to QH^{n+3}(M;\Sq^2)\] 
is always trivial for smooth closed oriented manifolds, so we have 
\[\ker(\varTheta:n)=\ker(\Sq^2_\Z:n),\]
and the SES (\ref{SES:ker-alpha3}) for $\ker(\alpha^3_\ast:n)$ becomes
\[0\to \tfrac{\ker(\Sq^2:n+1)}{\Sq^2_\Z(H^{n-1}(M;\Z))} \!\oplus\! \z{1-\varepsilon(\Sq^4_\Z)}\to \ker(\alpha^3_\ast:n)\xra{(q_1)_\ast} \ker(\Sq^2_\Z:n)\to 0,\]
whose splitting criterion is given in Proposition \ref{prop:Q1S}. 
Now we are ready to characterize  $\pi^n(M)$ for closed smooth oriented $(n+3)$-manifolds by interpretating the group isomorphisms in Theorem \ref{thm:cohtp=-3:CW} into SESs of groups.

\begin{theorem}\label{thm:cohtp=-3:mfds}
	Let $M$ be a connected closed smooth oriented manifold of dimension $n+3$ with $n\geq 5$. 
	\begin{enumerate}
		\itemsep=5pt
		\item If $w_2(M)=0$, then there is a SES of groups 
		\begin{equation*}
			0\to N_1\to \pi^n(M)\to \ker(\ol{\Sq^2}_\ast:n)\!\oplus\! \tfrac{H^{n+2}(M;\z{})}{\varTheta(\ker(\Sq^2_\Z\cap\Sq^4_\Z:n-1))}\to 0,
		\end{equation*}
	where $\epsilon\in\{0,1\}$ and $N_1=\Z/3^{1-\epsilon}\oplus\z{3-\varepsilon(\Sq^4_\Z)-\varepsilon(\varPhi)-\varepsilon(\mathbb{T})}$.
	\item If $w_2(M)\neq 0$ and $w_3(M)=0$, then there is a SES of groups
	\begin{equation*}
	0\to N_2\to\pi^n(M)\to \ker(\ol{\Sq^2}_\ast:n)\!\oplus\! \tfrac{QH^{n+2}(M;\Sq^2)}{\varTheta(\ker(\Sq^2_\Z\cap\Sq^4_\Z:n-1))}\to 0,
	\end{equation*}
	where $\epsilon\in\{0,1\}$ and $N_2=\Z/3^{1-\epsilon}\oplus\z{2-\varepsilon(\Sq^4_\Z)-\varepsilon(\varPhi)}$.
	
    \item If $w_2(M)\neq 0$, $w_3(M)\neq 0$, and $\ker(w_2(M)\!\smallsmile\!:n)\subseteq \ker(w_3(M)\!\smallsmile\!:n)$ as subgroups of $H^n(M;\z{})$, then there is a SES of groups
   \begin{equation*}
		0\to N_3\to \pi^n(M)\to \ker(\ol{\Sq^2}_\ast:n)\!\oplus\! \tfrac{QH^{n+2}(M;\Sq^2)}{\varTheta(\ker(\Sq^2_\Z\cap\Sq^4_\Z:n-1))}\to 0,
	\end{equation*}
	where $\epsilon\in\{0,1\}$ and $N_3=\Z/3^{1-\epsilon}\oplus\z{2-\varepsilon(\Sq^4_\Z)}$.
	\item  If $w_2(M)\neq 0$ and $\ker(w_2(M)\!\smallsmile\!:n)\nsubseteq \ker(w_3(M)\!\smallsmile\!:n)$, then there is a SES of groups
	\begin{equation*}
		0\to \Z/3^{1-\epsilon}\to \pi^n(M)\to\ker(\alpha^3_\ast:n)\!\oplus\! \tfrac{QH^{n+2}(M;\Sq^2)}{\varTheta(\ker(\Sq^2_\Z\cap\Sq^4_\Z:n-1))}\to 0,
	\end{equation*}
where $\epsilon\in\{0,1\}$ and $\ker(\alpha^3_\ast:n)$ contains the direct summand $\Z/2^{1-\varepsilon(\Sq^4_\Z)}$ if and only if
\[\Sq^4\left(\delta^{-1}({}_2\ker(\Sq^2_\Z:n))\right)\subseteq \Sq^4_\Z\left(\ker(\Sq^2_\Z:n-1)\right).\]
	\end{enumerate}
If, in addition, $p_1(M)\equiv 0\pmod 3$, then $\epsilon=0$ and all the above SESs are split.  
\end{theorem}

The following example illustrates the different cases described in Theorem \ref{thm:cohtp=-3:mfds}.
\begin{example}\label{ex:Dold}
	Let $P(m,k)=S^m\times_{\z{}} \CP{k}$ be the Dold manifold, where $\z{}$ acts diagonally. For each $m,k\geq 0$, the projection $S^m\times \CP{k}\to S^m$ induces a fibration
\[\CP{k}\to P(m,k)\to \RP{m}.\]
The Dold manifold $P(m,k)$ is orientable if and only if $m+k\equiv 1\pmod 2$ or $m=0$. 	The mod-$2$ cohomology ring of $P(m,k)$ is given by 
	\[H^\ast(P(m,k);\z{})\cong \z{}[c,d]/(c^{m+1},d^{k+1}),\]
	where $c,d$ are generators with $|c|=1$ and $|d|=2$, and the action of the Steenrod squares $\Sq^i$ are given by 
	\begin{align*}
		&\Sq^1(c)=c^2,~~\Sq^i(c)=0 \text{ for } i>1;\\
		&\Sq^1(d)=cd,~~ \Sq^2(d)=d^2, ~~\Sq^i(d)=0 \text{ for } i>2.
	\end{align*}
%Consequently, the total Stiefel-Whitney class of $P(m,k)$ is 
%	\[w(P(m,k))=(1+c)^m(1+c+d)^{k+1}.\]
See \cite{Dold56,Khare89} for more details.

Let $n\geq 5$ and consider the closed oriented $(n+3)$-manifold
\[M_{m}=P(m,2)\times S^{n-m-1}\] 
with $m$ an odd integer.
Since $H^j(\RP{m};\Z/3)\cong\Z/3$ for $j=0,m$ and is zero otherwise, and the action of $\pi_1(\RP{m})$ on $H^2(\CP{2};\Z/3)$ is nontrivial,
we apply the Leray-Serre spectral sequence
\[E_2^{p,q}=H^p(\RP{m};\mathcal{H}^q(\CP{2};\Z/3))\Rightarrow H^\ast(P(m,2);\Z/3)\] to compute that 
\[H^\ast(P(m,2);\Z/3)\cong \Z/3[u,v]/(u^2,v^2),\]
where $u$ is the pullback of the generators of $H^m(\RP{m};\Z/3)$ and $v$ restricts to the generator $\bar{d}^2$ of $H^4(\CP{2};\Z/3)$.
It follows immediately that the Steenrod power operation $\PP_3$ acts trivially on $H^\ast(P(m,2);\Z/3)$ and hence on $H^{n-1}(M_m;\Z/3)$.
Let $a\in H^{n-m-1}(S^{n-m-1};\z{})$ be a generator.
\begin{enumerate}
	\item For $M_0=\CP{2}\times S^{n-1}$, we have $w_2(M_0)=d$ and $w_3(M_0)=0$.
	\item For $M_1=P(1,2)\times S^{n-2}$, $H^n(M_1;\z{})\cong\Z/2$ is generated by $da$. Since
\begin{align*}
	\Sq^2\Sq^1(da)&=da\!\smallsmile\! cd= cd^2a\neq 0,\\
	 \Sq^2(da)&=d^2a\neq 0,
\end{align*}
we have $\ker(w_3(M_1)\!\smallsmile\!:n)=\ker(w_2(M_1)\!\smallsmile\!:n)=0.$
\item For $M_3=P(3,2)\times S^{n-4}$, $H^n(M_3;\z{})\cong\z{}\oplus\z{}$ are generated by $c^2da$ and $d^2a$. Since
\begin{align*}
	\Sq^2(d^2a+c^2da)&=c^2d^2a+c^2d^2a=0,\\
\Sq^2\Sq^1(d^2a+c^2da)&=\Sq^2(c^3da)=c^3d^2a\neq 0,
\end{align*}
we have $\ker(w_2(M_3)\!\smallsmile\!:n)\nsubseteq \ker(w_3(M_3)\!\smallsmile\!:n)$.
\end{enumerate}

\end{example}

At the end of this section, we deduce the special case of Theorem \ref{thm:cohtp=-3:mfds} in which $M$ is a string $(n+3)$-manifold.
%Recall from Remark \ref{rmk:ChlgyOps} that there are the secondary cohomology operation $\Psi$ and the tertiary cohomology operation $\mathbb{T}$ based on the relations
%\[ \Sq^4\Sq^1+(\Sq^2\Sq^1)\Sq^2+\Sq^1\Sq^4=0,\quad \Sq^2\Theta+\Sq^1\Psi=0.\]

%The proof of the following lemma is similar to that of Lemma \ref{lem:spin:Theta=0}.
\begin{lemma}\label{lem:string:Psi=0}
	If $M$ is a connected closed string $(n+3)$-manifold, then the secondary cohomology operation $\Psi$ and the tertiary cohomology operation $\mathbb{T}$ both act trivially on $H^{n-1}(M;\z{})$.

	\begin{proof}
	Since $M$ is string, the normal bundle $\nu(M)$ in $\R^{n+d+3}$ with $d$ sufficiently large is  string and is induced by a map $f\colon M\to B\mathrm{String}_d$. The secondary operation $\Psi$ is defined on the Thom class $u\in H^d(\mathrm{MString}_d;\z{})$ and takes the value in 
	\[H^{d+4}(\mathrm{MString}_d;\z{})\cong H^4(B\mathrm{String}_d;\z{}),\]
	which is zero because $B\mathrm{String}_d$ is $4$-connected. It follows that $\Psi$ acts trivially on the Thom class $u(M)\in H^d(\mathrm{T}(\nu(M));\z{})\cong\z{}$, and thus $\Psi$ acts trivially on $H^{n-1}(M;\z{})$.	

	Since the secondary operations $\Theta$ and $\Psi$ both acts trivially on the Thom class $u\in H^d(\mathrm{MString}_d;\z{})\cong\z{}$,  the tertiary operation $\mathbb{T}$ is defined on $u$ and takes the values in $H^{d+4}(\mathrm{MString}_d;\z{})=0$. It follows that $\mathbb{T}$ acts trivially on $H^d(\mathrm{T}(\nu(M));\z{})$, and thus acts trivially on $H^{n-1}(M;\z{})$. 
	\end{proof}
\end{lemma}

%Note that one can also prove Lemmas \ref{lem:spin:Theta=0} and \ref{lem:string:Psi=0} by the formulas in \cite[Theorem 3.2.3]{MP64} with $k=1$, where the secondary operations $\phi_1$ and $\phi_2'$ in \cite{MP64} are exactly $\varTheta$ and $\varPsi$, respectively. 

\begin{theorem}\label{thm:cohtp=-3:string}
	Let $M$ be a connected closed string $(n+3)$-manifold with $n\geq 5$. Then there is a split SES of groups 
	\begin{equation*}
			0\to \Z/24\to \pi^n(M)\to \ker(\ol{\Sq^2}:n)\oplus QH^{n+2}(M;\varTheta)\to 0
	\end{equation*}
	with $\ker(\ol{\Sq^2}:n)$ characterized by the SES of groups
	\[0\to QH^{n+1}(M;\Sq^2_\Z)\xra{(j_1)_\ast} \ker(\ol{\Sq^2}:n)\xra{(p_1)_\ast}H^n(M;\Z)\to 0,\]
	which splits if and only if 
	\[\Sq^2_\Z\big(H^{n-1}(M;\Z)\big)=\Sq^2\big(H^{n-1}(M;\Z/2)\big).\]

	\begin{proof}
		Combine Theorem \ref{thm:cohtp=-3:mfds}, Lemma \ref{lem:string:Psi=0}, and Proposition \ref{prop:MP2S}.
	\end{proof}
\end{theorem}

%\begin{remark}\label{rmk:cohtp=-3:string}
%By Remark \ref{rmk:MP2S}, for a connected closed string $(n+3)$-manifold $M$ with $n\geq 5$, the group $\ker(\ol{\Sq^2}:n)\oplus QH^{n+2}(M;\varTheta)$ in Theorem \ref{thm:cohtp=-3:string} is isomorphic to the spin bordism group $\Omega_3^{\Spin}(M)$, and thus by the Pontryagin-Thom theorem and Theorem \ref{thm:cohtp=-3:string}, we obtain the following split SES of groups
%\begin{equation*}
%	0\to \Omega_3^{\fr}\to \Omega_3^{\fr}(M)\to \Omega_3^{\Spin}(M)\to 0.	
%\end{equation*}
%See Example \ref{ex:string-bdsm} for a geometric proof. 
%\end{remark}

%By Remark \ref{rmk:MP2S}, the group $\ker(\ol{\Sq^2}:n)\oplus QH^{n+2}(M;\varTheta)$ in the SES (\ref{SES:cohtp=-3:string}) is isomorphic to the spin bordism group $\Omega_3^{\Spin}(M)$, and thus by the Pontryagin-Thom theorem and Theorem \ref{thm:cohtp=-3:string}, we have the following split SES of groups 
%\begin{equation*}
%			0\to \Omega_3^{\fr}\to \Omega_3^{\fr}(M)\to \Omega_3^{\Spin}(M)\to 0.
%\end{equation*}
%See Example \ref{ex:string-bdsm} or Section \ref{sec:general-Gbdsm} for a geometric proof.

\section{$G$-bordism groups of $G$-manifolds}\label{sec:general-Gbdsm}

Let $G$ be a connected subgroup of the orthogonal group $\mathrm{O}$, with primary examples including $\SO$, $\Spin$, $\String$, and $\Fivebrane$. In this section we establish the framework to study normally framed bordism groups $\Omega_k^{\fr}(M)$ via $G$-bordism groups $\Omega_k^G(M)$ along the Whitehead tower of $B\mathrm{O}$. %For the original sources and detailed treatments of stable normal or tangential $G$-structures on manifolds, see \cite{Lashof63,Atiyah61,Stongbook}.

Let $B\mathrm{O}_n\langle m\rangle$ denote the $(m-1)$-connected cover of the classifying space $B\mathrm{O}_n$. Let $\pi\colon E\to N^k$ be a smooth $n$-plane bundle over a smooth $k$-manifold $N$ that is classified by a map $f\colon N\to B\mathrm{O}_n$.  As illustrated in Figure \ref{fig:Whitehead-tower-BO}, we recall that
\begin{itemize}
	\item an \emph{orientation} of $E$ is a (homotopy class of) lift $f_1$ of $f$ to $B\SO_n=B\mathrm{O}_n\lra{2}$, which exists if and only if $w_1(E)=0$;
	\item a \emph{spin structure} on $E$ is a (homotopy class of) lift $f_2$ of $f_1$ to $B\mathrm{Spin}_n=B\mathrm{O}_n\lra{4}$, which exists if and only if $w_2(E)=0$;
	\item a \emph{string structure} on $E$ is a (homotopy class of) lift $f_4$ of $f_2$ to $B\mathrm{String}_n=B\mathrm{O}_n\lra{8}$, which exists if and only if $\frac12p_1(E)=0$;
	\item a \emph{fivebrane structure} on $E$ is a (homotopy class of) lift $f_8$ of $f_4$ to $B\mathrm{Fivebrane}_n=B\mathrm{O}_n\lra{9}$, which exists if and only if $\frac16p_2(E)=0$.
\end{itemize}
 \begin{figure}[H]
		\centering
	\begin{tikzcd}[sep=12pt]
		&&B\mathrm{Fivebrane}_n\ar[d]&\\
		&&B\mathrm{String}_n\ar[d]\ar[r,"\frac16p_2"]&K(\Z,8)
		\\
		&&B\mathrm{Spin}_n\ar[d]\ar[r,"{\frac{1}{2}p_1}"]&K(\Z,4)\\
		&&B\mathrm{SO}_n\ar[d]\ar[r,"{w_2}"]&K(\Z/2,2)\\
		N^k\ar[rr,"f"]\ar[urr,"f_1", bend left=1ex,dashed]\ar[uurr,"f_2", bend left=2ex, dashed]\ar[uuurr,"f_4", bend left=3ex, dashed]\ar[uuuurr,"f_8", bend left=4ex, dashed]
		&&B\mathrm{O}_n\ar[r,"{w_1}"]&K(\Z/2,1)
		\end{tikzcd}
		\caption{The Whitehead tower of $B\mathrm{O}_n$, $n\geq k$.}\label{fig:Whitehead-tower-BO}
\end{figure}
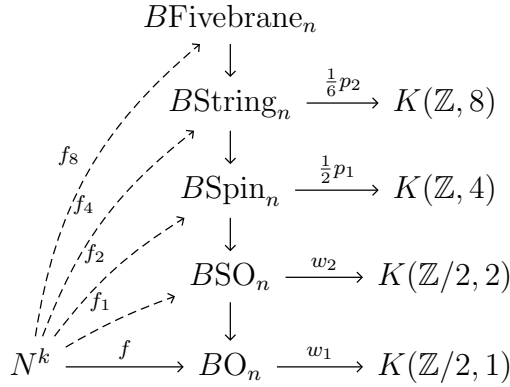
\begin{lemma}\label{lem:fr-G}
	Let $n\geq k+2$ and let $M$ be a closed smooth $(n+k)$-manifold  with a tangential $\mathrm{O}\lra{m}$-structure. The canonical homomorphism
\[\Omega_k^{\fr}(M)\to \Omega_k^{\mathrm{O}\lra{m}}(M) \]
is an isomorphism for $k\leq m-2$.
In particular, there are canonical group isomorphisms
\[
\Omega_k^{\fr}(M)\cong \left\{\begin{array}{ll}
	\Omega_k^{\Spin}(M) &\text{ for $k\leq 2$},\\[1ex]
	\Omega_k^{\String}(M) &\text{ for $k\leq 6$},\\[1ex]
	\Omega_k^{\Fivebrane}(M) &\text{ for $k\leq 7$}.	
\end{array}\right.
\]	
\end{lemma}

\begin{proof}
 Recall that there holds the Thom isomorphisms for any $k\geq 0$:
\[H_{n+k}(\mathrm{\mathrm{MO}}_n\lra{m};\Z)\cong H_{k}(B\mathrm{O}_n\lra{m};\Z) .\]
Let $i_{n,m}\colon S^n\to \mathrm{\mathrm{MO}}_n\lra{m}$ be the inclusion map of the bottom cell. Since $B\mathrm{O}_n\langle m\rangle$ is $(m-1)$-connected, the inclusion map $i_{n,m}$ is $(n+m-1)$-connected. It follows that the induced homomorphism
\[(i_{n,m})_\ast\colon [M,S^n]\to [M,\mathrm{\mathrm{MO}}_n\lra{m}]\]
is an isomorphism for $n+k\leq n+m-2$ or $k\leq m-2$. The Lemma then follows by the Pontryagin-Thom Theorem \ref{thm:Pontryagin-Thom}.

The ``particular'' group isomorphisms follow immediately.
\end{proof}

Note that the inclusion map $i_{n,m}$ in the proof of Lemma \ref{lem:fr-G} is $(n+m-1)$-connected implying that the unit map $\mathbf{S}\to\varinjlim_m\mathbf{MO}\lra{m}$ of spectra is a homotopy equivalence, see also \cite[Proposition 2.1.1]{Hovey97}. 

%Consider the fibration $BG\xra{\hbar}BH\xra{\lambda_G}K(A_G,d_G)$. Let $E\to N$ be a vector bundle with an $H$-structure, classified by a map $c_H\colon N\to BH$. A $G$-structure on $E$ is a lift $c_G\colon N\to BG$ of $c_H$ such that $\hbar\circ c_G=c_H$. If one such lift exists, the set of $G$-structures on $E$ is a torsor for $H^{d_G-1}(N;A_G)$: For any two $G$-structures $\sigma$ and $\sigma'$ on $E$, there exists a unique element $\varepsilon\in H^{d_G-1}(N;A_G)$ such that \[\sigma'=\sigma+\varepsilon.\]

The Whitehead tower in Figure \ref{fig:Whitehead-tower-BO} constitutes homotopy fibrations of the form 
\begin{equation}\label{fib:h_k}
	BG_n\xra{\hbar}BH_n\xra{\lambda_G}K(A_G,d_G),
\end{equation}
where $\lambda_G$ is the universal characteristic class to lifting an $H$-structure to a $G$-structure, and $A_G=\z{}$ for $G=\SO,\Spin$ and $A_G=\Z$ for $G=\String, \Fivebrane$. 
Let $M$ be a connected closed $G$-manifold of dimension $n+k$ with $n\geq k+2$. Consider the forgetful homomorphism 
\begin{equation}\label{eq:h_k}
	h_k \colon \Omega_k^G(M) \to \Omega_k^H(M),\quad [N,f,\sigma]\mapsto [N,f,\sigma_H],
\end{equation}
where $\sigma$ is a $G$-structure on the normal bundle $\nu_f$ and $\sigma_H$ is the induced $H$-structure on $\nu_f$.
%First, we have the following observation. 

\begin{lemma}\label{lem:surj-h_k}
The forgetful homomorphism $h_k$ in (\ref{eq:h_k}) is surjective in the following cases: 
\begin{enumerate}
	\item $k\leq d_{G}-1$;
	\item $G=\Spin, H=\SO$ and $k\leq 3$. 
\end{enumerate}
\begin{proof}
Let $[N,f,\sigma_H]\in \Omega_k^H(M)$, where $f\colon N\to M$ is (homotopic to) an embedding.
The obstruction to lifting the normal $H$-structure $\sigma$ to a $G$ structure is given by the characteristic class $\lambda_G(\nu_f)\in H^{d_G}(N;A_G)$.   If $k\leq d_G-1$, then $H^{d_G}(N;A_G)=0$, and hence such a lift of $\sigma$ exists. If $G=\Spin$ and $H=\SO$, then $M$ is spin and every oriented manifold $N$ of dimension $k\leq 3$ is spin implying that 
\[w_2(\nu_f)=w_2(N)=0\] by the Whitney sum $f^\ast(TM)\cong TN\oplus \nu_f.$
%:
%\begin{align*}
%	w_2(\nu_f)=w_2(N)=\Sq^1(v_1(N))=\Sq^1(w_1(N))=w_1(N)^2=0,
%\end{align*}
%where $v_1(N)$ is the first Wu class and the formula $w_2(N)=\Sq^1(v_1(N))$ refers to  \cite[Page 132]{MSbook}; see also  \cite[Page 88]{LM89}. 
Thus, in both cases, there exists a normal $G$-structure $\sigma$ on $\nu_f$ lifting the normal $H$-structure $\sigma_H$.
\end{proof}

\end{lemma}

%Next, we study the kernel of the forgetful homomorphism $h_k$ in (\ref{eq:h_k}).
The canonical inclusion $e\colon \{x_0\}\to M$ of the base-point $x_0$ of $M$ induces the natural homomorphism 
\[e_\ast\colon \Omega_k^G\to \Omega_k^G(M)\]
with a retraction $c_\ast\colon \Omega_k^G(M)\to \Omega_k^G$ induced by the constant map $c\colon M\to \{x_0\}$. 
When $\Omega_k^H=0$, we observe that $e_\ast$ takes values in $\ker(h_k)$. %The following lemma shows that the monomorphism $e_\ast$ completely describes the kernel $\ker(h_k)$ if additionally $k=d_G-1$.

\begin{lemma}\label{lem:kernel-h_k}
If $M$ is connected, $\Omega_k^H=0$ and $d_G=k+1$, then the homomorphism 
\[e_\ast\colon \Omega_k^G\to \ker(h_k)\]
is an isomorphism.

\end{lemma}

\begin{proof}
It suffices to show that $e_\ast$ is surjective. 
Given $[N,f,\sigma]\in \ker(h_k)$, let $(W,F,\tau)$ be an $H$ null-bordism of $(N,f,\sigma_H)$ in $M\times I$. The obstruction class $\lambda_G(\nu_F,\sigma)$ to lifting the normal $H$-structure $\tau$ to a normal $G$-structure on $\nu_F$ that extends $\sigma$ lies in the cohomology group $H^{d_G}(W,\partial W;A_G)$ by the fibration (\ref{fib:h_k}). 
Applying the Poincar\'e-Lefschetz duality yields a homology class 
	\[[P]=\mathrm{PD}(\lambda_G(\nu_F,\sigma))\in H_{k+1-d_G}(W;A_G).\]
Since $d_G=k+1$, $P$ is represented by a collection of signed points $\{p_1,\cdots,p_m\}$ in the interior of $W$. In other words, the normal $G$-structure $\sigma$ of $N$ can be extended over $W$ away from these isolated points $p_1,\cdots,p_m$. Let $\bar{\tau}$ be a normal $G$-structure of $W\setminus P$ that extends $\sigma$. For each point $p_i\in P$, choose a small $(k+1)$-disk neighborhood $D_i^{k+1}\subset \mathrm{int}(W)$ of $p_i$ such that the disks $D_i^{k+1}$ are mutually disjoint. Form the $(k+1)$-manifold
\[V=W\setminus \big(\bigsqcup_{i=1}^tD^{k+1}_i\big)\]
with $\partial V=N\sqcup\bigsqcup_{i=1}^t S^k_i$, where $S^k_i=\partial D^{k+1}_i$. Then the $G$-bordism $(V,\bar{\tau}|_V)$ yields the equality
\[[N,\sigma]=\sum_{i=1}^t[S^k_i,\bar{\tau}_i]\in \Omega_k^G,\]
where $\bar{\tau}_i=\bar{\tau}|_{S^k_i}$ is the induced $G$-structure.
Let $F_i=F|_{S^{k}_i}\colon S^k_i\to M$ be the restriction map. Then $F_i$ is null-homotopic because it extends over the disk $D^{k+1}_i$, and hence 
\[[S^k_i,F_i,\bar{\tau}_i]= [S^k_i,c_{q_i},\bar{\tau}_i]=e_\ast([S^k_i,\bar{\tau}_i]),\]
where $c_{q_i}$ is the constant map to the point $q_i=F(p_i)\in M\times  I$.  
Moreover, the $G$-bordism class $[V,F|_V,\bar{\tau}_V]$ gives a $G$-bordism between $(N,f,\sigma)$ and the disjoint union $\bigsqcup_{i=1}^t (S^k_i,F_i,\bar{\tau}_i)$ in $M\times I$. Thus we have
\[[N,f,\sigma]=\sum_{i=1}^t [S^k_i,F_i,\bar{\tau}_i]=\sum_{i=1}^t e_\ast\big([S^k_i,\bar{\tau}_i]\big)=e_\ast([N,\sigma])\in e_\ast(\Omega_k^G),\]
which completes the proof that $e_\ast\colon \Omega_k^G\to \ker(h_k)$ is surjective. 
\end{proof}

Combining Lemmas \ref{lem:surj-h_k} and \ref{lem:kernel-h_k},  we get the main result of this section.

\begin{theorem}\label{thm:SES-G-H}
	Let $M$ be a closed connected smooth $G$-manifold of dimension $n+k$ with $n\geq k+2$, and suppose that $\Omega_k^H=0$ and $d_G=k+1$. Then there is a split SES of groups
\[0\to \Omega_k^G\xra{e_\ast} \Omega_k^G(M)\xra{h_k}\Omega_k^H(M)\to 0,\]
whose splitting retraction is induced by the constant map $c\colon M\to \{x_0\}$.
\end{theorem}

\begin{example}\label{ex:spin-1-bdsm}
	When $k=1$,  Theorem \ref{thm:SES-G-H} states that for a closed spin $(n+1)$-manifold $M$ with $n\geq 3$, there is a split SES of groups
	\[0\to \Omega_1^{\Spin}\xra{e_\ast}\Omega_1^{\Spin}(M)\xra{h_1}\Omega_1^{\SO}(M)\to 0.\]
	This recovers Konstantis's geometric interpretation of $\pi^n(M)$ for spin $(n+1)$-manifolds in \cite{Kons2020} by the group isomorphisms 
	\[\Omega_1^{\fr}(M)\cong \Omega_1^{\Spin}(M), \quad \Omega_1^{\fr}\cong\Omega_1^{\Spin},\quad \Omega_1^{\SO}(M)\cong H_1(M;\Z).\]
\end{example}

\begin{example}\label{ex:string-bdsm}
	When $k=3$, Theorem \ref{thm:SES-G-H} states that for a closed string $(n+3)$-manifold $M$ with $n\geq 5$, there is a split SES of groups
	\[0\to \Omega_3^{\String}\xra{e_\ast}\Omega_3^{\String}(M)\xra{h_3}\Omega_3^{\Spin}(M)\to 0.\]
   This is a geometric interpretation of the codimension-$3$ cohomotopy group $\pi^n(M)$ for string manifolds in Theorem \ref{thm:cohtp=-3:string} by Remark \ref{rmk:MP2S}.   %: Since $\Omega_3^{\Spin}=0$,  the canonical inclusion map $\iota_n\colon S^n\to \MSpin_n$ of the bottom cell induces natural homomotopy equivalences \[ P_3\MSpin_n\simeq P_2\MSpin_n\simeq P_2S^n,\]   
  %and thus there are natural group isomorphisms
  % \[\Omega_3^{\Spin}(M)\cong [M,\MSpin_n]\cong [M,P_3\MSpin_n]\cong [M,P_2S^n].\]
\end{example}

%\begin{remark}
%	If the $(n+3)$-manifold $M$ is spin but not string,  Theorem \ref{thm:cohtp=-3:mfds} admits a geometric interpretation: there is a split SES of groups
%\[0\to \ker(h_3)\to \Omega_3^{\fr}(M)\xra{h_3}\Omega_3^{\Spin}(M)\to 0,\] 
%where $\ker(h_3)\cong \Z/2^{t}\oplus \Z/3^{\epsilon}$ is a quotient group of $\Omega_3^{\fr}\cong\Z/24$.  The direct summand $\Z/2^t$ is related to the actions of the secondary operation $\Psi$ and the tertiary operation $\mathbb{T}$ on $H^{n-1}(M;\z{})$. However, the geometric counterparts of $\Z/2^{t}$ in normally framed bordism remain unknown, and it would be interesting to investigate the related bordism invariants.
%\end{remark}

%Note that when $n\geq 5$, the Thom class $u\colon \mathrm{MSO}_n\to K_n(\Z)$ is $(n+4)$-connected, hence there is an isomorphism $u_\ast\colon\Omega_3^{\SO}(M;\Z)\to H_3(M;\Z)$. Then by Remark \ref{rmk:string:MP2}, there is a SES of groups 
%\begin{equation}\label{SES:3rdSpin}
%	0\to K\to \Omega_3^{\Spin}(M)\xra{h_3}H_3(M;\Z)\to 0,
%\end{equation}
%  where $K$ fits in the SES in (\ref{SES:kernel-K}). The SESs in (\ref{SES:3rdSpin}) and (\ref{SES:kernel-K}) may not split. For instance,  there hold isomorphisms for $m\geq 2$: 
%  \[\Omega_3^{\Spin}(\RP{4m+3})\cong \mathbf{ko}_3(\RP{4m+3})\cong \mathbf{ko}_3(\RP{\infty})\cong\Z/8,\] 
% where the first isomorphism follows by Lemma \ref{lem:h_kG} and the last one refers to  \cite[12.2.D]{BG10}. 

\begin{example}\label{ex:fivebrane-bdsm}
	When $k=7$, $\Omega_7^{\String}=0$ by \cite{Giambalvo71}. Then Theorem \ref{thm:SES-G-H} states that for a closed fivebrane $(n+7)$-manifold $M$ with $n\geq 9$, there is a split SES of groups 
	\[0\to \Omega_7^{\Fivebrane}\xra{e_\ast} \Omega_7^{\Fivebrane}(M)\xra{h_7}\Omega_7^{\String}(M)\to 0.\]
	This is a geometric interpretation of the stable cohomotopy groups in codimension $7$ for fivebrane manifolds.  
\end{example}

\section{Framed bordism groups in dimension two}\label{sec:cohtp-2-geom}

The main purpose of this section is to provide a geometric proof of Theorem \ref{mainthm:frbdsm-dim2}. Let $M$ be a closed smooth oriented $(n+2)$-manifold with $n\geq 4$, and consider the forgetful homomorphism
\[h_2\colon \Omega_2^{\fr}(M)\to \Omega_2^{\SO}(M), \quad [N,f,\varphi]\mapsto [N,f],\]
where $N$ has the orientation induced by the normal framing $\varphi$.

\begin{lemma}\label{lem:SO-H2}
The canonical homomorphism
\[\Omega^{\SO}_2(M)\to H_2(M;\Z),\quad [N,f]\mapsto f_\ast([N])\]
is a group isomorphism, where $[N]\in H_2(N;\Z)$ is the fundamental class.
\end{lemma}

\begin{proof}
%By Theorem \ref{thm:Pontryagin-Thom}, there hold group isomorphisms 
%\[\pi_{n+k}(\mathrm{MSO}_n)\cong \Omega_k^{\SO}\text{ for }n\geq k+2.\]   
Since $\Omega_0^{\SO}\cong\Z$ and $\Omega_{i}^{\SO}=0$ for $i=1,2,3$ \cite{Rohlin51},  the Thom class map $u\colon \mathrm{MSO}_n\to K_n(\Z)$ is $(n+3)$-connected for $n\geq 4$, and hence induces a group isomorphism \[u_\ast\colon [M,\MSO_n]\to H^n(M;\Z)\]  
for the $(n+2)$-manifold $M$.
The lemma then follows by the Pontryagin-Thom Theorem \ref{thm:Pontryagin-Thom} and the Poincar\'e duality. 
\end{proof}

Recall that a framing of an oriented vector bundle $E\to X$ induces a canonical spin structure on $E$ that is compatible with the orientation. 
Using the fact that $B\Spin_n$ is $3$-connected, we apply similar arguments to the proof of \cite[Lemma 3.1]{Kons2020} to prove the following lemma. 
\begin{lemma}\label{lem:spinnable-trivial}
Let $X$ be a CW complex of dimension $\leq 3$ and let $E\to X$ be an oriented $n$-plane bundle such that $w_2(E)=0$ and $n\geq 3$. Then $E$ is isomorphic to the trivial bundle. Moreover, if $X$ has dimension $\leq 2$, then a choice of a spin structure on $E$ determines a framing on $E$ compatible with the orientation.
\end{lemma}

\subsection{The spin case}\label{sec:spin-bdsm}
When $M$ is spin, by Lemma \ref{lem:fr-G} we identify the framed bordism group $\Omega_2^{\fr}(M)$ with the spin bordism group $\Omega_2^{\Spin}(M)$.  Recall that a normal spin structure $\sigma$ on $\nu_f$ is a principal $\Spin_n$-bundle $P_{\sigma}\to N$ together with a $2$-fold covering $P_{\sigma}\to \mathrm{Fr}(\nu_f)$ over the orthonormal frame bundle, and the set of a spin vector bundle $E\to N$ is a $H^1(N;\z{})$-torsor, which means that for any two spin structures $\sigma,\sigma_1$ on $E$, there exists a unique $\alpha\in H^1(N;\z{})$ such that 
\[\sigma_1=\sigma+\alpha.\]  
Note that $\alpha\in H^1(N;\z{})$ corresponds to a principal $\z{}$-bundle $Q_\alpha\to N$, and that $\sigma+\alpha$ corresponds to the spin bundle $P_\sigma\times_{\z{}}Q_{\alpha}$ with $\z{}$ acts diagonally.

We shall study the forgetful homomorphism 
\begin{equation}\label{eq:h_2:spin}
	h_2\colon \Omega_2^{\Spin}(M)\to \Omega_2^{\SO}(M),\quad [N,f,\sigma]\mapsto [N,f],
\end{equation}
where $\sigma$ is a spin structure on the normal bundle $\nu_f$ of $N$. 

\begin{lemma}\label{lem:h2-splitting:spin}
	The homomorphism $h_2$ in (\ref{eq:h_2:spin}) is surjective, and admits a canonical splitting section $s_0\colon \Omega_2^{\SO}(M)\to \Omega_2^{\Spin}(M)$ sending $[N,f]$ to $[N,f,\sigma_0]$, where $\sigma_0$ is the canonical normal spin structure on $N$ that is compactible with the orientation.
\end{lemma}

\begin{proof}
   The surjectivity of $h_2$ refers to Lemma \ref{lem:surj-h_k} (2) with $k=2$. For the splitting section $s_0$, observe that given a class $[N,f]\in \Omega_2^{\SO}(M)$, the Whitney sum decomposition 
   $f^\ast(TM)\cong TN\oplus \nu_f$
   implies that the normal bundle $\nu_f$ is spinnable, and hence trivial by Lemma \ref{lem:spinnable-trivial}. The set of trivializations of $\nu_f$ is bijective to the homotopy set $[N,B\Spin_n]$, which is a singleton because $B\Spin_n$ is $3$-connected and $N$ has dimension $2$. Hence there is a unique trivialization of $\nu_f$, which determines a canonical (or unique) spin structure $\sigma_0$ on $\nu_f$ that is compatible with the orientation. Thus the assignment $s_0\colon \Omega_2^{\SO}(M)\to \Omega_2^{\Spin}(M)$ sending $[N,f]$ to $[N,f,\sigma_0]$ is well-defined. It is clear that $s_0$ is a homomorphism and a splitting section of $h_2$.
\end{proof}

Next, we study the kernel $\ker(h_2)$ in (\ref{eq:h_2:spin}). Let $[N,f,\sigma]\in \ker(h_2)$ be represented by a closed smooth spin $2$-manifold $(N,f,\sigma)$ in $M$. There exists an oriented bordism $(W,F)$ in $M\times I$ such that $\partial W=N$, $F|_{\partial W}=(f,0)$ and the orientation of $W$ restricted to the orientation on $N$ induced by the normal spin structure $\sigma$. The obstruction to extending $\sigma$ over $W$ is given by the relative Stiefel-Whitney class $w_2(\nu_F,\sigma)\in H^2(W,N;\z{})$. 
Applying the Poincar\'e-Lefschetz duality, we get the dual homology class
\[[L]=\mathrm{PD}(w_2(\nu_F,\sigma))=w_2(\nu_F,\sigma)\smallfrown [W,N]\in H_1(W;\z{}).\]  
Identifying $H_1(M;\z{})$ with $H_1(M\times I;\z{})$, we define
	\begin{equation}\label{eq:Phi0-spin}
		\Phi_0\colon \ker(h_2)\to \Omega_2^{\Spin}\!\oplus\! H_1(M;\z{}), ~~ [N,f,\sigma]\mapsto \big([N,\sigma], F_\ast([L])\big).
	\end{equation}
The map $\Phi_0$ is well-defined by Lemma \ref{lem:Phi-welldef} below and is clearly a homomorphism because addition of bordism classes is given by disjoint union of bordism classes.

\begin{lemma}\label{lem:Phi-welldef}
	The map $\Phi_0$ in (\ref{eq:Phi0-spin}) is a well-defined homomorphism.
\end{lemma}

\begin{proof}
We need to show that $\Phi_0$ is independent of the choices of (1) the oriented null-bordisms $(W,F)$, and $(2)$ the representatives of the bordism class $[N,f,\sigma]\in\ker(h_2)$. 

(1) Suppose that $(W',F')$ is another oriented null-bordism of $(N,f,\sigma)$ in $M\times I$ and that the Poincar\'e dual class 
\[[L']=\mathrm{PD}(w_2(\nu_{F'},\sigma))\in H_1(W';\z{})\]    
is represented by a collection of disjoint circles embedded in the interior of $W'$. Form the closed oriented $3$-manifold
\[Y=W\cup_N (-W'),\quad \wtd{F}=F\cup F'\colon Y\to M\times I.\] 
Consider the following homomorphisms
\begin{equation}\label{eq:excision}
	\begin{aligned}
\epsilon_W&\colon H^2(W,N;\z{})\xra[\cong]{\mathrm{exc}_{W'}}H^2(Y,W';\z{})\xra{j_{W'}^\ast}H^2(Y;\z{}), \\
\epsilon_{W'}&\colon H^2(W',N;\z{})\xra[\cong]{\mathrm{exc}_{W}}H^2(Y,W;\z{})\xra{j_{W}^\ast}H^2(Y;\z{}), 
\end{aligned}
\end{equation}
where $\mathrm{exc}_{W'}$ and $\mathrm{exc}_{W}$ are the excision isomorphisms, and $j_{W}\colon Y\to (Y,W)$ and $j_{W'}\colon Y\to (Y,W')$ are the canonical inclusion maps.
Then the obstruction to extending the normal spin structure $\sigma$
over the glued normal bundle $\nu_{\widetilde F}$ is given by 
\[
w_2(\nu_{\widetilde F})=\epsilon_W(w_2(\nu_F,\sigma))+\epsilon_{W'}(w_2(\nu_{F'},\sigma)).
\]
Applying the Poincar\'e-Lefschetz duality, we have
\[
\PD_Y(w_2(\nu_{\widetilde F}))=j_\ast([L])-j'_\ast([L'])\in H_1(Y;\Z/2).
\]
Since every compact oriented $3$-manifold is spin and $M$ is spin, we have $w_2(\nu_{\wtd{F}})=0$, and hence
\[j_\ast([L])=j'_\ast([L'])\in H_1(Y;\Z/2),\]
where $j\colon W\to Y$ and $j'\colon W'\to Y$ are the canonical inclusion maps.	
Thus,
	\[F_\ast([L])=\widetilde{F}_\ast j_\ast([L])=\widetilde{F}_\ast j'_\ast([L'])=F'_\ast([L']),\]
which proves that $\Phi_0$ is independent of the choice of the oriented null-bordism $(W,F)$. 

(2) Suppose that $[N,f,\sigma]=[N',f',\sigma']\in\ker(h_2)$. 
There exists a $3$-dimensional normal spin bordism $(Z,\overline{F},\xi)$ in $M\times I$ such that
\[
\partial Z= N \sqcup (-N'),\quad 
\overline{F}|_{N\sqcup N'} = (f,0)\sqcup (f',1),\quad 
\xi|_{N\sqcup N'} = \sigma\sqcup \sigma'.
\]
In particular, via the spin bordism $(Z,\xi)$, we have 
\[[N,\sigma]=[N',\sigma'] \in\Omega_2^{\Spin}.\]
Let $(W,F)$ and $(W',F')$ be  oriented null-bordisms of $(N,f,\sigma)$ and $(N',f',\sigma')$, respectively, and let
\[[L]=\mathrm{PD}(w_2(\nu_F,\sigma)),\quad [L']=\mathrm{PD}(w_2(\nu_{F'},\sigma')).\]   
Form the closed oriented $3$-manifold 
\[X= W \cup_N Z\cup_{N'} (-W'),\quad \widehat{F}=F\cup\overline{F}\cup F'\colon X\to M\times I.\]
Since the normal spin structure $\xi$ on $\nu_{\ol{F}}$ already extends $\sigma\sqcup \sigma'$, the obstruction to extending $\sigma\sqcup \sigma'$ over the glued normal bundle $\nu_{\widehat{F}}$ is given by 
\[w_2(\nu_{\widehat{F}})=\epsilon_W(w_2(\nu_F,\sigma))+\epsilon_{W'}(w_2(\nu_{F'},\sigma')),\]
where $\epsilon_W$ and $\epsilon_{W'}$ are the homomorphisms as defined in (\ref{eq:excision}). Since $X$ and $M$ are spin, so is the normal bundle $\nu_{\widehat{F}}$. Applying the Poincar\'e-Lefschetz duality, we get 
\[0=j_\ast([L])-j'_\ast([L'])\in H_1(X;\Z/2),\]
where $j\colon W\to X$ and $j'\colon W'\to X$ are the canonical inclusion maps. Thus,
\[F_\ast([L])=\widehat F_\ast j_\ast([L])=\widehat F_\ast j'_\ast([L'])=F'_\ast([L'])\in H_1(X;\Z/2),\]
which shows that $\Phi_0$ is independent of the representatives of the bordism class $[N,f,\sigma]$. Therefore, the map $\Phi_0$ is well-defined, and the proof completes.
\end{proof}

We need the following two auxiliary lemmas.
\begin{lemma}\label{lem:circle-bundle}
Let $\Sigma$ be a compact connected surface with boundary $L$, let
$\bar g:\Sigma\to X$ be a map, and let $E\to \Sigma$ be a $2$-plane  bundle together with a boundary trivialization
\[
E|_L\cong \varepsilon^2_L=L\times \R^2.
\]
Let $\pi\colon P=S(E)\to \Sigma$ be the associated circle bundle and define $G=\bar g\circ \pi\colon P\to X$.
Then there hold
\[\partial P\cong L\times S^1,\quad G|_{\partial P}=g\circ \mathrm{pr}_1,\] 
where $\mathrm{pr}_1\colon L\times S^1\to L$ is the canonical projection, and $g=\bar g|_L$.
Moreover, there is a stable bundle isomorphism
\[TP\oplus \varepsilon^1\cong \pi^*(T\Sigma\oplus E).\]
\end{lemma}

\begin{proof}
 The total space of the unit circle bundle $P=S(E)\to \Sigma$ is a compact smooth $3$-manifold with boundary $\partial P=S(E|_L)$, where $E|_L$ is the restriction of $E$ to the boundary $L$. By the boundary trivialization of $E$, we have
\[
\partial P=S(E|_L)\cong S(L\times \R^2)\cong L\times S^1.
\]
Under this identification,  $\pi|_{\partial P}=\mathrm{pr}_1\colon L\times S^1\to L$ is the canonical projection, and hence
\[
G|_{\partial P}=(\overline g\circ \pi)|_{\partial P}=\overline g|_L\circ \mathrm{pr}_1=g\circ \mathrm{pr}_1.
\]

It remains to prove the stable bundle isomorphism. Let
\[
V:=\ker(d\pi)\subset TP
\]
be the vertical tangent line bundle of the circle bundle $\pi\colon P\to \Sigma$.
Then there is a bundle isomorphism
\[
TP\cong \pi^\ast(T\Sigma)\oplus V.
\]
Then by \cite[Fact 3.1]{CE03} or \cite[Corollary 21]{SS21}, we obtain the bundle isomorphisms
\[TP\oplus \varepsilon^1\cong\pi^\ast(T\Sigma)\oplus V\oplus \varepsilon^1
\cong\pi^\ast(T\Sigma)\oplus \pi^\ast E\cong\pi^\ast(T\Sigma\oplus E),\]
which completes the proof of the lemma. 
\end{proof}

\begin{lemma}\label{lem:spin-filling}
Let $g\colon L\to X$ be a smooth map from a closed smooth $1$-manifold $L$ to a smooth manifold $X$. Let $\sigma$ be a spin structure on $L$. Assume that
\[
[L,\sigma]=0\in \Omega_1^{\Spin}
\quad\text{and}\quad
g_\ast([L])=0\in H_1(X;\Z/2).
\]
Then the following hold:
\begin{enumerate}
	\itemsep=5pt
	\item There exist a connected compact smooth surface $\bar{g}\colon \Sigma \to X$, and a spin structure $\eta$ on $T\Sigma\oplus \det(T\Sigma)$ such that 
	  \[\partial \Sigma=L,\quad \bar{g}|_{\partial \Sigma}=g,\quad \eta|_{\partial \Sigma}=\sigma,\]
 where $\det(T\Sigma)$ is the determinant bundle of the tangent bundle $T\Sigma$.
	\item There exist a compact smooth $3$-manifold $G\colon P\to X$, and a spin structure $\bar{\eta}$ on $P$ such that 
	 \[\partial P\cong L\times S^1,\quad G|_{\partial P}=g\circ \mathrm{pr}_1,\quad \bar{\eta}|_{\partial P}\cong \sigma\times \sigma_{\Lie},\]
where $\sigma_{\Lie}$ is the Lie spin structure on $S^1$. 

Moreover, if the ambient manifold $X$ is spin, then there exists a normal spin structure on $\nu_G$ extending the induced normal spin structure of $g\circ \mathrm{pr}_1\colon L\times S^1\to X$.
\end{enumerate}

\end{lemma}

\begin{proof}
By Thom \cite{Thom54}, the condition $g_\ast([L])=0\in H_1(X;\Z/2)$ implying that there exist a connected compact smooth surface $\ol{g}\colon \Sigma\to X$ such that 
\[\partial \Sigma=L,\quad \overline{g}|_{\partial\Sigma}=g.\]

(1) Note that by definition, the Whitney sum $T\Sigma\oplus \det(T\Sigma)$ is a spin $3$-plane bundle.
Let $\eta_0$ be a spin structure on $T\Sigma\oplus \det(T\Sigma)$ and denote $\sigma_0=\eta_0|_{\partial \Sigma}$. Then there exists a unique $\alpha\in H^1(L;\z{})$ such that 
\[\sigma=\sigma_0+\alpha.\]
We may assume that $L=L_1\sqcup\cdots\sqcup L_b$.  Consider the exact sequence 
\[H^1(\Sigma;\Z/2)\to H^1(L;\Z/2)\xra{\delta}H^2(\Sigma,L;\Z/2)\cong\z{},\]
where $\delta$ sends $(x_1,\cdots,x_b)$ to the sum $x_1+\cdots+x_b$. 
Since 
\[ [L,\sigma]=[L,\sigma_0]=0 \in\Omega_1^{\Spin}\cong\z{},\]   
the difference class $\alpha$ must have even parity; or equivalently, $\delta(\alpha)=0$. Then by exactness, there is a class $\beta\in H^1(\Sigma;\z{})$ such that $\alpha=\beta|_L.$
Hence the spin structure $\eta=\eta_0+\beta$ on $T\Sigma\oplus \det(T\Sigma)$ extends $\sigma$ on $L$:
\[\eta|_{\partial \Sigma} =\sigma_0+\alpha=\sigma.\]

(2) Let $E=\det(T\Sigma)\oplus \varepsilon^1$ be the $2$-plane bundle over the surface $\Sigma$. Since $L=\partial\Sigma$, the line bundle $\det(T\Sigma)|_L$ is trivial. Choose the boundary
trivialization of $E|_L$ induced by a collar of $\partial\Sigma$ and the
standard trivialization of $\varepsilon^1$. Then by Lemma \ref{lem:circle-bundle}, there exist smooth maps 
\[
\pi\colon P=S(E)\to \Sigma,\quad G=\ol{g}\circ \pi \colon P\to X
\]
 such that $\partial P\cong L\times S^1$, $G|_{\partial P}=g\circ \mathrm{pr}_1$, and there hold bundle isomorphism
\[TP\oplus \varepsilon^1\cong\pi^\ast(T\Sigma\oplus E)\cong
\pi^\ast(T\Sigma\oplus \det(T\Sigma)\oplus \varepsilon^1).
\]
It follows that the spin structure $\eta$ in (1), together with the trivial spin structure on $\varepsilon^1$, induces a spin structure $\bar{\eta}$ on $TP$, which clearly extends the product spin structure $\sigma\times \sigma_{\Lie}$ on $\partial P=L\times S^1$. 

The ``moreover'' part follows from the bundle isomorphism 
\[TP\oplus \nu_G\cong G^\ast(TX),\]   
which restricted to the boundary $\partial P=L\times S^1$ gives $T(L\times S^1)\oplus \nu_h\cong h^\ast(TX)$ with $h=g\circ \mathrm{pr}_1$.
The proof is completed.
\end{proof}

\begin{proposition}\label{prop:Phi-spin-isom}
	The homomorphism $\Phi_0$ in (\ref{eq:Phi0-spin}) is a group isomorphism.
\end{proposition}

\begin{proof}
\underline{Surjectivity of $\Phi_0$}. Recall that the spin bordism group $\Omega_2^{\Spin}$ is generated by $[T^2,\sigma_1]$, where $\sigma_1$ is the normal Lie spin structure. Let $x_0$ be the base-point of $M$ and let $c_{x_0}\colon \{x_0\}\to M$ be the constant map.  Then \[\Phi_0([T^2,c_{x_0},\sigma_1])=([T^2,\sigma_1],0),\]
showing that the generator $([T^2,\sigma_1],0)$ of the summand $\Omega_2^{\Spin}$ is in the image of $\Phi_0$.
Let $a=[\gamma]\in H_1(M;\z{})$ be a generator represented by a simple closed curve $\gamma\colon S^1\to M$. Let $(T^2,\sigma_1)$ be as above. Define 
	\[f\colon T^2=S^1\times S^1\xra{\mathrm{pr}_1}S^1\xra{\gamma}M, \quad F\colon W=S^1\times D^2\xra{\mathrm{pr}_1}S^1\xra{\gamma}M\hookrightarrow M\times I.\]
	Then $F|_{\partial W}=(f,0)$ and $F|_{S^1\times \{\ast\}}=(\gamma,0)$: we view $\gamma$ as the image of the longitude $S^1\times \{\ast\}$ of the solid torus $W$. The obstruction class $w_2(\nu_F,\sigma_1)$ to extending the normal spin structure $\sigma_1$ over $W$ takes the value $1$ on the meridian disk $\{\ast\}\times D^2$ of $W$, and hence its Poincar\'e dual $[L]$ is represented by the longitude $S^1\times \{\ast\}$. It follows that $F_\ast([L])=[\gamma]=a$, and hence 
	\[\Phi_0([T^2,f,\sigma_1])=([T^2,\sigma_1],a).\]   
	After subtracting $\Phi_0([T^2,c_{x_0},\sigma_1])=([T^2,\sigma_1],0)$, we see that the direct summand $H_1(M;\z{})$ lies in the image of $\Phi_0$. Thus $\Phi_0$ is surjective.

\underline{Injectivity of $\Phi_0$}. Suppose that $[N,f,\sigma]\in \ker(h_2)$ satisfies
 \[\Phi_0([N,f,\sigma])=([N,\sigma],F_\ast([L]))=(0,0),\]  
where $[L]$ is the Poincar\'e dual of the obstruction class $w_2(\nu_F,\sigma)$ to extending $\sigma$ over some bounding oriented $3$-manifold $W$, and $F\colon W\to M\times I$ extends $(f,0)$.
The normal spin structure $\sigma$ of $N$ extends over $W$ away from the a disjoint union $L$ of circles representing the class $[L]$. 
Let
\[
V=W\setminus \mathrm{int}(\nu(L)),
\]
where $\nu(L)\cong L\times D^2$ is a small tubular neighborhood of $L$
in the interior of $W$. Then the normal spin structure $\sigma$ extends over $V$ to a normal spin structure $\tau$ on $\nu_{F|_V}$.

Because $M\times I$ is spin, the normal spin structure $\tau$ induces a
tangential spin structure $\bar{\tau}$ on $V$. Since $\partial\nu(L)\cong L\times S^1$, the restriction of $\bar{\tau}$ to $\partial\nu(L)$ is a spin structure on
\[
T(\partial\nu(L))\cong T(L\times S^1)\cong \mathrm{pr}_1^*TL\oplus \mathrm{pr}_2^*TS^1,
\] 
where $\mathrm{pr}_i$ is the projection of $L\times S^1$ onto the $i$-th factor, $i=1,2$.
Since the obstruction class evaluates nontrivially on
each meridian disk of $\nu(L)$, this spin structure restricts to the Lie
spin structure $\sigma_{\Lie}$ on each meridian circle. Hence there is a
unique spin structure $\tau_L$ on $L$ such that
\[
\bar{\tau}|_{\partial\nu(L)}\cong \tau_L\times \sigma_{\Lie}.
\]
Let $\bar{\sigma}$ be the tangent spin structure on $N$ induced by
$\sigma$ and the fixed spin structure on $M$. Then $(V,\bar{\tau})$ gives
a spin bordism from $(N,\bar{\sigma})$ to
$(L\times S^1,\tau_L\times\sigma_{\Lie})$. 
Under the isomorphism
\[
\Omega_1^{\Spin}\xrightarrow{\cong}\Omega_2^{\Spin},
\qquad
[L,\tau_L]\longmapsto [L\times S^1,\tau_L\times\sigma_{\Lie}],
\]
the condition $[N,\sigma]=0\in \Omega_2^{\Spin}$ implying that 
\[
[L,\tau_L]=0\in \Omega_1^{\Spin}.
\]

Since $F_\ast([L])=0\in H_1(M\times I;\Z/2)$, Lemma \ref{lem:spin-filling} gives a
compact smooth $3$-manifold $G:P\to M\times I$ and a normal spin
structure $\eta$ on $\nu_G$ such that
\[
\partial P\cong L\times S^1,\quad
G|_{\partial P}=F|_L\circ \mathrm{pr}_1,
\quad
\eta|_{\partial P}\cong \tau|_{\partial\nu(L)}.
\]
After a homotopy of $F$ supported in $\nu(L)$ and fixed on $N$, we may
assume
\[
F|_{\partial\nu(L)}=F|_L\circ \mathrm{pr}_1=G|_{\partial P}.
\]
Define
\[
W'=V\cup_{L\times S^1}P,\quad
F'=F|_V\cup G,
\]
and glue the normal spin structures $\tau$ and $\eta$ together to a normal spin structure $\tau'$ on $\nu_{F'}$. Then $(W',F',\tau')$ is a normal spin null-bordism of $(N,f,\sigma)$ in $M\times I$, $[N,f,\sigma]=0$, and therefore $\Phi_0$ is injective.
\end{proof}

\begin{theorem}
\label{thm:spin-bdsm}
	Let $M$ be a connected closed spin $(n+2)$-manifold with $n\geq 4$. Then there is a split SES of groups 
	\[0\to\Omega_2^{\Spin}\oplus  H_1(M;\z{})\to \Omega_2^{\Spin}(M)\xra{h_2} H_2(M;\Z)\to 0,\]
	where $h_2$ is the forgetful homomorphism sending a spin bordism class $[N,f,\sigma]$ to $f_\ast([N])$ with $[N]$ the fundamental class of $N$. 
	
	%Moreover, the subgroup $\Omega_2^{\Spin}$ is a direct summand of $\Omega_2^{\Spin}(M)$. 
\end{theorem}

\begin{proof}
	 Combine Lemmas \ref{lem:SO-H2}, \ref{lem:h2-splitting:spin}, \ref{lem:Phi-welldef} and Proposition \ref{prop:Phi-spin-isom}.
\end{proof}

Combining Theorem \ref{thm:spin-bdsm} and Lemma \ref{lem:fr-G}, we immediately get the following geometric interpretation of Theorem \ref{thm:cohtp=-2:mfds} via the Pontryagin-Thom isomorphism for cohomotopy.

\begin{theorem}\label{thm:frbdsm-spin}
	Let $M$ be a connected closed spin $(n+2)$-manifold with $n\geq 4$. Then there is a split SES of groups 
	\[0\to \Omega_2^{\fr}\oplus H_1(M;\z{})\to \Omega_2^{\fr}(M)\xra{h_2} H_2(M;\Z)\to 0,\]
	where $h_2$ is the forgetful homomorphism sending a normally framed bordism class $[N,f,\varphi]$ to $f_\ast([N])$ with $[N]$ the fundamental class of $N$. 

\end{theorem}

\begin{remark}\label{rm:retraction}
Results in this subsection indeed show that there is a canonical group isomorphism 
	\[(h_2,r,c_\ast)\colon \Omega_2^{\Spin}(M)\xra{\cong} H_2(M;\Z)\oplus H_1(M;\z{})\oplus \Omega_2^{\Spin},\]
 where $r\colon\Omega_2^{\Spin}(M)\to H_1(M;\z{})$ is defined by
	\[r(x)=\mathrm{pr}_2\Phi_0(x-s_0h_2(x)),\]
with $s_0$ the canonical section of $h_2$ in Lemma \ref{lem:h2-splitting:spin}, and $\mathrm{pr}_2$ the projection onto the second factor of $\Phi_0$ in (\ref{eq:Phi0-spin}). Write $x=[N,f,\sigma]$ and let $\sigma_0$ be the canonical spin structure on the closed oriented $2$-manifold $(N,f)$ in $M$, then there exists a unique class $\alpha\in H^1(N;\z{})$ such that $\sigma=\sigma_0+\alpha$. It follows that 
\begin{equation}\label{eq:retraction}
	r([N,f,\sigma])=f_\ast(\mathrm{PD}_N(\alpha)).
\end{equation}

\end{remark}

\begin{remark}\label{rmk:BM}
Brumfiel and Morgan \cite{BM16,BM18} also studied the reduced spin bordism groups $\wtd{\Omega}_k^{\Spin}(M)$ from the perspective of the Pontryagin duals $G^k(M)=\Hom(\wtd{\Omega}_k^{\Spin}(M),\R/\Z)$ for $k\leq 4$. They proved that for any CW complex $M$, there is a split  SES of groups 
\[0\to H^2(M;\R/\Z)\to G^2(M)\to H^1(M;\z{})\to 0.\]
Applying the Pontryagin dualization again, 
%and using the isomorphisms 
%\[\Hom(\R/\Z,\R/\Z)\cong \Z,\quad \Hom(\Z/m,\R/\Z)\cong \Z/m,~ \forall~m\geq 1,\]
we get the group isomorphisms
\[\Omega_2^{\Spin}(M)\cong \Omega_2^{\Spin}\oplus \wtd{\Omega}_2^{\Spin}(M)\cong  \Omega_2^{\Spin}\oplus H_2(M;\Z)\oplus H_1(M;\z{}).\]

Although both approaches lead to group decompositions of $\Omega^{\Spin}_2(M)$ for smooth oriented manifolds, the results of the present paper are of a different nature: the Brumfiel-Morgan decomposition arises from a choice of splitting on the dual side, which is not canonical, whereas Remark \ref{rm:retraction} shows that the group isomorphism for $\Omega_2^{\Spin}(M)$ in the spin case is canonical.
\end{remark}

\subsection{The nonspin case}\label{sec:nonspin-bdsm}
Recall that the oriented $(n+2)$-manifold $M$ is nonspin if and only if its second Stiefel-Whitney class $w_2(M)$ is nonzero, or equivalently, the evaluation pairing 
\[\lra{w_2(M),-}_\Z\colon H_2(M;\z{})\to \Z/2\]
is surjective. 
First, we study the image of the forgetful homomorphism
\begin{equation}\label{eq:h_2-nonspin}
	h_2\colon \Omega_2^{\fr}(M)\to \Omega_2^{\SO}(M)\xra{\cong}H_2(M;\Z),
\end{equation}
where the isomorphism is given by Lemma \ref{lem:SO-H2}.

\begin{lemma}\label{lem:nonspin-im-h2}
The homomorphism $h_2$ in (\ref{eq:h_2-nonspin}) has image 
	\[\mathrm{im}(h_2)=\ker\big(\lra{w_2(M),-}_\Z\big)\subseteq H_2(M; \Z).\]
\begin{proof}
	Let $(N,f)$ be a $2$-dimensional oriented bordism in $M$. Then the normal bundle $\nu_f$ is an oriented bundle of rank $n$. The obstruction to finding a trivialization of $\nu_f$ is given by $w_2(\nu_f)\in H^2(N;\z{})$.  Since $N$ is spin, we have $w_2(N)=0$, and hence $w_2(\nu_f)=f^\ast(w_2(M))$. It follows that $\nu_f$ admits a trivialization if and only if $f^\ast(w_2(M))=0$. Evaluating on the fundamental class $[N]$,  this is equivalent to the pairing 
	\[\lra{w_2(M),f_\ast([N])}_\Z=\lra{f^\ast(w_2(M)),[N]}_\Z=0.\]
	The proof of the lemma is completed.
\end{proof}
\end{lemma}

Next, we study the kernel $\ker(h_2)$ of the forgetful homomorphism $h_2$ in (\ref{eq:h_2-nonspin}). 
 Let $[N,f,\varphi]\in \ker(h_2)$ be represented by a closed smooth  $2$-manifold $f\colon N\to M$ with a normal framing $\varphi$. There exists an oriented bordism $(W,F)$ in $M\times I$ such that $\partial W=N$, $F|_{\partial W}=(f,0)$ and the orientation of $W$ restricted to the orientation on $N$ induced by the normal framing $\varphi$.
Let $w_2(\nu_{F},\varphi)\in H^2(W,N;\z{})$ be the obstruction class to extending the normal framing $\varphi$ over $W$ and denote
\[[L]=\mathrm{PD}(w_2(\nu_{F},\varphi))\in H_1(W;\z{}).\]

If $(W',F')$ is another oriented null-bordism of $(N,f,\varphi)$ and 
\[[L']=\mathrm{PD}(w_2(\nu_{F'},\varphi))\in H_1(W';\z{}).\] 
%is the Poincar\'e dual of the obstruction class $w_2(\nu_{W'},\varphi)$ to extending $\varphi$ over $W'$. 
Then $F$ and $F'$ coincide in the common boundary $N$. Define
\[Y=W\cup_N (-W'),\quad \wtd{F}=F\cup F'\colon Y\to M\times I.\]
Since the boundary framing $\varphi$ agrees on
the common boundary $N$, the obstruction to extending the normal framing
over the glued normal bundle $\nu_{\widetilde F}$ is
\[
w_2(\nu_{\widetilde F})=\epsilon_W(w_2(\nu_F,\varphi))+\epsilon_{W'}(w_2(\nu_{F'},\varphi)),
\]
where $\epsilon_W$ and $\epsilon_{W'}$ are the homomorphisms defined in (\ref{eq:excision}).
Applying the Poincar\'e-Lefschetz duality yields the equality
\[
\PD_Y(w_2(\nu_{\widetilde F}))=j_\ast([L])-j'_\ast([L'])\in H_1(Y;\Z/2).
\]

Since $Y$ is a closed oriented $3$-manifold, we have $w_2(TY)=0$ and
\[
w_2(\nu_{\widetilde F})=\widetilde{F}^\ast(w_2(M\times I)).
\]
Hence,
\[\mathrm{PD}(w_2(\nu_{\wtd{F}}))=\widetilde{F}^\ast(w_2(M\times I))\smallfrown_\Z [Y]=j_\ast([L])-j'_\ast([L']).\] 
Identifying $w_2(M)$ with $w_2(M\times I)$, we then obtain 
\begin{align*}
	F_\ast([L])-F'_\ast([L'])&=\widetilde{F}_\ast j_\ast([L])-\widetilde{F}_\ast j'_\ast([L'])\\
	&=\wtd{F}_\ast\big(\widetilde{F}^\ast(w_2(M))\smallfrown_\Z [Y]\big)\\
	&=w_2(M)\smallfrown_\Z \widetilde{F}_\ast([Y]),
\end{align*}
where the last equality holds by the naturality of the cap product:
\[ w_2(M)\smallfrown_\Z -\colon H_3(M;\Z)\xra{\rho_2}H_3(M;\z{})\xra{w_2(M)\smallfrown -} H_1(M;\z{}).\] 
Thus the map
	\begin{equation}\label{eq:Phi1-nonspin}
		\begin{aligned}
		\Phi_1\colon \ker(h_2)\to \tfrac{H_1(M;\z{})}{w_2(M)\smallfrown_\Z H_3(M;\Z)},\quad 
		[N,f,\varphi]\mapsto [F_\ast([L])]_{w_2(M)\smallfrown_\Z}
	\end{aligned}	
	\end{equation}
is independent of the choices of oriented null-bordisms of $(N,f,\varphi)$,	where $[F_\ast([L])]_{w_2(M)\smallfrown_\Z}$ denotes the coset of $F_\ast([L])$ in the quotient group. By similar arguments to that of Lemma \ref{lem:Phi-welldef}, the map $\Phi_1$ is independent of the choices of the representatives of the bordism class $[N,f,\sigma]\in\ker(h_2)$, and hence is a well-defined homomorphism.

\begin{lemma}\label{lem:normal-framing-filling}
	Let $(L,g,\phi_L)$ be a closed $1$-manifold in $M\times I$ with a normal structure $\tau|_L$ on $\nu_g$.  
	Let $\bar{\phi}_L$ be the framing of the normal bundle of the map 
	\[g\circ \mathrm{pr}_1\colon L\times S^1\to M\times I\]
	induced from $\phi_L$ and the Lie framing of $TS^1$; equivalently, $\bar{\phi}_L$ is the framing corresponding to the stable isomorphism
	\[\nu_{g\circ \mathrm{pr}_1}\oplus \varepsilon^1 \cong \mathrm{pr}_1^\ast\nu_g \oplus TS^1.\]
	Assume that $g_\ast([L])=0\in H_1(M\times I;\z{})$. Then the following hold:
	\begin{enumerate}
		\itemsep=5pt
		\item There exists a $2$-plane bundle $E\to \Sigma$ such that  
		 \[w_1(E)=w_1(\Sigma),\quad w_2(E)=\bar{g}^\ast(w_2(M\times I)),\quad E|_L\cong \varepsilon^2_L.\]
	where $\Sigma$ is a compact connected surface with boundary $\partial\Sigma=L$,   and $\bar{g}\colon \Sigma\to M\times I$ extends $g$.                                                        
		\item There exists a compact smooth oriented $3$-manifold $G\colon P\to M\times I$, and a framing $\varphi$ of $\nu_G$ such that 
	\[\partial P\cong L\times S^1,\quad
G|_{\partial P}=g\circ \mathrm{pr}_1,\quad \varphi|_{\partial P}=\bar{\phi}_L.\]

	\end{enumerate}
	
\end{lemma}

\begin{proof}
(1) Since $g_\ast([L])=0\in H_1(M\times I;\z{})$, there exist a compact connected smooth surface $\Sigma$ with boundary $\partial \Sigma=L$ and a map $\bar{g}\colon \Sigma\to M\times I$ extending $g$. The given framing $\phi_L$ of $\nu_g$, together with the Lie framing of $TL$, gives a trivialization of the pullback bundle $g^\ast T(M\times I)$ over $L$. 

We construct the required $2$-plane bundle $E\to \Sigma$ with a chosen boundary trivialization. Choose a boundary trivialization
\[
\tau_\partial\colon E|_L\xrightarrow{\cong}\varepsilon^2_L
\]
so that the induced boundary trivialization of
\[
T\Sigma|_L\oplus E|_L\oplus \varepsilon^{n-1}
\]
is compatible with the above trivialization of $g^\ast T(M\times I)$.
First choose $E\to \Sigma$ with $w_1(E)=w_1(\Sigma)$.
%Since $H^3(\RP{\infty};\Z)=0$, the fibration 
%	\[K(\Z,2)\simeq B\SO_2\to B\mathrm{O}_2\xra{w_1=B\det}K(\z{},1)\] 
%implies a homotopy equivalence $B\mathrm{O}_2\simeq K(\z{},1)\times K(\Z,2)$. 
Note that a $2$-plane bundle over the surface $\Sigma$ is classified by its first Stiefel-Whitney class and the twisted Euler class
\[\wtd{e}(E,\tau_{\partial})\in H^2(\Sigma,L;\Z_{w_1(E)}),\]
where the local coefficient system $\Z_{w_1(E)}$ is determined by the first Stiefel-Whitney class $w_1(E)$ of $E$.
Since $\Sigma$ is connected, Poincar\'e-Lefschetz duality with local
coefficients gives 
\[
H^2(\Sigma,L;\Z_{w_1(E)})\cong \Z,
\]
and the reduction map
\[
\rho_2:H^2(\Sigma,L;\Z_{w_1(E)})
\longrightarrow H^2(\Sigma,L;\Z/2)
\]
is surjective.   Since $g^\ast(w_2(M\times I))|_L=0$, we regard $g^\ast(w_2(M\times I))$ as a class in $H^2(\Sigma,L;\Z/2)$.  Choose a lift of it in $H^2(\Sigma,L;\Z_{w_1(E)})$ and let $E\colon \Sigma$ be the corresponding $2$-plane bundle. Then $E$ satisfies the properties in the statement.

(2) Define
\[\pi\colon P=S(E)\to \Sigma,\quad G=\bar{g}\circ \pi\colon P\to M\times I.\]  
By Lemma \ref{lem:circle-bundle}, we have 
 \begin{align*}
	\partial P=S(E|_L)\cong L\times S^1,\quad G|_{\partial P}=g\circ \mathrm{pr}_1,\quad TP\oplus \varepsilon^1\cong \pi^\ast(T\Sigma\oplus E).
 \end{align*}
 We claim that there is a bundle isomorphism 
\begin{equation}\label{eq:stable-iso-aux2}
	T\Sigma\oplus E\oplus \varepsilon^{n-1}\cong \bar{g}^\ast T(M\times I)
\end{equation}
compatible with the chosen boundary trivializations.
 First, we have 
 \[w_1(T\Sigma\oplus E)=w_1(\Sigma)+w_1(E)=0=w_1(\bar{g}^\ast T(M\times I)),\]so the bundle $T\Sigma\oplus E\oplus \varepsilon^{n-1}$ is oriented. Recall that oriented $2$-plane bundles are classified by the second Stiefel-Whitney classes. By the properties of $E$ in (1), we compute that
 \begin{align*}
w_2(T\Sigma\oplus E\oplus \varepsilon^{n-1})
&=w_2(T\Sigma)+w_1(T\Sigma)w_1(E)+w_2(E) \\
&=w_1(\Sigma)^2+w_1(\Sigma)^2+\bar{g}^\ast(w_2(M\times I)) \\
&=w_2(\bar{g}^\ast T(M\times I)).
\end{align*}
Hence the claim is proved.
Pulling the isomorphism (\ref{eq:stable-iso-aux2}) back by $\pi$, we obtain 
\[TP\oplus \varepsilon^n \cong \pi^\ast(T\Sigma\oplus E\oplus \varepsilon^{n-1})
\cong G^\ast T(M\times I)\cong TP\oplus\nu_G.\]
It follows that $\nu_G$ is stably trivial. Since $\mathrm{rank}(\nu_G)=n>\dim(P)=3$, the normal bundle $\nu_G$ is actually trivial, and hence admits admits a framing, say $\varphi$. By construction, we can require $\varphi|_{\partial P}$ coincides with the normal framing $\bar{\phi}_L$ on $\partial P=L\times S^1$. The proof is completed.
\end{proof}

\begin{proposition}\label{prop:nonspin-kernel}
	The homomorphism $\Phi_1$ in (\ref{eq:Phi1-nonspin}) is a group isomorphism.
\end{proposition}

\begin{proof}
	\underline{Surjectivity of $\Phi_1$}: Let $\bar{a}\in \tfrac{H_1(M;\Z/2)}{w_2(M)\smallfrown_Z H_3(M;\Z)}$ and choose a representative $a=[\gamma]\in H_1(M;\z{})$ represented by an embedded circle $\gamma\colon S^1\to M$. Define
	\[f\colon T^2=S^1\times S^1\xra{\mathrm{pr}_1}S^1\xra{\gamma}M, \quad F\colon S^1\times D^2\xra{\mathrm{pr}_1}S^1\xra{\gamma}M\hookrightarrow M\times I.\]
	Equip $T^2=\partial(S^1\times D^2)$ with the normal framing $\phi_1$ induced from the product framing of the solid torus $S^1\times D^2$. Then the obstruction class $w_2(\nu_F,\phi_1)$ evaluates nontrivially on the meridian disk $\{\ast\}\times D^2$, so its Poincar\'e dual is represented by the longitude $S^1\times\{\ast\}$. Hence 
	\[\Phi_1([T^2,f,\phi_1])=\bar{a},\]
    and $\Phi_1$ is surjective.
	
	\underline{Injectivity of $\Phi_1$}. Suppose that $[N,f,\phi]\in \ker(h_2)$ satisfies 
	\[\Phi_1([N,f,\phi])=0\in\tfrac{H_1(M;\Z/2)}{w_2(M)\smallfrown_Z H_3(M;\Z)}.\]
	Choose an oriented null-bordism $(W,F)$ of $[N,f]$ in $M\times I$ and let 
	\[[L]=\PD(w_2(\nu_F,\phi))\in H_1(W;\z{}).\]
Since $[F_\ast([L])]_{w_2(M)\smallfrown_\Z}=0$, Thom's realization theorem (see \cite{Thom54} or \cite[Theorem 7.37]{Rudyakbook}) implies that there exist a closed oriented smooth $3$-manifold $Z$ and a map $G\colon Z\to M\times I$ such that
		\begin{equation*}\label{eq:FL}
			F_\ast([L])=w_2(M)\smallfrown_\Z G_\ast([Z]).
		\end{equation*}
Let 
		\[W'=W\sqcup Z,\quad F'=F\sqcup G\colon W'\to M\times I.\]
	Then $\partial W'=N$, $F'|_N=(f,0)$, and the obstruction to extending the normal framing $\varphi$ of $N$ over $W'$ is the sum 
		\[w_2(\nu_{F'},\varphi)=w_2(\nu_F,\varphi)+w_2(\nu_G).\]
	Since $Z$ is spin,  we have $w_2(\nu_G)=G^\ast(w_2(M))$, whose Poincar\'e dual is 
		\[[L_Z]=\mathrm{PD}(w_2(\nu_G))=G^\ast(w_2(M))\smallfrown_\Z [Z] \in H_1(Z;\z{}).\]
	It follows that
		\begin{equation*}\label{eq:GK}
		G_\ast([L_Z])=G_\ast\big(G^\ast(w_2(M))\smallfrown_\Z [Z]\big)=w_2(M)\smallfrown_\Z G_\ast([Z]).
		\end{equation*}
	Let $[L']=\mathrm{PD}(w_2(\nu_{F'},\varphi))$ be represented by a disjoint circle embedded in the interior of $W'$, then we get  
		\[F'_\ast([L'])=F_\ast([L])+G_\ast([L_Z])=0\in H_1(M\times I;\z{}).\]

Let 
\[V'=W'\setminus \mathrm{int}(\nu(L')).\]   
Then $\partial V'=N\sqcup L'\times S^1$, and the normal framing $\phi$ on $N$ extends to a normal framing $\phi_{V'}$ on $\nu_{F'|_{V'}}$. Denote by $\phi_{L'}$ the framing of $\nu_{F'|_L}$ induced by $\phi_V$. Denote by $\bar{\phi}_{L'}=\phi_{V'}|_{\partial \nu(L')}$ the induced framing on $\partial \nu(L')=L'\times S^1$.
Applying Lemma \ref{lem:normal-framing-filling} to the framed $1$-manifold $(L',F'|_{L'},\phi_{L'})$, there exist a compact smooth oriented $3$-manifold $H\colon P\to M\times I$ and a framing $\varphi$ on $\nu_H$ such that 
\[\partial P=L'\times S^1,\quad H|_{\partial P}=F'|_{L'}\circ \mathrm{pr}_1,\quad \varphi|_{\partial P}=\bar{\phi}_{L'}.\]
After a homotopy of $F'$ supported in $\nu(L')$, we may assume 
\[F'|_{\partial\nu(L')}=F'|_{L'}\circ \mathrm{pr}_1,\]   
so that $F'|_{V'}$ and $H$ coincides on the common boundary $L'\times S^1$, and that the normal framings $\phi_{V'}$ and $\varphi$ also coincide there. Define
\[W''=V'\cup_{L'\times S^1}P,\quad F''=F'\cup H,\quad \phi''=\phi_{V'}\cup \varphi.\]
Then $\partial W''=N$, $F''|_{\partial W''}=(f,0)$, and hence $(W'',F'',\phi'')$ is a normally framed null-bordism of $(N,f,\phi)$. Thus $[N,f,\phi]=0\in\ker(h_2)$,  the homomorphism $\Phi_1$ in (\ref{eq:Phi1-nonspin}) is injective. The proof of the Proposition is completed.
\end{proof}

\begin{theorem}\label{thm:frbdsm-nonspin}
	Let $M$ be a closed oriented nonspin $(n+2)$-manifold with $n\geq 4$. Then there is a SES of groups
	\begin{equation}\label{SES:frbdsm-nonspin}
		0\to \tfrac{H_1(M;\z{})}{w_2(M)\smallfrown_\Z H_3(M;\Z)}\to \Omega_2^{\fr}(M)\xra{h_2} \ker\big(\lra{w_2(M),-}_\Z\big)\to 0,
	\end{equation}
	which splits if and only if 
	\[w_2(M)\smallfrown \delta^{-1}\big({}_2\ker\left(\lra{w_2(M),-}_\Z\right)\big)\subseteq w_2(M)\smallfrown_\Z H_3(M;\Z).\]
\end{theorem}

\begin{proof}
The SES follows by Lemma \ref{lem:nonspin-im-h2} and Proposition \ref{prop:nonspin-kernel}.
It suffices to prove the splitting criterion. First,  the quotient group
\[Q:=\tfrac{H_1(M;\z{})}{w_2(M)\smallfrown_\Z H_3(M;\Z)}\]
is a elementary $2$-torsion group implying that for any $x\in {}_2\ker\big(\lra{w_2(M),-}_\Z\big)$ and $\alpha\in h_2^{-1}(x)$, the extension obstruction class is given by $2\alpha\in Q$ and the SES (\ref{SES:frbdsm-nonspin}) splits if and only if $2\alpha=0$ for all $x\in {}_2\ker\big(\lra{w_2(M),-}_\Z\big)$.

Write $\alpha=[N,f,\varphi]\in \Omega_2^{\fr}(M)$ with $x=f_\ast([N])$ a elementary $2$-torsion class. Then there exists some $y\in H_3(M;\z{})$ satisfying $x=\delta(y)$, where $\delta\colon H_3(M;\Z/2)\to H_2(M;\Z)$ is the Bockstein homomorphism. 
Let $F\colon W\to M\times I$ be an embedding of a smooth oriented compact $3$-manifold such that 
\[\partial W=N\sqcup N,\quad F|_{\partial W}=(f,0)\sqcup (f,1),~~ \text{ and } y=F_\ast([W]_2),\] 
where $[W]_2\in H_3(W;\z{})$ is a fundamental class that is not a mod-$2$ reduction of any integral class. Then the class 
\[F_\ast\big(\mathrm{PD}(w_2(\nu_F,\varphi\sqcup\varphi))\big)\in H_1(M;\z{})\]
exactly represents $2\alpha\in Q$.
%The obstruction to extending the normal framing of $\partial W$ over $W$ is given by the relative Stiefel-Whitney class $w_2(\nu_F,\varphi\sqcup\varphi)$, whose Poincar\'e dual $\mathrm{PD}(w_2(\nu_F,\varphi\sqcup\varphi))$ into $H_1(M;\z{})$ exactly represents $2\alpha\in Q$.

Let $j\colon W\to (W,\partial W)$ be the canonical inclusion. Then we have 
\[j_\ast([W]_2)=[W,\partial W]_2, \quad j^\ast\big(w_2(\nu_F;\varphi\sqcup\varphi)\big)=w_2(\nu_F), \] 
where the first equality holds because $[\partial W]_2=0\in H_2(\partial W;\z{})$. 
Using the naturality of cap products, we compute that
\begin{align*}
	w_2(\nu_F,\varphi\sqcup\varphi)\smallfrown [W,\partial W]_2&=w_2(\nu_F,\varphi\sqcup\varphi)\smallfrown j_\ast([W]_2)\\
	&=j^\ast(w_2(\nu_F,\varphi\sqcup\varphi))\smallfrown [W]_2\\
	&=w_2(\nu_F)\smallfrown [W]_2\\
	&=F^\ast(w_2(M\times I))\smallfrown [W]_2,
\end{align*}
where the last equality holds for $w_2(TW)=0$. 
Consequently, we get
\begin{align*}
	F_\ast\big(\mathrm{PD}(w_2(\nu_F,\varphi\sqcup\varphi))\big)&=F_\ast\big(w_2(\nu_F,\varphi\sqcup\varphi)\smallfrown [W,\partial W]_2\big)\\
	&=F_\ast(F^\ast(w_2(M\times I))\smallfrown [W]_2)\\
	&=w_2(M\times I)\smallfrown F_\ast([W]_2)\\
	&=w_2(M\times I)\smallfrown y.
\end{align*}
Therefore in the quotient group $Q$, we obtain the formula 
\[2\alpha=[w_2(M)\smallfrown y]_{w_2(M)\smallfrown_\Z},\]
which proves the splitting criterion in the statement.
\end{proof}

\begin{proof}[Geometric Proof of Theorem \ref{mainthm:frbdsm-dim2}]
	Combine Theorems \ref{thm:frbdsm-spin} and \ref{thm:frbdsm-nonspin}.
\end{proof}

The following lemma is a corollary of Lemmas \ref{lem:normal-framing-filling} and Lemma \ref{lem:spinnable-trivial}.

\begin{lemma}\label{lem:normal-spin-filling}
	Let $(L,g,\sigma_L)$ be a closed $1$-manifold in $M\times I$ with a normal structure $\sigma_L$ on $\nu_g$. Denote by $\bar{\sigma}_L$ the spin structure of the normal bundle of the map 
	\[g\circ \mathrm{pr}_1\colon L\times S^1\to M\times I\]
	induced from $\sigma_L$ and the Lie framing of $TS^1$. If $g_\ast([L])=0$,
then there exist a compact smooth oriented $3$-manifold $G\colon P\to M\times I$, and a normal spin structure $\eta$ on $\nu_G$ such that
\[
\partial P\cong L\times S^1,\quad
G|_{\partial P}=g\circ \mathrm{pr}_1,\quad \eta|_{\partial P}\cong \bar{\sigma}_L.
\]

\end{lemma}

\begin{proof}
Since $L$ has dimension $1$ and $\mathrm{rank}(\nu_g)\geq 3$,
Lemma \ref{lem:spinnable-trivial} implies that the normal spin structure
$\sigma_L$ determines a compatible normal framing $\varphi_L$ of
$\nu_g$. Let $\overline{\phi}_L$ be the induced framing of the normal
bundle of $g\circ \mathrm{pr}_1$. By construction, the spin structure
induced by $\overline{\phi}_L$ is precisely $\bar{\sigma}_L$.
Applying Lemma \ref{lem:normal-framing-filling} to the normally framed
$1$-manifold $(L,g,\phi_L)$, we obtain a compact smooth oriented
$3$-manifold $G\colon P\to M\times I$, and a framing $\varphi$
of $\nu_G$ such that
\[
\partial P\cong L\times S^1,\quad
G|_{\partial P}=g\circ \mathrm{pr}_1,
\quad
\varphi|_{\partial P}=\overline{\phi}_L.
\]
The framing $\varphi$ induces a normal spin structure $\eta$ on $\nu_G$.
Since $\varphi|_{\partial P}=\overline{\phi}_L$ and the spin structure induced by $\overline{\phi}_L$ is $\bar{\sigma}_L$, we get $\eta|_{\partial P}\cong \bar{\sigma}_L$. The proof is completed.
\end{proof}

Utilizing Lemma \ref{lem:normal-spin-filling} and similar arguments to the proof of Theorems \ref{thm:spin-bdsm} and \ref{thm:frbdsm-nonspin}, we obtain the following more general result for the spin bordism groups $\Omega_2^{\fr,\Spin}(M)$, which recovers Theorem \ref{thm:spin-bdsm} by taking $w_2(M)=0$.

\begin{theorem}\label{thm:spinbdsm-dim2}
	Let $M$ be a closed smooth oriented (nonspin) $(n+2)$-manifold with $n\geq 4$. There is a SES of groups
	\[0\to \Omega_2^{\Spin}\oplus \tfrac{H_1(M;\z{})}{w_2(M)\smallfrown_\Z H_3(M;\Z)}\to \Omega_2^{\fr,\Spin}(M)\xra{h_2} \ker\big(\lra{w_2(M),-}_\Z\big)\to 0,\]
   which splits if and only if 
   \[w_2(M)\smallfrown \delta^{-1}\big(\ker(\lra{w_2(M),-}_\Z)\big)\subseteq w_2(M)\smallfrown_\Z H_3(M;\Z).\]
\end{theorem}

Recall that there is a canonical homomorphism $\phi_2\colon \Omega_2^{\fr}(M)\to \Omega_2^{\fr,\Spin}(M)$. Comparing Theorems \ref{thm:frbdsm-spin}, \ref{thm:frbdsm-nonspin} and \ref{thm:spinbdsm-dim2} and applying the snake lemma, we get the following corollary.
\begin{corollary}\label{cor:nonspin-fr-to-spin}
	Let $M$ be a closed smooth oriented $(n+2)$-manifold with $n\geq 4$. There is a split SES of groups
	\[0\to \Omega_2^{\fr}(M)\xra{\phi_2} \Omega_2^{\fr,\Spin}(M)\to \Omega_2^{\Spin}[\varepsilon]\to 0,\]
	where $\Omega_2^{\Spin}[\varepsilon]=\Omega_2^{\Spin}$ if $M$ is nonspin ($\varepsilon=1$), and $\Omega_2^{\Spin}[\varepsilon]=0$ otherwise.
\end{corollary}

%As a direct corollary, we have the following  description of the natural homomorphism $\phi_2\colon\Omega_2^{\fr}(M)\to \Omega_2^{\Spin}(M)$.

%\begin{corollary}
%\label{cor:nonspin-fr-to-spin}
%	Let $M$ be a closed smooth oriented $(n+2)$-manifold with $n\geq 4$. There is a split SES of groups
%	\[0\to \Omega_2^{\fr}(M)\xra{\phi_2} \Omega_2^{\Spin}(M)\to \Omega_2^{\Spin}[\varepsilon]\to 0,\]
%	where $\Omega_2^{\Spin}[\varepsilon]=\Omega_2^{\Spin}$ if $M$ is nonspin ($\varepsilon=1$), and $\Omega_2^{\Spin}[\varepsilon]=0$ otherwise.
%\end{corollary}

%Note that the SESs in theorems of this section are clearly split if $H_2(M;\Z)$ is elementary $2$-torsion-free.

%\begin{example}
%	If $M=\CP{2}\times \R P^3$, then the cap product and the pairing with $w_2(M)$ are both trivial. The SES of groups
%	\[0\to \z{}\to \Omega_2^{\fr}(M)\to \Z\to 0\]
%	splits to yield an isomorphism $\Omega_2^{\fr}(M)\cong\Z\oplus\z{}$.
%\end{example}

\section{Application to vector bundles}\label{sec:appl-VB}

Throughout this section, let $0\leq k\leq 3$, unless otherwise specified, we assume that $G=G(k)$ and $H=H(k)$ are the groups in Table \ref{table:GH-k}, that is
\begin{equation*}\label{eq:G_k}
	G(k)=\left\{\begin{array}{ll}
		\Spin&\text{ for }k=1,2;\\
		\String&\text{ for }k=3.
	\end{array}\right. \quad H(k)=\left\{\begin{array}{ll}
		\SO&\text{ for }k=1,2;\\
		\Spin&\text{ for }k=3.
	\end{array}\right. 
\end{equation*}
We also set $G(0)=\SO$.
Let $M$ be a closed smooth $G$-manifold of dimension $n+k$ with $n\geq k+2$, and let $\pi\colon E\to M$ be an oriented vector bundle of rank $n$ such that the classifying map $g\colon M\to B\SO_n$ admits a lift to $BG_n$; alternatively, we simply say that $\pi\colon E\to M$ is a $G$-vector bundle or is classified by a map $g\colon M\to BG_n$.

When $G=\SO$, recall that the universal vector bundle $\gamma^n_{\SO}$ over $B\SO_n$ possesses the universal Thom class $u=u_{\SO}\in \widetilde{H}^n(\MSO_n;\Z)$ and the universal Euler class 
\[e_{\SO}=\mathfrak{z}_{\SO}^\ast(u)\in H^n(B\SO_n;\Z),\] 
where $\mathfrak{z}_{\SO}\colon B\SO_n\to \MSO_n$ is the normalized zero section of the universal $\SO_n$-bundle over $B\SO_n$.  
For the vector bundle $E\to M$ classified by a map $g\colon M\to B\SO_n$, the ordinary Thom class $u(E)$ and the ordinary Euler class $e(E)$ for the oriented vector bundle $E\to M$ are defined by 
\begin{equation}\label{eq:Thom-Eulerclss}
	\begin{aligned}
		u(E)&=(Tg)^\ast(u)\in \widetilde{H}^n(\mathrm{T}(E);\Z),\\
		 e(E)&=\mathrm{pr}^\ast\mathfrak{z}_E^{\ast}(u(E))=g^\ast(e_{\SO})\in  H^n(M;\Z),
	\end{aligned}
\end{equation}
where $\mathrm{pr}\colon M_+\to M$ is the map such that $\mathrm{pr}|_M=\id_M$.% and $\mathrm{pr}^\ast\colon \widetilde{H}^n(M;\Z)\to \widetilde{H}^n(M_+;\Z)\cong H^n(M;\Z)$.                           
                     
Note that the ordinary Thom class and the ordinary Euler class are the ones with respect to the Eilenberg-MacLane spectrum $\mathbf{H}\Z$. There are also Thom class and Euler class with respect to the spectrum $\mathbf{ko}$ for the connective real $K$-theory and the spectrum $\mathbf{tmf}$ for the connective topological modular forms.
In \cite{ABS64}, Atiyah-Bott-Shapiro constructed the \emph{universal spin orientation} (also called the \emph{universal spin Thom class})
	\begin{align*}
		&u_{\Spin}\colon \mathbf{MSpin}\to \mathbf{ko},
		\text{or }u_{\Spin}\in \widetilde{\mathbf{ko}}^n(\MSpin_n),
	\end{align*}
	refining the $\widehat{A}$-genus; 
	Ando-Hopkins-Rezk \cite{AHR10} constructed the \emph{universal string orientation} (also called the \emph{universal string Thom class})
\begin{align*}
	u_{\String}&\colon \mathbf{MString}\to \mathbf{tmf}, 
	\text{ or } u_{\String}\in \widetilde{\mathbf{tmf}}^n(\MString_n),
\end{align*}
refining the Witten genus.
The universal spin and string Thom classes respectively induce the \emph{universal spin Euler class} 
	\[e_{\Spin}=\mathrm{pr}^\ast\mathfrak{z}_{\Spin}^\ast(u_{\Spin})\in H^n(B\Spin_n;\mathbf{ko}).\]
	and the \emph{universal string Euler class} 
	\[e_{\String}=\mathrm{pr}^\ast\mathfrak{z}_{\String}^\ast(u_{\String})\in \mathbf{tmf}^n(B\String_n).\]
	See also \cite[Part~I, Chapter~10]{tmfBook14}.   Let $G=\Spin$ or $\String$. For the $G$-vector bundle $E\to M$ that is classified by a map $g\colon M\to BG_n$, the \emph{$G$-Thom class} and the \emph{$G$-Euler class} of $E$
	\begin{equation}\label{eq:G-Eulercls-def}
		u_G(E)=(Tg)^\ast(u_G)\in \widetilde{h_G}^n(\mathrm{T}(E)),\quad e_G(E)=g^\ast(e_G)\in h_G^n(M)
	\end{equation}
	 are defined similarly to those in (\ref{eq:Thom-Eulerclss}), where $h_\Spin=\mathbf{ko}$ and $h_{\String}=\mathbf{tmf}$.
	Moreover, there is a commutative diagram of morphisms of spectra
	\begin{equation}\label{diag:tmf-ko-hz}
		\begin{tikzcd}
			\mathbf{MString}\ar[r,"u_{\String}"]\ar[d]&\mathbf{tmf}\ar[d,"c_0"]&\\
			\mathbf{MSpin}\ar[r,"u_{\Spin}"]\ar[d]&\mathbf{ko}\ar[d,"P_0"]\\
			\mathbf{MSO}\ar[r,"u_{\SO}"]&\mathbf{H}\Z
			\end{tikzcd},
	\end{equation}
	where the unlabeled vertical morphisms are the forgetful morphisms, $c_0$ is the map evaluating a modular form at the cusp $q=0$ (see  \cite[Section 3]{LN14}), and $P_0$ is the $0$-th Postnikov truncation $P_0(\mathbf{ko})=\mathbf{H}\Z$. Note that the morphism $c_0$ recovers the $\widehat{A}$-genus from the Witten genus \cite{AHR10,LN14}.
	It follows that for a $G(k)$-vector bundle $E$ over a $G(k)$-manifold $M$, there hold
	\begin{equation}\label{eq:Eulercls}
		(c_0)_\ast(e_{\String}(E))=e_{\Spin}(E),\text{~~ or ~~} (P_0)_\ast(e_{\Spin}(E))=e(E).
	\end{equation}

It is well-known that an oriented vector bundle $E\to M$ of rank $n$ admits a nowhere-vanishing section if and only if the ordinary Euler class $e(E)$ vanishes. The following theorem is a generalization of this classical result to the setting of $G$-vector bundles, which is a consequence of the Gysin cohomology exact sequence for the associated sphere bundle and the Moore-Postnikov tower of the map $B\iota_{n-1}\colon B\SO_{n-1}\to B\SO_n$.

\begin{theorem}\label{thm:nonzero-section}
	Let $M$ be a closed smooth $G$-manifold of dimension $n+k$ with $n\geq k+2$, $1\leq k\leq 3$. Then a $G$-vector bundle $E\to M$ of rank $n$ admits a nowhere-vanishing section if and only if the $G$-Euler class $e_G(E)$ vanishes.
\end{theorem}

\begin{proof}
If the vector bundle $\pi\colon E\to M$ admits a nowhere-vanishing section $s$, then the normalization $\bar{s}$ of $s$ is a section of the associated $S^{n-1}$-bundle $\bar{\pi}\colon S(E)\to M$. Then from the Gysin cohomology exact sequence for $S(E)$ (see  \cite[Chapter V, Theorem 1.25]{Rudyakbook}), we deduce that $e_G(E)=0$.
%\[\cdots\to h^i_G(M)\xra{e_G(E)}h_G^{i+n}(M)\xra{\bar{\pi}^\ast}h_G^{i+n}(S(E))\to\cdots,\]
%together with $\bar{s}^\ast\circ \bar{\pi}^\ast=id$, 
%implies that the cup product multiplication 
%\[e_G(E)\!\smallsmile\!\colon h_G^i(M)\to h_G^{i+n}(M)\]
%is trivial for arbitrary $i$. In particular, $e_G(E)=e_G(E)\!\smallsmile\! 1=0$.

For the opposite direction, consider the homotopy fibration 
 \[S^{n-1}\to B\SO_{n-1}\xra{B\iota_{n-1}}B\SO_{n}.\]
Recall that the oriented vector bundle $E\to M$ admits a nowhere-vanishing section if and only if the classifying map $g\colon M\to B\SO_n$ admits a lift $\hat{g}\colon M\to B\SO_{n-1}$. 
Let $\sk_r(M)$ be the $r$-skeleton of $M$. Since the map $B\iota_{n-1}$ is $(n-1)$-connected, such a lift $\hat{g}$ exists on $\sk_{n-1}(M)$.  We shall analyze the following Moore-Postnikov tower (see  \cite[Theorem 4.71]{Hatcher}) of the map $f=B\iota_{n-1}$:
\begin{figure}[H]
		\centering
\begin{tikzcd}
		&Y_4\ar[d,"p_4"]\\
		&Y_{3}\ar[d,"p_3"]\ar[r,"k_4"]&K_{n+3}(\pi_3^S)\\
		&Y_{2}\ar[d,"p_2"]\ar[r,"k_2"]&K_{n+2}(\pi_2^S)\\
		&Y_1\ar[r,"k_1"]\ar[d,"p_1"]&K_{n+1}(\pi_1^S)\\
		B\SO_{n-1}\ar[ur,"f_1",bend left]\ar[uur,"f_2",bend left]\ar[uuur,"f_3",bend left]\ar[uuuur,"f_4",bend left]\ar[r,"f=B\iota_{n-1}"]&B\SO_n=Y_0\ar[r,"k_0=e_{\SO}"]&K_n(\Z)
		\end{tikzcd}
		\caption{Partial Moore-Postnikov tower of $B\iota_{n-1}$.}\label{tower:Moore-Postnikov}
\end{figure}		
\noindent For each $i\geq 1$, the composition 
\[p_{1,i}=p_1\circ p_2\circ \cdots\circ p_i\colon Y_i\to B\SO_n\]   
is $(n+i-1)$-anticonnected, which means that the induced homomorphism $(p_{1,i})_\ast\colon \pi_j(Y_i)\to \pi_j(B\SO_n)$ is an isomorphism for $j\geq n+i$ and is a monomorphism for $j=n+i-1$; the map $f_i\colon B\SO_{n-1}\to Y_i$ is $(n+i-1)$-connected, implying that the induced map 
\begin{equation}\label{eq:lift-equiv}
	(f_i)_\ast\colon [\sk_{n+i-1}(X),B\SO_{n-1}]\to [\sk_{n+i-1}(X),Y_i]
\end{equation}
is surjective for any CW complex $X$. Hence the classifying map $g$ lifts to a map $\hat{g}\colon \sk_{n+i-1}(M)\to B\SO_{n-1}$ if and only if a lift $\tilde{g}_i\colon M\to Y_i$ of $g$ exists,  if and only if the composition $k_j\circ \tilde{g}_j$ is null-homotopic, or equivalently $\tilde{g}^\ast(k_j)=0\in H^{n+j}(M;\pi_j^S)$ for each $j\leq i-1$. 

For $i=1$, the composition $e_{\SO}\circ g$ is the ordinary Euler class $e(E)=e_{\SO}(E)$. Then $e(E)=0$ guarantees that the classifying map $g\colon M\to B\SO_n$ lifts to a map $\tilde{g}_1\colon M\to Y_1$. When $k=0$ and $G=\SO$, the lift $\hat{g}$ on $M$ exists by (\ref{eq:lift-equiv}). 
When $1\leq k\leq 3$ and $G=\Spin$ or $\String$, the $G$-Euler class $e_G(E)$ is zero implying the ordinary Euler class $e(E)$ is zero by (\ref{eq:Eulercls}), hence a lift $\tilde{g}_1\colon M\to Y_1$ of $g$ exists.

When $k=1,2$ and $G=\Spin$, the classifying map $g\colon M\to B\SO_n$ lifts to a map $g\colon M\to B\Spin_n$ and we replace $\SO$ by $\Spin$ in the above Moore-Postnikov tower. Let $F_3$ be the homotopy fibre of the composition $p_{1,3}\colon Y_3\to B\Spin_n$, and consider the homotopy commutative diagram with homotopy fibration rows
\[\begin{tikzcd}
S^{n-1}\ar[r]\ar[d]&B\Spin_{n-1}\ar[r,"f=B\iota_{n-1}"]\ar[d,"f_3"]&B\Spin_n\ar[d,equal]\\
F_3\ar[r]&Y_3\ar[r,"p_{1,3}"]&B\Spin_n
\end{tikzcd}.\]
Since $p_{1,3}$ is $(n+2)$-anticonnected,  $\pi_i(F_3)=0$ for $i\geq n+2$.
Since $f_3$ is $(n+2)$-connected, applying the Five-Lemma shows that the excision map $S^{n-1}\to F_3$ induces an isomorphism of homotopy groups in dimensions $\leq n+1$. It follows that $F_3$ is homotopy equivalent to the $(n+1)$-th Postnikov section $P_2S^{n-1}$ of $S^{n-1}$, and that the excision map $S^{n-1}\to F_3$ is homotopy equivalent to the canonical inclusion map $S^{n-1}\hookrightarrow P_2S^{n-1}$.

Note that the space $Y_3$ is constructed such that the pullback $p_{1,3}^\ast(\gamma^n_{\Spin})$ of the universal bundle $\gamma^n_{\Spin}$ over $B\Spin_n$ (or $B\SO_n$) admits a nowhere-vanishing section over the $(n+2)$-skeleton of $Y_3$. Let $i\colon \sk_{n+2}(Y_3)\to Y_3$ be the inclusion map. Then in $\mathbf{ko}^n(\sk_{n+2}(Y_3))$, we have 
\[0=e_{\Spin}(p_{1,3}^\ast(\gamma^n_{\Spin})|_{\sk_{n+2}(Y_3)})=i^\ast(e_{\Spin}(p_{1,3}^\ast(\gamma^n_{\Spin})))=i^\ast p_{1,3}^\ast(e_{\Spin}). \]
Since $i^\ast\colon \mathbf{ko}^n(Y_3)\to \mathbf{ko}^n(\sk_{n+2}(Y_3))$ is an isomorphism, $p_{1,3}^\ast(e_{\Spin})=0$, or equivalently the composition $e_{\Spin}\circ p_{1,3}$ is null homotopic. Denote by $F_{\mathbf{ko}}$ the homotopy fibre of the universal spin Euler class $e_{\Spin}$, then there is a homotopy commutative diagram with homotopy fibration rows 
\[\begin{tikzcd}
P_2S^{n-1}\ar[r]\ar[d,"\tilde{e}"]&Y_3\ar[d,"\phi_{\mathbf{ko}}"]\ar[r,"p_{1,3}"]&B\Spin_n\ar[d,equal]\\
\Omega \mathbf{ko}_n\ar[r]&F_{\mathbf{ko}}\ar[r]&B\Spin_n\ar[r,"e_{\Spin}"]&\mathbf{ko}_n
\end{tikzcd},\]
where the vertical maps $\tilde{e},\phi_{\mathbf{ko}}$ are excision maps. Note that the restriction of $\tilde{e}$ to the bottom sphere $S^{n-1}$ is the inclusion map $S^{n-1}\hookrightarrow \Omega \mathbf{ko}_n$, and that the induced homomorphism 
\[\tilde{e}_\ast\colon \pi_{n+i}(P_2S^{n-1})\to \pi_{n+i}(\Omega\mathbf{ko}_n)\]
is an isomorphism for $i\leq 2$. Then by the Five-Lemma, the induced homomorphism 
\[(\phi_{\mathbf{ko}})_\ast\colon \pi_{n+i}(Y_3)\to \pi_{n+i}(F_{\mathbf{ko}})\]   
is an isomorphism for $i\leq 2$, and hence the induced map 
\[(\phi_{\mathbf{ko}})_\ast\colon [\sk_{n+2}(M),Y_3]\to [\sk_{n+2}(M),F_{\mathbf{ko}}]\]
is a bijection. It follows that the spin Euler class $e_{\Spin}(E)=g^\ast(e_{\Spin})$ is zero implying that a lift $\tilde{g}\colon \sk_{n+2}(M)\to Y_3$, and hence a lift $\hat{g}$ of $g$ exists on $\sk_{n+2}(M)$. The proof in the case $G=\Spin$ and $k=1,2$ is completed.

When $k=3$ and $G=\String$, similar arguments above show that there is an excision map $\phi_{\mathbf{tmf}}\colon Y_4\to F_{\mathbf{tmf}}$ inducing a bijection 
\[(\phi_{\mathbf{tmf}})_\ast\colon [M,Y_4]\to [M,F_{\mathbf{tmf}}],\]
where $F_{\mathbf{tmf}}$ is the homotopy fibre of the universal string Euler class 
\[e_{\String}\colon B\String_n\to \mathbf{tmf}_n.\] Thus $e_{\String}(E)=0$ implying that the classifying map $g\colon M\to B\String_n$ admits a lift $\tilde{g}\colon M\to Y_4$ of $g$, and therefore a lift $\hat{g}\colon M\to B\SO_{n-1}$ by (\ref{eq:lift-equiv}). The proof of the Proposition is completed.
\end{proof}

Note that the spin Euler class $e_{\Spin}(E)$ encodes the first three obstructions to lifting the classifying map $g\colon M\to B\SO_n$ to $Y_3$, while the string Euler class $e_{\String}(E)$ encodes the first four obstructions.
In general, the string Euler class $e_{\String}(E)$ and $e_{\Spin}(E)$ are not fully understood in geometric topology. In the following we shall combine Theorems \ref{thm:SES-G-H}, \ref{thm:spin-bdsm} and \ref{thm:nonzero-section} to establish conditions equivalent to $e_G(E)=0$.

Let $h^{\bullet}_{G}$ (resp. $h_{\bullet}^{G(k)}$) be the cohomology (resp. homology) theory represented by the spectrum 
\[\mathbf{E}_{G(k)}=\begin{cases}
\mathbf{H}\Z& \text{ for }k=0,\\
\mathbf{ko}&\text{ for }k=1,2,\\
\mathbf{tmf}&\text{ for }k=3.
\end{cases}\]
%$\mathbf{E}_{G(k)}=\mathbf{H}\Z$ for $k=0$, $\mathbf{ko}$ for $k=1,2$, and $\mathbf{tmf}$ for $k=3$.	

\begin{lemma}\label{lem:h_kG}
	Let $G=G(k)$ be given by Table \ref{table:GH-k} and let $M$ be a closed smooth $G$-manifold of dimension $n+l$ with $n\geq l+2$, then the forgetful homomorphism 
\begin{align*}
	h_l\colon\Omega_l^G(M)\xra{}h_l^G(M),\quad [N,f,\sigma]\mapsto f_\ast ([N]_{\mathbf{E}_{G}})
\end{align*}
is an isomorphism in the following cases:
\begin{enumerate}
	\item\label{case:so} $G=\SO$ and $l\leq 2$.
	\item\label{case:spin} $G=\Spin$ and $l\leq 7$.
	\item\label{case:string} $G=\String$ and $l\leq 15$.
\end{enumerate}
\begin{proof}
	The isomorphism in the case $G=\SO$ and $l\leq 2$ refers to Lemma \ref{lem:SO-H2}; the isomorphisms in the second and third cases follow from the Pontryagin-Thom Theorem \ref{thm:Pontryagin-Thom}, the Poincar\'e duality theorem (\cite[Chapter V, Theorem 2.9]{Rudyakbook}), and the following facts: 
	when $n\geq l+2$,  the induced homomorphism 
	\[(u_{\Spin})_\ast\colon \pi_{n+i}(\mathbf{MSpin}_n)\to \pi_{n+i}(\mathbf{ko}_n)\]
	is an isomorphism for $i\leq l\leq 7$, which follows by the Anderson-Brown-Peterson's splitting theorem  \cite[Theorem 2.2]{ABP67}; and the induced homomorphism 
	\[(u_{\String})_\ast\colon \pi_{n+i}(\mathbf{MString}_n)\to \pi_{n+i}(\mathbf{tmf}_n)\]
	is an isomorphism for $i\leq l\leq 15$, see \cite[Theorem 2.1]{Hill09} and \cite{Hopkins02}.
\end{proof}
\end{lemma}

Let $s\colon M\to E$ be a transversal section,  meaning that it is transversal to the zero section $\mathfrak{z}_E\colon M\to E$. Then the zero locus $Z=s^{-1}(0)$ is a closed smooth $k$-dimensional submanifold embedded in $M$. Let $i_Z\colon Z\to M$ be the embedding map and denote $\nu(Z)=\nu_{i_Z}$. Then the differential map  \[ds|_{\nu(Z)}\colon \nu(Z)\to E|_Z=i_Z^\ast(E)\] is a bundle isomorphism. Hence, the normal bundle $\nu(Z)$ admits a $G$-structure induced by the $G$-structure $\sigma|_Z$ on $E|_Z$. We still denote the induced $G$-structure on $\nu(Z)$ by $\sigma|_Z$.
Thus we get a $G$-bordism class 
\[[Z,i_Z,\sigma|_Z]\in\Omega_k^G(M).\]

\begin{lemma}\label{lem:G-divisor-welldef}
	The $G$-bordism class $[Z,i_Z,\sigma|_Z]$ is independent of the choices of the transversal section $s$.
\end{lemma}

\begin{proof}
	Let $s'\colon M\to E$ be another transversal section and denote the corresponding $G$-bordism class of the zero locus by $[Z',i_{Z'},\sigma|_{Z'}]$. Since the space of sections of a vector bundle is contractible, there is a smooth homotopy $H\colon M\times I\to E$ such that 
	\[H(x,0)=s(x),\quad H(x,1)=s'(x).\]
Let $\mathrm{pr}\colon M\times I\to M$ be the projection, then $H$ is a section of the pullback bundle $\mathrm{pr}^\ast(E)$.  Using the transversality extension theorem (see  \cite[Theorem 6.2.12]{Mukh15}), we can homotope $H$ to a transversal section $M\times I\to \mathrm{pr}^\ast(E)$ such that $H|_{M\times \partial I}=s\sqcup s'$.
Then the zero locus $W=H^{-1}(0)\subseteq M\times I$ is a compact smooth $(k+1)$-manifold with boundary $\partial W=Z\sqcup -Z'$. Let $i_W\colon W\hookrightarrow M\times I$ be the inclusion, then the normal bundle $\nu(W)$ inherits a normal $G$-structure $\sigma|_W$ from $E$ by the isomorphism
\[dH|_{\nu(W)}\colon \nu(W)\to \mathrm{pr}^\ast(E)|_W.\]
By construction, $\sigma|_W$ restricts on the boundary to the given normal $G$-structures $\sigma|_Z$ and $\sigma|_{Z'}$. Thus $(W,i_W,\sigma|_W)$ is a normal $G$-bordism between $(Z,i_Z,\sigma|_Z)$ and $(Z',i_{Z'},\sigma|_{Z'})$. 
\end{proof}

%Recall the bordism invariants $\kappa_G$ in (\ref{eq:kappa}) and (\ref{eq:kappa2}):
%\[\kappa_G=\left\{\begin{array}{ll}
%	\kappa_1\colon \Omega_1^{\Spin}\to \pi_1^S & \text{ for }G=\Spin, k=1;\\[1ex]
%	\kappa_2\colon \Omega_2^{\Spin}\to \pi_2^S&  \text{ for }G=\Spin, k=2;\\[1ex]
%	\kappa_3\colon \Omega_3^{\String}\to \pi_3^S& \text{ for }G=\String, k=3.
%\end{array}\right.\]

\begin{definition}\label{def:G-divisor}
	Let $E\to M$ be a $G$-vector bundle as above. The \emph{$G$-divisor} $\kappa_G(E)$ of $E$ is defined by
	\[\kappa_G(E)\coloneqq c_\ast\big([Z,i_Z,\sigma|_Z]\big)=[Z,\sigma|_Z]\in \Omega_k^G.\]
	For $k=2$ and $G=\Spin$, we define the \emph{spin defect class} of $E$ by
	\[\theta_{\Spin}(E)\coloneqq r\big([Z,i_Z,\sigma|_Z]\big)=(i_Z)_\ast\big(\PD_\ast([Z,i_Z,\alpha_Z])\big),\] 
	where $\alpha_Z\in H^1(Z;\z{})$ is the class such that $\sigma|_Z=\sigma_0+\alpha_Z$ with $\sigma_0$ the canonical spin structure on $Z$; see Remark \ref{rm:retraction} and (\ref{eq:retraction}).  For $k=1,3$, set $\delta_G(E)=0$.
  
\end{definition}

By  \cite[Proposition 12.8]{BT89}, the ordinary Euler class $e(E)$ is precisely the Poincar\'e dual of
	$(i_Z)_\ast([Z])$, that is,
	\begin{equation}\label{eq:Eulercls-PD}
		e(E)=\mathrm{PD}\big((i_Z)_\ast([Z])\big),
	\end{equation}
	where $[Z]\in H_k(M;\Z)$ is the ordinary fundamental class of $Z$.
We note that analogous arguments to that of  \cite[Proposition~12.8]{BT89} prove the
	following analogues of \eqref{eq:Eulercls-PD} in generalized cohomology:
	\begin{equation}\label{eq:G-Eulercls-PD}
		\begin{aligned}
			e_{\Spin}(E)
			&=\mathrm{PD}_{\mathbf{ko}}\big((i_Z)_\ast([Z]_{\mathbf{ko}})\big), \\
			e_{\String}(E)
			&=\mathrm{PD}_{\mathbf{tmf}}\big((i_Z)_\ast([Z]_{\mathbf{tmf}})\big),
		\end{aligned}
	\end{equation}
where $[Z]_{\mathbf{E}_{G(k)}}$ denotes the fundamental class (or $\mathbf{E}_{G(k)}$-orientation) of the spin or string $k$-manifold $Z$ in the generalized homology theory represented by the spectrum $\mathbf{E}_{G(k)}$.

\begin{proposition}\label{prop:G-H-Eulercls}
	Let $k=1,2,3$, and let $(G,H)=(G(k),H(k))$ be the pair of groups in Table \ref{table:GH-k}. 
	Let $M$ be a closed $G$-manifold of dimension $n+k$ with $n\geq k+2$ and let $E\to M$ be a $G$-vector bundle of rank $n$. Then $e_G(E)=0$ if and only if 
	\[\kappa_G(E)=0, ~ \delta_G(E)=0, \text{ and } e_{H}(E)=0.\]
\end{proposition}
\begin{proof}
The formulas for $e_G(E)$ in (\ref{eq:G-Eulercls-PD}) imply that 
	\[e_G(E)=0\iff (i_{Z})_\ast([Z]_{\mathbf{E}_{G}})=0\in h_k^G(M),\]   
	which is equivalent to $[Z,i_Z,\sigma|_Z]=0\in \Omega_k^G(M)$ by Lemma \ref{lem:h_kG}. By Theorem \ref{thm:SES-G-H} for $k=1,3$ and Theorem \ref{thm:spin-bdsm} for $k=2$, we have   
	\begin{equation*}
		[Z,i_Z,\sigma|_Z]=0\iff \kappa_G(E)=0,~ \delta_G(E)=0 \text{ and }[Z,i_Z,\sigma_H|_Z]=0\in \Omega_k^H(M),
	\end{equation*} 
	where $\sigma_H$ is the normal $H$-structure on $E$ induced by the normal $G$-structure $\sigma$. On the other hand, combining Lemma \ref{lem:h_kG} and the formula for $e_H(E)$ in (\ref{eq:Eulercls-PD}) or (\ref{eq:G-Eulercls-PD}) yields the equivalences
	\[[Z,i_Z,\sigma_H|_Z]=0\iff (i_Z)_\ast([Z]_{\mathbf{E}_H})=0 \iff e_H(E)=0. \] 
	Thus $e_G(E)=0$ if and only if $\kappa_G(E)=0$, $\delta_G(E)=0$ and $e_H(E)=0$. 
	The proof is completed.
\end{proof}

\begin{proof}[Proof of Theorem \ref{mainthm:appl}]
	Combine Theorem \ref{thm:nonzero-section} and Proposition \ref{prop:G-H-Eulercls}.
\end{proof}

\medskip

\noindent\textbf{Funding.}
%The authors would like to thank the anonymous referee for their careful reading and valuable comments that improved the quality of this paper.
This work was supported by the National Natural Science Foundation of China (grant number 12571075) and the High-level Scientific Research Foundation of Hebei Province.
%Jianzhong Pan were partially supported by the National Natural Science Foundation of China (Grant No. 11971461) and Jie Wu was partially supported by the High-level Scientific Research Foundation of Hebei Province.

%\noindent\textbf{Acknowledgements.}

\bibliography{cohtp}
\bibliographystyle{amsplain}

\end{document}